\documentclass[12pt,twoside]{elsarticle}
\usepackage[latin9]{inputenc}
\usepackage{geometry}
\usepackage{mathtools}
\usepackage{bm}
\usepackage{amsmath}
\usepackage{amsthm}
\usepackage{amssymb}
\usepackage{cancel}
\usepackage{mathdots}
\usepackage{stmaryrd}
\usepackage{stackrel}
\usepackage{undertilde}
\usepackage{graphicx}
\usepackage{booktabs}

\PassOptionsToPackage{version=3}{mhchem}
\usepackage{mhchem}
\usepackage{esint}

\makeatletter

\theoremstyle{plain}
\ifx\thesection\undefined
\newtheorem{thm}{\protect\theoremname}
\else

\fi
\theoremstyle{remark}
\ifx\thesection\undefined
\newtheorem{rem}{\protect\remarkname}
\else

\fi
\theoremstyle{plain}
\ifx\thesection\undefined
\newtheorem{lem}{\protect\lemmaname}
\else

\fi
\theoremstyle{plain}
\ifx\thesection\undefined
\newtheorem{Example}{\protect\lemmaname}
\else

\fi

\allowdisplaybreaks

\usepackage{lineno}
\usepackage{nameref}
\usepackage{xcolor}
\usepackage[colorlinks,
linkcolor=blue,
anchorcolor=yellow,
citecolor=red,
]{hyperref}
\usepackage{bm}
\usepackage{version}
\usepackage{cases}
\usepackage{color}
\usepackage{graphicx}
\usepackage{amscd}

\usepackage{latexsym}
\usepackage{mathrsfs}

\@ifundefined{showcaptionsetup}{}{%
 \PassOptionsToPackage{caption=false}{subfig}}
\usepackage{subfig}
\makeatother

\journal{}

\begin{document}


\newtheorem{The}{Theorem}[section]
\newtheorem{lemma}{Lemma}[section]
\newtheorem{Remark}{Remark}[section]
\newtheorem{Assumption}{Assumption}[section]
\newtheorem{Definition}{Definition}[section]

\renewcommand{\baselinestretch}{1.25}
\def\a{\alpha}
\def\b{\beta}
\def\g{\gamma}
\def\d{\delta}
\def\D{\Delta}
\def\v{\varphi}
\def\l{\lambda}
\def\m{\mathcal}

\def\o{\omega}
\def\O{\Omega}
\def\s{\sigma}
\def\e{\eta}
\def\n{\nabla}
\def\t{\tilde}

\def\no{{\nonumber}}
\def\r{\rightarrow}
\def\p{\partial}
\def\f{\frac}
\def\div{\mbox{div}}

\begin{frontmatter}

\title{A fully discretization, unconditionally energy stable finite element method solving the thermodynamically consistent  diffuse interface model for incompressible two-phase MHD flows with large density ratios\tnoteref{t1}}


\author{Ke Zhang}
\ead{zhangkemath@139.com/zkmath@stu.xju.edu.cn}
\author{Haiyan Su\corref{cor1}}
\ead{shymath@126.com }
\cortext[cor1]{Corresponding author. Tel./fax number: 86 9918582482.}


\begin{abstract}

A diffusion interface two-phase magnetohydrodynamic model has been used  for matched densities  in our previous work \cite{2022Highly, 2023Energy}, which may limit the applications of the model. In this work, we derive a thermodynamically consistent diffuse interface model for diffusion interface two-phase magnetohydrodynamic fluids with large density ratios by Onsager's variational principle and conservation law for the first time. The finite element method for spatial discretization and the first order semi-implicit scheme linked with convect splitting method for temporal discretization, is proposed to solve this new model. The mass conservation, unconditionally energy stability and convergence  of the scheme can be proved.  Then  we derive the existence of  weak solutions of governing system employing the above properties of the scheme and compactness
method. Finally, we show some numerical results to test the effectiveness and well behavior of proposed scheme.

\end{abstract}

\end{frontmatter}

\section{Introduction}
The change in interface topology of immiscible or partially miscible fluids is a fundamental issue in fluid dynamics. Topological transformations such as splintering and reconnection of fluid interfaces are important features of many systems and have a strong effect on flow, see for instance \cite{1998Quasi, David1999Calculation, 2006Fully}. The recently developed diffusion interface model is proved to be an appropriate method to describe this phenomenon. This method allows mixing of fluids in the interface region, see for instance \cite{ABELS2012THERMODYNAMICALLY, 2016Diffuse, 2012Thermodynamically}. When the interface thickness tends to 0, diffused interface model converges to the sharp interface model, see for instance \cite{2021SharpZHANG, 2003ALIU, 1994ConvergenceALI}.  It is worth mentioning that the study of diffuse interface two-phase magnetohydrodynamic (CH-MHD) system in which two incompressible and insoluble conductive fluids interact with electromagnetic fields has been widely used in geophysical fluid dynamics, energy metallurgy, petroleum industry, chemical industry and other fields related to multiphase flow in \cite{1964EFFECTDJ, 2017Second, 2018EntropyEN, 2022Highly, 2023Energy}.


A two-phase diffusion interface system with matched densities can be found in  \cite{1977TheoryHO, 1996TwoGU}. In addition, the existence of weak solutions of the two-phase diffusion interface about singular free energy densities with matched densities is builded in \cite{2009OnABELS}. Immediately after, Abels et al. unveil  two-phase diffusion interface fluids based on sharp interface system, which has the existence of weak solutions, more details refer  to \cite{2009Existence}. When the densities are different, there are also  some works that can be used as a reference. A physically consistent two-phase incompressible system with different density and viscosity that admits an energy law, which is proposed in  \cite{2010ASHEN}. When small density ratios,  the Boussinesq  approximation is introduced to give the unmatched model in above work. The similar process also found in \cite{2015DecoupledSHEN}. And several  numerical tests are carried out to validate the accuracy for problems with large density and viscosity ratios in  \cite{2010ASHEN}.  The applicability  for two-phase incompressible fluid flows with large viscosity and density ratios are contrasted in \cite{2007DiffuseJH}, but it is uncertain whether the model is thermodynamic consistent.  The new two-phase diffuse interface system with unmatched densities based on rational continuum mechanics is introduced in \cite{2012Thermodynamically}, and derived the sharp interface model when the interfacial thickness tends to 0. Abels et al. consider the existence of weak solutions globally of two-phase incompressible flow with different densities in \cite{2019Existence}. The well-posedness of the  Cahn-Hilliard-Navier-Stokes (CH-NS) system can be focused in \cite{1999MathematicalBOYER}, which stands for a model of a multiphase fluid under shear. And the incompressible CH-NS model with unmatched density and give the local existence of a unique very regular solution  in \cite{2001Nonhomogeneous}. And the well-posedness of diffuse interface Cahn-Hilliard-Stokes-Darcy model is derived in \cite{2014ExistenceHAN}. In addition, the two-phase ferrofluid  system by phase field method is proposed in  \cite{2016A}, which  give the  stability, convergence, and existence of solutions. And the  diffuse interface CH-MHD system with different viscosities and electric conductivities is shown in \cite{2019A}, the existence of weak solutions is established.


And the two-phase flow problem is the basis for considering multiple physical field models. For numerical schemes, the fully decoupled, energy stable scheme for solving two-phase flows are proposed in \cite{2010ASHEN,MR3299207, 2021AMI, 2022ALIU, MR4483330, doishen,2022HighlyY, 2021A}. And the algorithms and analysis are focused in  \cite{2019AZHANG, MR4411485,  zhang2021efficient, MR4634319, MR4366278,MR3786791, 2018A, 2010A,2009A,MR3028356,1991On}, to handling the MHD related models. For the CH-MHD problem can be traced to \cite{1964EFFECT}, the first order and unconditionally energy stable scheme is presented in \cite{2019A}. And the magnetic field is considered in $H(curl, \Omega)$ refer to \cite{2022Highly}, which the scheme also fully decoupled and energy stable. The second order and unconditionally energy stable schemes are shown in \cite{MR4243525, 2010Advanced, Goedecke2004Principles, 2020Second}. The considered CH-MHD model mainly with matched densities in above works. In actually life, the density of the two fluids is different, may large density ratios or small density ratios. Often these diffuse interface models (also called phase field models) allow for local entropy or free energy inequalities and in such situations they can be called thermodynamically consistent. The thermodynamically consistent diffuse interface model for incompressible CH-MHD flows with unmatched densities (allow large density ratios) and it's reliability analysis have not been considered, which is the mainly work in this paper. To complete this idea, we mainly take the following steps.


The first objective is to establish a diffusion interface model conforming to the thermodynamic consistency, which is based on Onsager's variational principle. The new system is coupled by incompressible Navier-Stokes system with unmatched densities, Cahn-Hilliard system and Maxwell system by convection, stresses, generalized Ohm's law and Lorentz forces. The second objective is to  design a fully discrete finite element method with  first order semi-implicit scheme in temporal discretization for the model, which has a mass conservation, unconditional energy stability. And the existence of solutions to the numerical method is based on Leray-Schauder fixed point theorem. The third objective is to prove the existence to the weak solution of the model is based on stability of the proposed numerical method and the compactness method, which exist subsequence of the fully discrete solutions converges to the weak solution of the model variation as the space size $h$ and time size $\Delta t$ tend to zero. The forth objective is to verify the theoretical results and validate of the  presented algorithm by the several numerical simulations.

The article consists of the following parts.  In Section 2, we present the preliminaries, the new model derivation and original energy are shown.  Fully discrete energy stable finite element method and it's unconditionally energy stability can be proposed in Section 3. In addition, convergence of the numerical scheme and existence of the weak solution are displayed in Section 4. And several numerical examples are tests in Section 5.
\section{Model derivation and original  energy}
In this section, we focus on the derivation of incompressible diffuse interface  CH-MHD system with large density ratios based on the Onsager's variational principle in  \cite{1931Reciprocal, Onsager1931Reciprocal, 1953Fluctuations}. And the established model satisfies thermodynamic consistency. For simplicity, we give the following notations.
\subsection{Preliminaries}
Firstly, we set convex, bounded region $\Omega\subset R^{d}$, d=2, 3. The $\textbf{L}^{2}$ inner product: $(\textbf{a}, \textbf{b})=\int_{\Omega}\textbf{a}\cdot \textbf{b} \mbox{d}\textbf{x}$ for   $\textbf{a}, \textbf{b}$ be vector functions. Considering $\textbf{L}^{2}$ norm is $\|\textbf{a}\| $=$(\textbf{a}, \textbf{a})^{1/2}$. And  $L^{s}(0, T; W)$, for $1\leq s\leq \infty$, the norm can be defined as $\|\cdot\|_{L^{\infty}(W)}$:=ess$\sup\limits_{0\leq t\leq T}\|\cdot\|_{W}$ and $\|\cdot\|_{L^{s}(W)}$:=$(\int_{0}^{T}\|\cdot\|_{W}^{s} dt )^{\frac{1}{s}}$ for  $W$ is a real Banach space with the norm $\|\cdot\|_{W}$. In addition, we give the following standard Sobolev spaces as follows
\begin{subequations}
\begin{align}
&H^{1}(\Omega)=\{\psi \in L^{2}(\Omega),\, \nabla\psi \in \textbf{L}^{2}(\Omega) \},\\
&\textbf{H}_{0}^{1}(\Omega)=\{\textbf{v}\in  H^{1}(\Omega)^{d},\, \textbf{v}|_{\partial\Omega}=\textbf{0} \},\\
& L_{0}^{2}(\Omega)=\{p\in L^{2}(\Omega),\, \int_{\Omega}p\mbox{d}\textbf{x}=0\},\\
&\textbf{H}_{\tau}^{1}(\Omega)=\{\textbf{m }\in H^{1}(\Omega)^{d},\, \nabla\cdot\textbf{m }=0,\, \textbf{n}\times \textbf{m}=\textbf{0} \},\\
&\textbf{H}_{n}^{1}(\Omega)=\{\textbf{B }\in H^{1}(\Omega)^{d},\,   \textbf{B}\cdot\textbf{n} =\textbf{0} \}.
\end{align}
\end{subequations}

Based on the  Poincare  inequalities and embedding inequalities in \cite{1975Sobolev, Jean2006Mathematical, 1986On, 2019A}, there are estimates as
\begin{subequations}
\begin{align}
&C_{1}\|\nabla \textbf{v}\|\leq\|D(\textbf{v})\|\leq \|\nabla \textbf{v}\|, \quad \forall \textbf{v}\in \textbf{H}_{0}^{1},\\
&\| \textbf{v}\|_{L^{q}}\leq C_{2}\|\nabla \textbf{v}\|, \quad  \forall \textbf{v}\in \textbf{H}_{0}^{1},\, 2\leq q\leq 6,\\
&\| \textbf{u}\|_{L^{3}}\leq C_{3}\| \textbf{u}\|^{\frac{6-d}{6}}\|\nabla \textbf{u}\|^{\frac{d}{6}},\quad \forall \textbf{u}\in \textbf{H}_{0}^{1}, \\
&\| \textbf{u}\|_{L^{4}}\leq C_{4}\| \textbf{u}\|^{\frac{4-d}{4}}\|\nabla \textbf{u}\|^{\frac{d}{4}},\quad \forall \textbf{u}\in \textbf{H}_{0}^{1},\\
&\| \textbf{B}\|_{L^{p}}\leq C_{5} \| \textbf{B}\|_{\textbf{H}^{1}_{n}},\quad \forall \textbf{B}\in \textbf{H}_{n}^{1},\, 2\leq p\leq 6,\\
&\| \textbf{B}\|_{L^{3}}\leq C_{6}\| \textbf{B}\|^{\frac{6-d}{6}}\|\textbf{B}\|^{\frac{d}{6}}_{\textbf{H}^{1}_{n}},\quad \forall \textbf{B}\in \textbf{H}_{n}^{1},\\
&\| \textbf{B}\|_{L^{4}}\leq C_{7}\| \textbf{B}\|^{\frac{4-d}{4}}\|\textbf{B}\|^{\frac{d}{4}}_{\textbf{H}^{1}_{n}},\quad \forall \textbf{B}\in \textbf{H}_{n}^{1},\\
&\|\psi\|_{L^{q}}\leq C_{8}\|\psi\|_{H^{1}},\quad \forall \psi\in H^{1},\, 2\leq q \leq 6,
\end{align}
\end{subequations}
here  $C_{i}$, $i$=1, $\cdots$, 8 are positive constants, and  $\| \textbf{B}\|_{\textbf{H}^{1}_{n}}$ can be defined as $(\|\nabla\times \textbf{B}\|^{2}+\|\nabla\cdot \textbf{B}\|^{2})^{1/2}$.
And there are the summarize common vector operations
\begin{subequations}
\begin{align}
&\nabla\cdot(\textbf{A}\textbf{D})=\textbf{A}\nabla\cdot \textbf{D}+\nabla \textbf{A}\cdot \textbf{D}, \label{2.1a}\\
&\textbf{A}\otimes \textbf{D} =\textbf{A} (\textbf{D})^{T}.
\end{align}
\end{subequations}

And  we need to introduce some basic equations
\begin{equation}\label{e-1}
\begin{aligned}
&  (A-D, A)=\frac{1}{2}(\|A\|^{2}- \|D\|^{2}+\|A-D\|^{2}   ), \\
& (A^{3}-D)(A-D)=\frac{1}{4}[ (A^{2}-1)^{2}-(D^{2}-1)^{2}   ]+\frac{1}{4}(A^{2}-D^{2})^{2}+\frac{1}{2}A^{2}(A-D)^{2}+\frac{1}{2}(A-D)^{2},\\
& \theta(A^{3}-B)(A-\theta B)=\theta  (A^{3}-\theta B)(A-\theta B)-(1-\theta)(A-\theta B)\theta B.
\end{aligned}
\end{equation}

Inspired by Section 4 of \cite{1975Sobolev}, we show the following interpolation inequality. If $1\leq p\leq r\leq q\leq \infty$, which satisfy $\frac{1}{r}=\frac{\varsigma}{p}+\frac{1-\varsigma}{q}$, such that $\textbf{L}^{p}\cap \textbf{L}^{q}\hookrightarrow \textbf{L}^{r}$ and $\forall \textbf{u}\in \textbf{L}^{p}\cap \textbf{L}^{q}$ satisfy the following inequality
\begin{equation}\label{inter}
\|\textbf{u}\|_{\textbf{L}^{r}}\leq \|\textbf{u}\|_{\textbf{L}^{p}}^{\varsigma}\|\textbf{u}\|_{\textbf{L}^{q}}^{1-\varsigma}, \quad  \varsigma=[0, 1].
\end{equation}

\begin{lemma}\label{lax}
(\textbf{Lax-Milgram Theorem}) See for instance \cite{2008}, given a Hilbert space (\textbf{V}, ($\cdot,\cdot$)), a continuous, coercive bilinear form a($\cdot,\cdot$) and a continuous linear functional $\textbf{F}\in \textbf{V}'$, there exists a unique $\textbf{u}\in \textbf{V}$ such that
\begin{equation*}
a(\textbf{u},\textbf{v})=\textbf{F}(\textbf{v}), \quad \forall \textbf{v}\in \textbf{V}.
\end{equation*}
\end{lemma}

\subsection{Derivation of the model}
From the point of mixed energy view, the governing system is coupled by Cahn-Hilliard system for the phase field and chemical potential, single-phase MHD system for velocity field, pressure field, and magnetic field. Unlike previous diffusion interface CH-MHD with matched densities \cite{2022Highly, 2023Energy}, we mainly consider large density ratios problem, so it is uncertain whether the new model will produce coupled terms. The Onsager's variational principle and conservation laws will give results as follows. The total  energy consists of three parts
 \begin{equation}\label{17}
\begin{aligned}
E&=E_{\phi}+E_{u}+E_{B}\\
&=\gamma\int_{\Omega}\left(\frac{\varepsilon}{2}|\nabla\phi|^{2}+\frac{1}{\varepsilon}F(\phi)\right)\mbox{d}\textbf{x}+\frac{1}{2}\int_{\Omega}\rho (\phi) |\textbf{u}|^{2}\mbox{d}\textbf{x}+\frac{1}{2}\int_{\Omega}\frac{|\textbf{B}|^{2}}{\mu}\mbox{d}\textbf{x}.
\end{aligned}
\end{equation}
$\bullet$ The mixed energy $E_{\phi}$ is defined by Helmtoltz free energy as
\begin{equation*} E_{\phi}=\gamma\int_{\Omega}\left(\frac{\varepsilon}{2}|\nabla\phi|^{2}+\frac{1}{\varepsilon}F(\phi)\right)\mbox{d}\textbf{x},
\end{equation*}
where  $\gamma$ and $\varepsilon$ stand for surface tension, and interfacial thickness respectively.  According to classical self consistent mean field theory in statistical physics in \cite{2001Principles}, mixed energy $E_{\phi}$ is defined by  Helmholtz free energy functional. A  mixture of two immiscible, incompressible fluids, and assume that the immiscible fluid is partially mixed at the interface junction, see for instance \cite{2012Thermodynamically, 2009On}. The phase field  is recommended to label two fluids
\begin{equation}\label{2-1}
\phi (x, t)=\left\{
\begin{aligned}
-1, \qquad \rm fluid\, I,\\
1, \qquad \rm fluid\, II.
\end{aligned}
\right.
\end{equation}
The phase field $\phi$ is almost constants in bulk regions and smoothly transitions between these values in an interfacial region of thickness $\varepsilon$. Refer to \cite{Novick2008}, $F(\phi)$=$\frac{1}{4}(\phi^{2}-1)^{2}$  is the Ginzburg-Laudau double-well potential.  Both   $\nabla\phi$ ( gradient energy) and $F(\phi)$ (bulk energy)  contribute  to hydrophilic interactions, hydrophobic interactions respectively. The equilibrium configuration is the consequence of the competition between the two types of interactions.

Cahn-Hilliard equation is an important mathematical model to describe interface evolution in phase separation. By considering the competition between gradient energy and  bulk energy, the equation reveals the physical nature of the interface evolution during phase separation. Depending on our construction objectives, phase field will be described  as following classical system
\begin{equation}\label{2-fai}
 \phi_{t} +\nabla\cdot (\phi \textbf{u})+\nabla\cdot \textbf{J}_{\phi}=0,
\end{equation}
for  $\textbf{u}$ is velocity field of incompressible fluids, $\textbf{J}_{\phi}$ represents the mass flux and will give the specific form later. With the help of  $\nabla\cdot \textbf{u}$=0 and equation (\ref{2.1a}), we can rewrite the nonconservative advection operator $(\textbf{u}\cdot\nabla)\phi$ as the conservative $\nabla\cdot (\phi \textbf{u})$ equivalently.

The chemical potential $\omega$ is defined by the variational derivative of the energy $E_{\phi}$ with respect $\phi$
\begin{equation}
\omega=\frac{\delta E_{\phi}}{\delta \phi}=-\gamma\varepsilon\Delta \phi+ \frac{\gamma}{\varepsilon}f(\phi),
\end{equation}
where  $f(\phi)=F(\phi)'$.

\begin{Remark}
There are two popular choices for $F(\phi)$, the Ginzburg-Laudau double-well potential in \cite{Novick2008} and Flory-Huggins logarithmic potential \cite{2019The} as
\begin{equation*}
F(\phi)=\frac{1+\phi}{2} \ln (\frac{1+\phi}{2})+\frac{1-\phi}{2} \ln (\frac{1-\phi}{2})+\frac{\vartheta}{4}(\phi^{2}-1)^{2},
\end{equation*}
here $\vartheta> 2$ is the energy parameter.
\end{Remark}

$\bullet$ The kinetic energy $E_{u}$ is defined as
\begin{equation*}
E_{u}=\frac{1}{2}\int_{\Omega}\rho (\phi) |\textbf{u}|^{2}\mbox{d}\textbf{x},
\end{equation*}
where  $\textbf{u}$ is velocity field, and $\rho (\phi)=\frac{\rho_{2}-\rho_{1}}{2}\phi+\frac{\rho_{1}+\rho_{2}}{2}$ is the density of fluid material.  The relation between $\rho$ and $\phi$ is given in \cite{ABELS2012THERMODYNAMICALLY, 2011Existence}. Navier-Stokes equation reflects the basic mechanical law of viscous fluid (also known as real fluid) flow, which describing the conservation of momentum in viscous incompressible fluids and has very important significance in fluid mechanics. Thus, we model the incompressible flow, see for instance \cite{ABELS2012THERMODYNAMICALLY}
\begin{equation}\label{2-u}
\begin{aligned}
&\rho(\phi) \textbf{u}_{t} +((\rho(\phi) \textbf{u}+\frac{\partial\rho(\phi)}{\partial\phi} \textbf{J}_{\phi})\cdot\nabla)\textbf{u} -\nabla\cdot \textbf{S}+\nabla p =\textbf{F},\\
&\nabla\cdot\textbf{ u} =0,
\end{aligned}
\end{equation}
where  $\textbf{S}$ is the symmetric stress tensor, $p$ is the pressure field, $\textbf{F}$ is coupled terms. We will give the specific forms of $\textbf{F}$ and $\textbf{S}$ latter.

$\bullet$ The electromagnetic field has a contribution to $E_{B}$, which the form as
\begin{equation*}
E_{B}=\frac{1}{2}\int_{\Omega}\frac{|\textbf{B}|^{2}}{\mu}\mbox{d}\textbf{x},
\end{equation*}
for $\mu$ is magnetic permeability, $\textbf{B}$ is the magnetic displacement vector \cite{2002Electromagnetic}. There is no contribution to the total energy from the electric field, since the displacement current is neglected in MHD model in \cite{Ming2007A}. Remarkable, the  magnetic permeability $\mu$ of the two fluids are relatively weak, and there is no significant difference between the two fluids. Namely, we assume  $\mu_{1}$=$\mu_{2}$=$\mu$. As a classical magnetic equation, we can well describe the evolution of magnetic field $\textbf{B}$ as
\begin{equation}\label{2-B}
\begin{aligned}
& \textbf{B}_{t} + \frac{1}{\mu}\nabla\times (\frac{1}{\sigma (\phi)} \nabla\times \textbf{B})-\nabla\times (\textbf{u}\times \textbf{B})=\textbf{0},\\
&\nabla\cdot\textbf{ B} =0,
\end{aligned}
\end{equation}
for $\sigma (\phi)$ is electric conductivity, more  specifically $
\sigma (\phi)=\frac{\sigma_{2}-\sigma_{1}}{2}\phi+\frac{\sigma_{1}+\sigma_{2}}{2}.$  The following boundary conditions are assumed as
\begin{equation}
\textbf{u}\cdot \textbf{n}=0, \quad \textbf{J}_{\phi}\cdot\textbf{ n}=0, \quad \textbf{B}\times \textbf{n}=\textbf{0},
\end{equation}
for $\textbf{n}$ is unit outer normal vector.

Hence, by taking the time derivative of the total energetic functional from equation (\ref{17}), we have
\begin{equation}\label{2-a}
\begin{aligned}
\frac{dE}{dt}&=\int_{\Omega}\omega\cdot \phi_{t}\mbox{d}\textbf{x}+\int_{\Omega} (\rho(\phi)\textbf{u}\cdot  \textbf{u}_{t}+ \frac{|\textbf{u}|^{2}}{2}\rho(\phi)'\phi_{t})\mbox{d}\textbf{x}+\int_{\Omega}\frac{\textbf{B}}{\mu}\cdot \textbf{B}_{t} \mbox{d}\textbf{x}\\
&= \int_{\Omega}J_{\phi}\cdot\nabla\omega\mbox{d}\textbf{x}- \int_{\Omega}\textbf{S} \cdot D(\textbf{u})\mbox{d}\textbf{x}-\int_{\Omega}\frac{1}{\mu^{2}\sigma (\phi)}|\nabla\times \textbf{B}|^{2}\mbox{d}\textbf{x}\\
&\quad +\int_{\Omega}\phi \textbf{u}\cdot \nabla\omega\mbox{d}\textbf{x} +  \int_{\Omega} \textbf{u}\cdot \textbf{F}\mbox{d}\textbf{x}-\int_{\Omega}(\frac{1}{\mu}\nabla\times \textbf{B}\times \textbf{B})\cdot \textbf{u}\mbox{d}\textbf{x} ,
\end{aligned}
\end{equation}
where  $D(u)$=$\frac{1}{2}(\nabla \textbf{u}+(\nabla \textbf{u})^{T})$ is strain velocity tensor.  Here, we used   equations (\ref{2-u}),   (\ref{2-B}), and   (\ref{2-fai}), the boundary of velocity field and the following equation, refer to \cite{ABELS2012THERMODYNAMICALLY}.
\begin{equation*}
\int_{\Omega}\textbf{u}\cdot ( (\rho(\phi)\textbf{u}+\frac{\partial\rho(\phi)}{\partial\phi} \textbf{J}_{\phi})\cdot\nabla)\textbf{u}\mbox{d}\textbf{x}=-\frac{1}{2}\int_{\Omega}\frac{\partial\rho(\phi)}{\partial\phi}|\textbf{u}|^{2}(\nabla\cdot\textbf{J}_{\phi}+\nabla\phi\cdot\textbf{u})\mbox{d}\textbf{x}.
\end{equation*}

Collecting all terms having a scalar product with the velocity field in above equation, gives the rate of change of the mechanical work with. We define it as $\frac{dW}{dt}$, namely
\begin{equation*}\label{2-W}
\begin{aligned}
\frac{dW}{dt}=-\int_{\Omega}(\frac{1}{\mu}\nabla\times \textbf{B}\times \textbf{B})\cdot \textbf{u}\mbox{d}\textbf{x} +\int_{\Omega}\phi \textbf{u}\cdot \nabla\omega\mbox{d}\textbf{x} + \int_{\Omega} \textbf{u}\cdot \textbf{F}\mbox{d}\textbf{x}.
\end{aligned}
\end{equation*}
There is no external force in the system, the rate of change of mechanical work is 0, namely
\begin{equation}\label{W}
\frac{dW}{dt}=0.
\end{equation}
Hence, we find  the forces $\textbf{F}$ is
\begin{equation*}\label{2-F}
\begin{aligned}
\textbf{F}= \frac{1}{\mu}\nabla\times \textbf{B}\times \textbf{B}- \phi  \cdot \nabla\omega,
\end{aligned}
\end{equation*}
where the first term is the Lorentz force, the second term is surface tension induced by  microscopic internal energy. To determine the fluxes $\textbf{J}_{\phi}$ and $\textbf{S}$, we introduce the dissipation functional
\begin{equation*}\label{22-fai}
\begin{aligned}
\Phi (\textbf{J}, \textbf{J})=\int_{\Omega}\frac{|\textbf{J}_{\phi}|^{2}}{2M(\phi)}\mbox{d}\textbf{x}+\int_{\Omega}\frac{|\nabla\times \textbf{B}|^{2}}{\mu^{2}\sigma (\phi)}\mbox{d}\textbf{x}+\int_{\Omega}\frac{|\textbf{S}|^{2}}{4\eta (\phi)}\mbox{d}\textbf{x},
\end{aligned}
\end{equation*}
for  $\textbf{J}$=$(\textbf{J}_{\phi}, \nabla\times \textbf{B}, \textbf{S})$. Considering $M(\phi)$ is the diffusion mobility related to the relaxation time scale, $\eta (\phi)$ is the viscosity of the fluids. In fact, the postulated form of the dissipation function $\phi$ assumes that the fluxes depend linearly on the thermodynamic forces implicitly. We use Onsager's variational principle, which postulates
\begin{equation}
\delta_{J}\left(\frac{dE}{dt}+  \Phi (\textbf{J}, \textbf{J}) \right)=0.
\end{equation}
According to above equation, we can obtain
\begin{equation*}
\textbf{J}_{\phi}=-M(\phi)\nabla\omega, \quad \textbf{S}=2\eta (\phi)D(\textbf{u}),
\end{equation*}
where  $M(\phi)=\frac{M_{2}-M_{1}}{2}\phi+\frac{M_{1}+M_{2}}{2}$, $\eta (\phi)= \frac{\eta_{2}-\eta_{1}}{2}\phi+\frac{\eta_{1}+\eta_{2}}{2}$.

In summary, we can conclude that  incompressible diffuse interface CH-MHD with large density ratios model as
\begin{equation}\label{2-m}
\left\{
\begin{aligned}
&\rho (\phi)\textbf{u}_{t} +((\rho (\phi)\textbf{u}-\frac{\partial\rho(\phi)}{\partial\phi} M(\phi)\nabla\omega)\cdot\nabla)\textbf{u} -2\eta (\phi)\nabla\cdot D(\textbf{u}) \\
&\quad\,\, +\nabla p -\frac{1}{\mu}\nabla\times \textbf{B}\times \textbf{B}+\phi  \cdot \nabla\omega=\textbf{0},\\
&\nabla\cdot\textbf{ u} =0,\\
& \textbf{B}_{t} + \frac{1}{\mu}\nabla\times (\frac{1}{\sigma (\phi)} \nabla\times \textbf{B})-\nabla\times (\textbf{u}\times \textbf{B})=\textbf{0},\\
&\nabla\cdot\textbf{ B} =0,\\
& \phi_{t} +\nabla\cdot (\phi \textbf{u})-\nabla\cdot (M(\phi)\nabla\omega)=0,\\
&\omega=-\gamma\varepsilon\Delta \phi+ \frac{\gamma}{\varepsilon}f(\phi),
\end{aligned}
\right.
\end{equation}
and adapt to the following initial boundary conditions
\begin{equation}\label{2-boundary}
\left\{
\begin{aligned}
&\textbf{u}(0)=\textbf{u}^{0}, \quad \textbf{B}(0)=\textbf{B}^{0},\quad \phi(0)=\phi^{0}, \,\, \mbox{on} \,\, \Omega_{0},\\
&\textbf{u}=\textbf{0}, \quad \textbf{B} \times \textbf{n}=\textbf{0}, \,\, \mbox{on}\,\, \partial\Omega,\\
&\partial_{n}\phi=0,\quad M(\phi)\partial_{n} w=0,\,\, \mbox{on}\,\, \partial\Omega.
\end{aligned}
\right.
\end{equation}

$(\rho_{j}, \sigma _{j}, M_{j}, \eta_{j})>0$ be the positive specific constants, for $j$=1, 2. And for simplicity's sake, we define $\rho=\rho(\phi)$ here and after.

In such framework, the overall dynamics of an isothermal hydrodynamical system are governed by the Fist Law of Thermodynamics and the Second Law of Thermodynamics that lead to the following Theorem. The similar process refer to \cite{Li2017TWO, 2021SharpZHANG}.

\begin{The} The considered diffuse interface CH-MHD with large density  ratios system  (\ref{2-m})-(\ref{2-boundary}), which is thermodynamically consistent in the sense that
\begin{equation}\label{model-energy}
\frac{dE}{dt} =-(\|\sqrt{M(\phi)}\nabla\omega\|^{2}+2\|\sqrt{\eta (\phi)}D(\textbf{u})\|^{2}+\frac{1}{\mu^{2}} \|\frac{1}{\sqrt{\sigma (\phi)}}\nabla\times \textbf{B}\|^{2})\leq 0,
\end{equation}
where  $ E=\gamma\int_{\Omega}\left(\frac{\varepsilon}{2}|\nabla\phi|^{2}+\frac{1}{\varepsilon}F(\phi)\right)\mbox{d}\textbf{x}+\frac{1}{2}\int_{\Omega}\rho |\textbf{u}|^{2}\mbox{d}\textbf{x}+\frac{1}{2}\int_{\Omega}\frac{|\textbf{B}|^{2}}{\mu}\mbox{d}\textbf{x}.
$

\end{The}

\begin{proof}
The theorem is given by immediate consequence of (\ref{2-a}) and (\ref{W}).

\end{proof}

\section{Fully discrete numerical scheme }
In this section, we propose a fully discrete finite element method (FEM) to solving the considered system. The unconditional energy stability, existence of the solution of scheme are proved. The meshes $\zeta_{h}$ is a regular and quasi-uniform partition of $\Omega$, which the mesh $h$. To approximate  unknowns fields, we employ the finite element spaces such that $\phi$ $\in Y_{h}$, $ \textbf{u}\in \textbf{X}_{h}$, $p \in M_{h}$,  $\textbf{B}\in \textbf{W}_{h} $, where the $Y_{h}\subset H^{1} (\Omega)$, $\textbf{X}_{h}\subset \textbf{H}^{1}_{0}(\Omega)$, $M_{h}\subset L^{2}_{0}(\Omega)$, $\textbf{W}_{h}\subset H^{1}_{n}(\Omega)$.

\begin{Assumption}\label{a1}
The finite element space $(\textbf{X}_{h}, M_{h})$   satisfy the following  inf-sup conditions, 
\begin{subequations}
\begin{align}
&\inf_{q_{h}\in M_{h}\backslash\{0\} }\sup_{\textbf{v}_{h}\in \textbf{X}_{h}\backslash\{\textbf{0}\}}\frac{(\nabla\cdot \textbf{v}_{h}, q_{h})}{\|q_{h}\|\|\nabla \textbf{v}_{h}\|}\geq C_{9}. \label{41}
\end{align}
\end{subequations}
\end{Assumption}
The more popular mixed finite element pairing  of $(\textbf{X}_{h}, M_{h})$, such as MINI element, Taylor-Hood element, which all can fulfill the inf-sup condition (\ref{41}), see for instance \cite{1986FiniteG, 1991BF}.
In addition,  $\textbf{W}_{h}=\{\textbf{C}_{h}\in C^{0}(\bar{\Omega})\cap \textbf{H}^{1}_{n}:  \textbf{C}_{h}|_{k}\in \textbf{P}_{k}(K), \forall K\in \zeta_{h}\}$, for $k\geq 1$, $\textbf{P}_{k}(K)$ is the space of polynomials of total degree.
\subsection{Description of the scheme}
Focusing the  fifth equation of (\ref{2-m})
\begin{equation*}
\phi_{t} +\nabla\cdot (\phi \textbf{u})-\nabla\cdot (M(\phi)\nabla\omega)=0.
\end{equation*}
Multiplying above equation with $-\frac{\rho_{2}-\rho_{1}}{2}$, combining the definition with $\rho (\phi)$, we obtain the mass balance equation
\begin{equation}
 \rho_{t} +\nabla\cdot (\rho\textbf{u}-\frac{\partial\rho}{\partial\phi} M(\phi) \nabla\omega)=0.
\end{equation}
And we can rewrite the first equation of (\ref{2-m}) (momentum equation) as
\begin{equation*}
(\rho\textbf{u})_{t} +\nabla\cdot (\rho\textbf{u} \otimes\textbf{u})+\nabla\cdot(\textbf{u}\otimes (-\frac{\partial\rho}{\partial\phi} M(\phi)\nabla\omega))-2\eta (\phi)\nabla\cdot  D(\textbf{u}) +\nabla p -\frac{1}{\mu}\nabla\times \textbf{B}\times \textbf{B}+\phi  \cdot \nabla\omega=\textbf{0}.
\end{equation*}
Finally, we obtain the following unsteady incompressible diffusion interface CH-MHD  with large density ratios system (\ref{TWO-PHASE MHD}) and the initial boundary conditions (\ref{2-boundary}).
\begin{equation}\label{TWO-PHASE MHD}
\left\{
\begin{aligned}
&(\rho\textbf{u})_{t} +\nabla\cdot (\rho\textbf{u} \otimes\textbf{u})+\nabla\cdot(\textbf{u}\otimes (-\frac{\partial\rho}{\partial\phi} M(\phi)\nabla\omega))-2\eta (\phi) \nabla\cdot D(\textbf{u})\\
&\qquad\,\, +\nabla p -\frac{1}{\mu}\nabla\times \textbf{B}\times \textbf{B}+\phi  \cdot \nabla\omega=\textbf{0},\\
&\nabla\cdot\textbf{ u} =0,\\
& \textbf{B}_{t} + \frac{1}{\mu}\nabla\times (\frac{1}{\sigma (\phi)} \nabla\times \textbf{B})-\nabla\times (\textbf{u}\times \textbf{B})=\textbf{0},\\
&\nabla\cdot\textbf{ B} =0,\\
&\phi_{t} +\nabla\cdot (\phi \textbf{u})-\nabla\cdot (M(\phi)\nabla\omega)=0,\\
&\omega=-\gamma\varepsilon\Delta \phi+ \frac{\gamma}{\varepsilon}f(\phi).
\end{aligned}
\right.
\end{equation}

For arbitrary but fixed $T>0$ and the time step size $\Delta t$=$\frac{T}{N}$,  $d_{t} \textbf{u}_{h}^{k+1}$=$\frac{\textbf{u}_{h}^{k+1}-\textbf{u}_{h}^{k}}{\Delta t}$. Considering model (\ref{TWO-PHASE MHD})-(\ref{2-boundary}), we give the following fully discrete  finite element scheme as
\begin{subequations}\label{fully discrete scheme}
\begin{align}
&(d_{t}(\rho_{h}^{k+1} \textbf{u}_{h}^{k+1}),\textbf{v}_{h}) -( \rho_{h}^{k}  \textbf{u}_{h}^{k+1} \otimes\textbf{u}_{h}^{k+1}, \nabla\textbf{v}_{h})+( \textbf{u}_{h}^{k+1} \otimes (\frac{\rho_{2}-\rho_{1}}{2} M(\phi_{h}^{k})\nabla\omega_{h}^{k+1}), \nabla\textbf{v}_{h})   \nonumber \\
&  +2 (\eta (\phi_{h}^{k})D(\textbf{u}_{h}^{k+1}), D( \textbf{v}_{h})   )-(\nabla\cdot \textbf{v}_{h},  p_{h}^{k+1}) +\frac{1}{\mu}(\nabla\times \textbf{B}_{h}^{k+1}, \textbf{v}_{h}\times \textbf{B}_{h}^{k})+(\phi_{h}^{k}  \cdot \nabla\omega_{h}^{k+1}, \textbf{v}_{h})=0,\label{equau}\\
&(\nabla\cdot\textbf{ u}_{h}^{k+1}, q_{h}) =0,\\
& (d_{t}\textbf{B}_{h}^{k+1}, \textbf{C}_{h}) + \frac{1}{\mu}(  \frac{1}{\sigma (\phi_{h}^{k})} \nabla\times \textbf{B}_{h}^{k+1}, \nabla\times \textbf{C}_{h})+ \frac{1}{\mu}  (\frac{1}{ \sigma (\phi_{h}^{k})}\nabla\cdot\textbf{ B}_{h}^{k+1}, \nabla\cdot \textbf{C}_{h})\nonumber  \\
& \quad  -(\textbf{u}_{h}^{k+1}\times \textbf{B}_{h}^{k}, \nabla\times \textbf{C}_{h})=0,\label{B-1}\\
&(d_{t}\phi_{h}^{k+1}, \psi_{h}) -(\phi_{h}^{k} \textbf{u}_{h}^{k+1}, \nabla \psi_{h})+ (M(\phi_{h}^{k})\nabla\omega_{h}^{k+1}, \nabla\psi_{h})=0, \label{eqn123}\\
&(\omega_{h}^{k+1}, \chi_{h})=\gamma\varepsilon(\nabla\phi_{h}^{k+1}, \nabla\chi_{h})+ \frac{\gamma}{\varepsilon}(f(\phi_{h}^{k+1}), \chi_{h}), \label{22}\\
&\textbf{u}_{h}^{0}= P_{uh}\textbf{u}_{0}, \, \,  \textbf{B}_{h}^{0}=P_{Bh}\textbf{B}_{0},  \, \,\phi_{h}^{0}=P_{\phi h}\phi_{0},
\end{align}
\end{subequations}
where  $(\textbf{v}_{h}, q_{h},  \textbf{C}_{h}, \psi_{h},  \chi_{h})\in  \textbf{X}_{h}\times M_{h}\times \textbf{W}_{h}\times Y_{h}\times Y_{h}$ and $f(\phi_{h}^{k+1})$=$(\phi_{h}^{k+1})^{3}-\phi_{h}^{k}$ is derived from a convex splitting approximation to the non-convex function $F(\phi)$, see for instance \cite{1998UnconditionallyE, 2014ExistenceHAN, 2010UnconditionallyWISE}. And  $P_{0h}$ is $L^{2}$ orthogonal projection operator from $L^{2}$ into $\textbf{V}_{h}$, which   $\textbf{V}_{h}=\{\textbf{u}_{h}\in \textbf{X}_{h}: (\nabla\cdot \textbf{u}_{h}, q_{h})=0, \forall q_{h}\in M_{h} \}$. $P_{uh}(P_{Bh})$ is $L^{2}$ orthogonal projection operator from $L^{2}$ into $\textbf{X}_{h}(\textbf{M}_{h})$. And  $P_{\phi h}$ is $L^{2}$ orthogonal projection operator from $L^{2}$ into $Y_{h}$. Inspired by \cite{1991Some, 2002On}, above operators all have $H^{1}$ stability, namely,
\begin{subequations}
\begin{align}
& \|\nabla P_{uh}\textbf{u}_{h}^{k+1}\|\leq C_{u} \|\nabla \textbf{u}_{h}^{k+1}\|, \label{H1}\\
&\|\nabla P_{Bh}\textbf{B}_{h}^{k+1}\|\leq C_{B} \|\nabla \textbf{B}_{h}^{k+1}\|,\label{H2}\\
&\|\nabla P_{\phi h}\phi^{k+1}_{h}\|\leq C_{\phi} \|\nabla \phi^{k+1}_{h}\|,\label{H3}
\end{align}
\end{subequations}
for $C_{u}, C_{B}, C_{\phi}$ be generic positive constants. Assume the following limiting conditions
\begin{equation}\label{equation4}
\lim_{h\rightarrow 0}\|\phi_{h}^{0}-\phi_{0}\|_{H^{1}}= \lim_{h\rightarrow 0}\|\textbf{u}_{h}^{0}-\textbf{u}_{0}\|_{L^{2}}=\lim_{h\rightarrow 0}\|\textbf{B}_{h}^{0}-\textbf{B}_{0}\|_{L^{2}}=0.
\end{equation}

The symbol $\langle\cdot, \cdot\rangle$ denotes for dual product between the space and its space, such as $(H^{1}(\Omega))'$ and $H^{1}(\Omega)$, $(\textbf{H}_{0}^{1}(\Omega))'$ and $\textbf{H}_{0}^{1}(\Omega)$, $(\textbf{H}_{n}^{1}(\Omega))'$ and $\textbf{H}_{0}^{1}(\Omega)$, see for instance \cite{1991}. And we define $\|\textbf{v}\|_{\textbf{H}_{0}^{1}}=\|\nabla\textbf{v}\|$.
\begin{equation*}
\begin{aligned}
&\|\phi_{h}\|_{(H^{1})'}=\sup_{\psi\in H^{1}(\Omega)}\frac{\langle\phi_{h}, \psi\rangle}{\|\psi\|_{H^{1}}},\\
&\|\textbf{u}_{h}\|_{(\textbf{H}_{0}^{1})'}=\sup_{\textbf{v}\in\textbf{H}_{0}^{1}(\Omega)\setminus\{0\}}\frac{\langle\textbf{u}_{h}, \textbf{v}\rangle}{\|\nabla\textbf{v}\|},\\
&\|\textbf{B}_{h}\|_{(\textbf{H}_{n}^{1})'}=\sup_{\textbf{C}\in H_{n}^{1}(\Omega) }\frac{\langle\textbf{B}_{h}, \textbf{C}\rangle}{\| \textbf{C}\|_{\textbf{H}_{n}^{1}}}.\\
\end{aligned}
\end{equation*}

\begin{Definition}
Based on the $ \phi_{0}\in H^{1}$, $\textbf{u}_{0}\in \textbf{H}_{0}^{1}$, $p\in L_{0}^{2}$, $\textbf{B}_{0}\in \textbf{H}_{n}^{1}$, the $( \phi, \omega, \textbf{u}, p, \textbf{B})$ be  weak solution of model (\ref{TWO-PHASE MHD}), if it satisfies the following  regularity assumptions
\begin{subequations}
\begin{align}
&\phi  \in L^{\infty}(0, T; H^{1}), \,\, \phi_{t},  \in L^{2}(0, T; (H^{1})'),\\
& \omega\in L^{\infty}(0, T;  H^{1} ),\\
&\textbf{u} \in L^{\infty}(0, T; H^{1}_{0}),\,\, \textbf{u}_{t} \in L^{2}(0, T; (H^{1}_{0})'),\\
&p\in L^{2}(0, T; L^{2}_{0}),\\
&\textbf{B} \in L^{\infty}(0, T; H^{1}_{n}),\,\, \textbf{B}_{t} \in L^{\frac{4}{d}}(0, T; (H^{1}_{n})').
\end{align}
\end{subequations}

\end{Definition}

\subsection{Unconditionally energy stable of the scheme}
In this subsection, we mainly consider the unconditional energy stable of the above scheme. And the following equations can be employed, see for instance \cite{2011Existence, 2021ExistenceKA},
\begin{equation}\label{e-2}
\begin{aligned}
& \int_{\Omega}(\nabla\cdot (u^{k+1}\otimes J^{k+1}))\psi_{1}\mbox{d}\textbf{x}\\
&\quad =\int_{\Omega} (\nabla\cdot J^{k+1}-\frac{\rho^{k+1}-\rho^{k}}{\Delta t}-u^{k+1}\cdot\nabla\rho^{k} )\frac{u^{k+1}}{2}\cdot\psi_{1}\mbox{d}\textbf{x}+\int_{\Omega}(J^{k+1}\cdot\nabla)u^{k+1}\cdot\psi_{1}\mbox{d}\textbf{x},\\
& \int_{\Omega}( (\nabla\cdot J)\frac{u}{2} +(J\cdot\nabla)u   )\cdot u\mbox{d}\textbf{x}=\int_{\Omega} \nabla\cdot (J\frac{|u|^{2}}{2})\mbox{d}\textbf{x}=0,\\
& \int_{\Omega} ( \nabla\cdot(\rho^{k} u\otimes u)-(\nabla\rho^{k}\cdot u )\frac{u}{2}      )\cdot u\mbox{d}\textbf{x}=0,\\
& (\rho u-\rho^{k} u^{k})\cdot u=(\rho\frac{|u|^{2}}{2}-\rho^{k}\frac{|u^{k}|^{2}}{2}   )+(\rho-\rho^{k})\frac{|u|^{2}}{2}+\rho^{k}\frac{|u-u^{k}|^{2}}{2}.
\end{aligned}
\end{equation}

\begin{The}\label{THEOREM-1}
Presuming \textbf{Assumption \ref{a1}} is effective and  ($\phi_{h}^{k+1},  \omega_{h}^{k+1}, \textbf{u}_{h}^{k+1},  p_{h}^{k+1},    \textbf{B}_{h}^{k+1}$) are a solution of the scheme (\ref{fully discrete scheme}). Then the scheme is unconditionally energy stable in following sense
\begin{equation}\label{45}
\begin{aligned}
E( \phi_{h}^{k+1}, \textbf{u}_{h}^{k+1}, \textbf{B}_{h}^{k+1} )-E(\phi_{h}^{k}, \textbf{u}_{h}^{k}, \textbf{B}_{h}^{k} ) \leq 0,
\end{aligned}
\end{equation}
where
\begin{equation}\label{E}
E(\phi_{h}^{k+1}, \textbf{u}_{h}^{k+1}, \textbf{B}_{h}^{k+1} )=\frac{1}{2}\rho_{h}^{k+1}\| \textbf{u}_{h}^{k+1}\|^{2}+\frac{1}{2\mu}\|\textbf{B}_{h}^{k+1}\|^{2}+\frac{\gamma\varepsilon}{2}\|\nabla \phi_{h}^{k+1}\|^{2}+\frac{\gamma\varepsilon}{4}\|(\phi_{h}^{k+1})^{2}-1\|^{2}.
\end{equation}

\end{The}

\begin{proof}

Letting $(\textbf{v}_{h}, q_{h})$=$\Delta t(\textbf{u}_{h}^{k+1},  p_{h}^{k+1})$ in scheme (\ref{fully discrete scheme}), we obtain
\begin{equation*}
\begin{aligned}
&( \rho_{h}^{k+1} \textbf{u}_{h}^{k+1}-\rho_{h}^{} \textbf{u}_{h}^{k},\textbf{u}_{h}^{k+1}) +\Delta t( \nabla\cdot (\rho_{h}^{k}  \textbf{u}_{h}^{k+1} \otimes\textbf{u}_{h}^{k+1}), \textbf{u}_{h}^{k+1})+ (\phi_{h}^{k}  \cdot \nabla\omega_{h}^{k+1}, \textbf{u}_{h}^{k+1})+2  \|\sqrt{\eta(\phi_{h}^{k})}D(\textbf{u}_{h}^{k+1})\|^{2} )\\
&+\frac{\Delta t}{\mu}(\nabla\times \textbf{B}_{h}^{k+1}, \textbf{u}_{h}^{k+1}\times \textbf{B}_{h}^{k})-\Delta t(\nabla\cdot(\textbf{u}_{h}^{k+1} \otimes (\frac{\rho_{2}-\rho_{1}}{2} M(\phi_{h}^{k})\nabla\omega_{h}^{k+1})), \textbf{u}_{h}^{k+1})=0,    \\
\end{aligned}
\end{equation*}
and letting $\textbf{C}_{h}$=$\frac{\Delta t}{\mu}\textbf{B}_{h}^{k+1}$ in scheme (\ref{fully discrete scheme}), we have
\begin{equation*}
\begin{aligned}
&\frac{1}{\mu}(\textbf{B}_{h}^{k+1}-\textbf{B}_{h}^{k}, \textbf{B}_{h}^{k+1}) + \frac{\Delta t}{\mu^{2}} \|\frac{1}{\sqrt{\sigma (\phi_{h}^{k})}} \textbf{B}_{h}^{k+1}\|^{2}   -\frac{\Delta t}{\mu}(\textbf{u}_{h}^{k+1}\times \textbf{B}_{h}^{k}, \nabla\times \textbf{B}_{h}^{k+1})=0,\\
\end{aligned}
\end{equation*}
letting $(\psi_{h}, \chi_{h})$=$\Delta t( \omega_{h}^{k+1}, d_{t} \phi_{h}^{k+1} )$ in scheme (\ref{fully discrete scheme}), we get
\begin{equation*}
\begin{aligned}
&(\phi_{h}^{k+1}-\phi_{h}^{k}, \omega_{h}^{k+1}) -\Delta t(\phi_{h}^{k} \textbf{u}_{h}^{k+1}, \nabla \omega_{h}^{k+1})+ \Delta t(M(\phi_{h}^{k})\nabla\omega_{h}^{k+1}, \nabla\omega_{h}^{k+1})=0, \label{eqn1}\\
&(\omega_{h}^{k+1}, \phi_{h}^{k+1}-\phi_{h}^{k})=\gamma\varepsilon(\nabla\phi_{h}^{k+1}, \nabla(\phi_{h}^{k+1}-\phi_{h}^{k}))+ \frac{\gamma}{\varepsilon}(f(\phi_{h}^{k+1}), \phi_{h}^{k+1}-\phi_{h}^{k}),\\
\end{aligned}
\end{equation*}
adding above equations  and combining it with the first and second  identity of (\ref{e-1}), (\ref{e-2}) to obtain
\begin{equation}\label{energy-zong}
\begin{aligned}
&\frac{1}{2}(\rho_{h}^{k+1}\|\textbf{u}_{h}^{k+1}\|^{2}- \rho_{h}^{k}\| \textbf{u}_{h}^{k}\|^{2}) +\frac{1}{2\mu}(\|\textbf{B}_{h}^{k+1}\|^{2}-\|\textbf{B}_{h}^{k}\|^{2})+\frac{\gamma\varepsilon}{2}(\|\nabla \phi_{h}^{k+1}\|^{2}-\|\nabla \phi_{h}^{k}\|^{2}) \\
&+\frac{\gamma}{4\varepsilon}(\|(\phi_{h}^{k+1})^{2}-1\|^{2}-\|(\phi_{h}^{k})^{2}-1\|^{2})+\frac{1}{2}\rho_{h}^{k}\| \textbf{u}_{h}^{k+1}-\textbf{u}_{h}^{k}\|^{2}+2\Delta t\|\sqrt{\eta(\phi_{h}^{k})}D(\textbf{u}_{h}^{k+1})\|^{2}\\
&+\frac{1}{2\mu}\|\textbf{B}_{h}^{k+1}-\textbf{B}_{h}^{k}\|^{2}+\frac{\Delta t}{\mu^{2}} \|\frac{1}{\sqrt{\sigma (\phi_{h}^{k})}}\textbf{B}\|^{2}_{\textbf{H}_{n}^{1}} +\Delta t \|\sqrt{M(\phi_{h}^{k})}\nabla\omega_{h}^{k+1}\|^{2}+\frac{\gamma\varepsilon}{2}\|\nabla\phi_{h}^{k+1} -\nabla\phi_{h}^{k} \|^{2}\\
&+\frac{\gamma}{4\varepsilon}\|(\phi_{h}^{k+1} )^{2} -(\phi_{h}^{k} )^{2}   \|^{2}+\frac{\gamma}{2\varepsilon} (\phi_{h}^{k+1})^{2}\| \phi_{h}^{k+1} - \phi_{h}^{k} \|^{2}+\frac{\gamma}{2\varepsilon} \| \phi_{h}^{k+1} - \phi_{h}^{k} \|^{2}=0.
\end{aligned}
\end{equation}

Obviously, it can be concluded that (\ref{45}) and (\ref{E}) is effective. Thus, the numerical scheme (\ref{fully discrete scheme}) is unconditionally energy stable.
\end{proof}

Since the solution of a discretized phase eqution does not necessarily satisfy a maximum principle,
We shall set $\hat{\phi}_{h}^{k+1}$  as
\begin{equation}\label{cut-off1}
\begin{aligned}
&\hat{\phi}_{h}^{k+1}=\left\{
            \begin{aligned}
            &\phi_{h}^{k+1}, \quad |\phi_{h}^{k+1}|<1, \\
            &\rm sign(\phi_{h}^{k+1}), \quad   |\phi_{h}^{k+1}|>1,
            \end{aligned}
             \right.\\
\end{aligned}
\end{equation}
Thus, the density, electric conductivity, diffusion mobility, and viscosity can be updated as
\begin{equation}\label{cut-off2}
\left\{
\begin{aligned}
&\rho_{h}^{k+1}=\frac{\rho_{2}-\rho_{1}}{2}\hat{\phi}_{h}^{k+1}+\frac{\rho_{1}+\rho_{2}}{2}, \\
&\sigma_{h}^{k+1}=\frac{\rho_{2}-\rho_{1}}{2}\hat{\phi}_{h}^{k+1}+\frac{\rho_{1}+\rho_{2}}{2}, \\
&M_{h}^{k+1}=\frac{\rho_{2}-\rho_{1}}{2}\hat{\phi}_{h}^{k+1}+\frac{\rho_{1}+\rho_{2}}{2},\\
&\eta_{h}^{k+1}=\frac{\rho_{2}-\rho_{1}}{2}\hat{\phi}_{h}^{k+1}+\frac{\rho_{1}+\rho_{2}}{2},
\end{aligned}
\right.
\end{equation}
From the definition of the cut-off function $\hat{\phi}_{h}^{k+1}$ given in equations (\ref{cut-off1}), we have
the following restriction conditions
\begin{equation}\label{canshu}
\begin{aligned}
&0<\rho_{1}\leq\rho_{h}^{k+1}\leq\rho_{2},\\
&0<\sigma _{1}\leq \sigma_{h}^{k+1}\leq\sigma _{2},\\
&0<M_{1}\leq M_{h}^{k+1}\leq M_{2},\\
&0<\eta_{1}\leq\eta_{h}^{k+1} \leq\eta_{2}.\\
\end{aligned}
\end{equation}

\begin{The}\label{t-fully}
Supposing the \textbf{Assumption \ref{a1}} is effective, let $E(\phi_{h}^{k+1}, \textbf{u}_{h}^{k+1}, \textbf{B}_{h}^{k+1})<\infty$, and  ($\phi_{h}^{k+1},  \omega_{h}^{k+1}, \textbf{u}_{h}^{k+1},  p_{h}^{k+1},    \textbf{B}_{h}^{k+1}$) ($k=0, \cdots, N-1$) are a solution of the scheme (\ref{fully discrete scheme}). For any $0\leq m\leq N-1$, we have  following estimates, where  $C_{i}$, $i$=9, $\cdots$, 23 are general  constants independent $h$ and $\Delta t$,
\begin{subequations}
\begin{align}
& \int_{\Omega}\phi_{h}^{m+1}\mbox{d}\textbf{x}= \int_{\Omega}\phi_{h}^{0}\mbox{d}\textbf{x}, \label{theorem-fai}\\
& \max_{0\leq k\leq N-1}\left\{ \rho_{h}^{k+1}\| \textbf{u}_{h}^{k+1}\|^{2}+ \frac{1}{\mu} \|\textbf{B}_{h}^{k+1}\|^{2}+ \gamma\varepsilon \|\nabla \phi_{h}^{k+1}\|^{2} + \frac{\gamma}{\varepsilon} \|(\phi_{h}^{k+1})^{2}-1\|^{2}    \right \}\leq C_{10}, \label{theorem-L2}\\
&\Delta t\sum_{k=0}^{m}\left(\eta_{1}\|D(\textbf{u}_{h}^{k+1})\|^{2}+  \frac{1}{\mu^{2}\sigma_{2} }(\|\nabla\times \textbf{B}_{h}^{k+1}\|^{2}+\|\nabla \cdot\textbf{B}_{h}^{k+1}\|^{2}) + M_{1}\|\nabla\omega_{h}^{k+1}\|^{2}    \right)\leq C_{11}, \label{theorem-h1}\\
&\sum_{k=0}^{m}\left(  \gamma\varepsilon \|\nabla\phi_{h}^{k+1} -\nabla\phi_{h}^{k} \|^{2} + \frac{\gamma}{\varepsilon}  \| \phi_{h}^{k+1} - \phi_{h}^{k} \|^{2} \right)\leq C_{12}, \label{theorem-faichaH1}\\
&\sum_{k=0}^{m} \left ( \rho_{h}^{k} \| \textbf{u}_{h}^{k+1}-\textbf{u}_{h}^{k} \|^{2} +\frac{1}{\mu}\|\textbf{B}_{h}^{k+1}-\textbf{B}_{h}^{k}\|^{2} \right)\leq C_{13},\label{theorem-uchaL2}\\
&\max_{0\leq k\leq N-1}\|\phi_{h}^{k+1}\|^{2}_{H^{1}}\leq C_{14}, \label{theorem-faiH1}\\
&\max_{0\leq k\leq N-1}\|\rho_{h}^{k+1}\|^{2}_{H^{1}}\leq C_{15}, \label{rho-faiH1}\\
&\Delta t\sum_{k=0}^{m}  \|\omega _{h}^{k+1}\|_{H^{1}}^{2}\leq C_{16}T+\frac{C_{11}}{M_{1}}, \label{theorem-WH1}\\
& \Delta t\sum_{k=0}^{m}  \|d_{t}\phi_{h}^{k+1}\|^{2}_{(H^{1})'}  \leq C_{17},\label{theorem-faiCHAl2}\\
& \Delta t\sum_{k=0}^{m} \|d_{t}\rho_{h}^{k+1}\|^{2}_{(H^{1})'}  \leq C_{18},\label{theorem-rhoCHAl2}\\
& \Delta t\sum_{k=0}^{m}(\|d_{t}(\rho_{h}^{k+1} \textbf{u}_{h}^{k+1}  )\|_{(\textbf{H}_{0}^{1})'}^{\frac{12}{6+d}}+\|p_{h}^{k+1}\|^{\frac{12}{6+d}} +\| d_{t}\textbf{B}_{h}^{k+1} \|^{\frac{4}{d}}_{(\textbf{H}_{n}^{1})'}   )\leq C_{19}(T+1). \label{theorem-PRESSURE}
\end{align}
\end{subequations}

\begin{proof}
Letting $\psi_{h}$=1 in equation (\ref{eqn123}), we can find
\begin{equation*}
(\phi_{h}^{k+1}, 1)=(\phi_{h}^{k}, 1)=\cdots=(\phi_{h}^{0}, 1),
\end{equation*}
we get equation (\ref{theorem-fai}) is valid. And  we can obtain   (\ref{theorem-L2})-(\ref{theorem-uchaL2}) through equation (\ref{energy-zong}). Applying
\begin{equation*}
(F(\phi_{h}^{k+1}), 1)=(\frac{1}{4}((\phi_{h}^{k+1})^{2}-1)^{2}, 1  )\geq \|\phi_{h}^{k+1}\|^{2}-2|\Omega|.
\end{equation*}
Linking above inequality and equation (\ref{theorem-L2}), we  find the equation (\ref{theorem-faiH1}) is derived. Letting $\chi_{h}$=$\Delta t \omega_{h}^{k+1}$ in equation (\ref{22}), and sum from $k$=$0, 2, \cdots, m$
\begin{equation*}
\begin{aligned}
\sum_{k=0}^{m} \Delta t\|\omega_{h}^{k+1}\|^{2}&=\sum_{k=1}^{n} (\Delta t\gamma\varepsilon(\nabla\phi_{h}^{k+1}, \nabla\omega_{h}^{k+1})+ \frac{\Delta t\gamma}{\varepsilon}(f(\phi_{h}^{k+1}), \omega_{h}^{k+1})  )\\
&=\sum_{k=1}^{n} ( \Delta t\gamma\varepsilon(\nabla\phi_{h}^{k+1}, \nabla\omega_{h}^{k+1})+ \frac{\Delta t\gamma}{\varepsilon}((\phi_{h}^{k+1})^{3}-\phi_{h}^{k}, \omega_{h}^{k+1}) )\\
&\leq \sum_{k=1}^{n}  (\Delta t \gamma\varepsilon\|\nabla\phi_{h}^{k+1}\| \|\nabla\omega_{h}^{k+1}\|+ \frac{\Delta t\gamma}{\varepsilon} (\|\phi_{h}^{k+1}\|^{3}_{L^{6}}+\|\phi_{h}^{k}\|) \| \omega_{h}^{k+1}\|  )\\
& \leq  \sum_{k=1}^{n} (\frac{\Delta t\gamma^{2}\varepsilon^{2}}{2}  \|\nabla\phi_{h}^{k+1}\|^{2}+ \frac{\Delta t}{2}  \|\nabla\omega_{h}^{k+1}\|^{2}+\frac{ \Delta t}{2} \|\omega_{h}^{k+1}\|^{2}+ \frac{\Delta t\gamma^{2}}{2\varepsilon^{2}} (\|\phi_{h}^{k+1}\|^{6}_{H^{1}}+ \|\phi_{h}^{k}\|^{2}) ).
\end{aligned}
\end{equation*}
Combining (\ref{theorem-L2})-(\ref{theorem-h1}), (\ref{theorem-faiH1}) and above inequality, we obtain (\ref{theorem-WH1}), for $C_{16}$=$\max \{C_{10}\gamma\varepsilon, \frac{\gamma^{2} C_{14}^{3}}{\varepsilon^{2}}, \frac{\gamma^{2} C_{14}}{\varepsilon^{2}} \}$.

Taking $\psi_{h}= P_{\phi h}\psi$ in equation (\ref{eqn123}) and  combining $H^{1}$ stability inequality (\ref{H3}) we have
\begin{equation*}
\begin{aligned}
(d_{t}\phi_{h}^{k+1}, \psi )&= (d_{t}\phi_{h}^{k+1}, P_{\phi h}\psi )\\
 &= \left( (\phi_{h}^{k} \textbf{u}_{h}^{k+1}, \nabla P_{\phi h}\psi )-  (M(\phi_{h}^{k})\nabla\omega_{h}^{k+1}, \nabla P_{\phi h}\psi) \right),\\
&\leq  (\|\phi_{h}^{k}\|_{L^{3}}\|\textbf{u}_{h}^{k+1}\|_{L^{6}} + M_{2}\|\nabla\omega_{h}^{k+1}\|  )\|\nabla P_{\phi h}\psi \| \\
&\leq  C_{\phi}( \|\phi_{h}^{k}\|_{H^{1}}\|\nabla \textbf{u}_{h}^{k+1}\|+ M_{2}\|\nabla\omega_{h}^{k+1}\|  )\|\nabla\psi \|.
\end{aligned}
\end{equation*}
Combining (\ref{theorem-h1}), (\ref{theorem-faiH1}), and above inequality, we can get (\ref{theorem-faiCHAl2}), where $C_{17}$=$ \frac{C_{\phi}^{2}C_{14}^{2}C_{11}}{\eta_{1}}+\frac{C_{\phi}^{2}M_{2}^{2}C_{11}}{M_{1}}$.

Thus, the inequality (\ref{theorem-faiCHAl2}) can be proved. According to  the definition of $\rho$, which is the linear relationship to $\phi$.  In view of this point, the (\ref{rho-faiH1}) and (\ref{theorem-rhoCHAl2}) is also valid, for $C_{15}$=$\frac{\rho_{2}-\rho_{1}}{2}C_{14}+\frac{\rho_{2}+\rho_{1}}{2}$ and $C_{18}$=$\frac{\rho_{2}-\rho_{1}}{2}C_{17}+\frac{\rho_{2}+\rho_{1}}{2}$.

Then we need to give the proof of inequality (\ref{theorem-PRESSURE}) in details. Firstly, the discrete inverse Stokes operator $I_{h}$ from $(\textbf{H}_{0}^{1})'$ to $\textbf{X}_{h}$ as follows: for all $\textbf{v}\in (\textbf{H}_{0}^{1})'$, the $(I_{h}(\textbf{v}), r_{h})\in \textbf{X}_{h}\times M_{h}$ satisfies the following equations
\begin{equation*}
\begin{aligned}
&(\nabla I_{h}(\textbf{v}), \nabla \textbf{v}_{h})+(\nabla\cdot \textbf{v}_{h}, r_{h})=(\textbf{v}, \textbf{v}_{h}), \quad \forall \textbf{v}_{h}\in \textbf{X}_{h},\\
&(\nabla\cdot I_{h}(\textbf{v}), q_{h})=0,\quad \forall q_{h}\in M_{h}.
\end{aligned}
\end{equation*}

See for instance \cite{2019error}, if  $\textbf{u}_{h}\in \textbf{V}_{h}$, there exist a positive constant $c$ independent of $h$ such that
\begin{equation*}
\sup_{\textbf{v}_{h}\in \textbf{X}_{h}\setminus\{\textbf{0}\}}\frac{\langle\textbf{u}_{h}, \textbf{v}_{h}\rangle}{\|\nabla \textbf{v}_{h}\|}\leq c\|\nabla I_{h}(\textbf{u}_{h})\|,
\end{equation*}

Thus, we hold
\begin{equation*}
\begin{aligned}
\|\textbf{u}_{h}\|_{(\textbf{H}_{0}^{1})'}=\sup_{\textbf{v}\in \textbf{H}_{0}^{1}\setminus\{\textbf{0}\}}\frac{\langle\textbf{u}_{h}, \textbf{v} \rangle}{\|\nabla \textbf{v} \|}=\sup_{\textbf{v}\in \textbf{H}_{0}^{1}\setminus\{\textbf{0}\}}\frac{\langle\textbf{u}_{h}, \textbf{v} \rangle}{\|\nabla P_{uh}\textbf{v} \|}\cdot\frac{\|\nabla P_{uh}\textbf{v} \|}{\|\nabla \textbf{v} \|}\leq c\|\nabla I_{h}(\textbf{u}_{h})\|.
\end{aligned}
\end{equation*}

Based on  \cite{2011Existence}, we consider the following elliptic boundary value problem to obtain $ \|\omega_{h}^{k+1}\|_{H^{2}}\leq C_{20}( \|\omega_{h}^{k+1}\|_{H^{1}}+\|f_{h}^{k+1}\|  )$, for given $f_{h}^{k+1}\in L^{2}$, and we limit $\|f_{h}^{k+1}\|\leq C_{21}$
\begin{equation*}
\left\{
\begin{aligned}
&-\nabla\cdot (M(\phi_{h}^{k})\nabla\omega_{h}^{k+1})+\int_{\Omega}\omega_{h}^{k+1} d\textbf{x}=f_{h}^{k+1}, \quad \rm{in}\, \Omega, \\
&\partial_{n}\omega_{h}^{k+1}|_{\partial\Omega}=0.
\end{aligned}
\right.
\end{equation*}

And taking  $\textbf{v}_{h}$=$I_{h}(d_{t}(\rho_{h}^{k+1} \textbf{u}_{h}^{k+1}))$ in equation (\ref{equau}), we know
\begin{equation*}
\begin{aligned}
&\|\nabla I_{h}(d_{t }(\rho_{h}^{k+1}   \textbf{u}_{h}^{k+1})) \|^{2}=(d_{t}(\rho_{h}^{k+1} \textbf{u}_{h}^{k+1}),I_{h}(d_{t} ( \rho_{h}^{k+1}\textbf{u}_{h}^{k+1}) ))= ( \rho_{h}^{k}  \textbf{u}_{h}^{k+1} \otimes\textbf{u}_{h}^{k+1}, \nabla I_{h}(d_{t} (\rho_{h}^{k+1}  \textbf{u}_{h}^{k+1} )))\\
&-( \textbf{u}_{h}^{k+1} \otimes (\frac{\rho_{2}-\rho_{1}}{2} M(\phi_{h}^{k})\nabla\omega_{h}^{k+1}), \nabla I_{h}(d_{t}( \rho_{h}^{k+1} \textbf{u}_{h}^{k+1} )))-2 (\eta (\phi_{h}^{k})D(\textbf{u}_{h}^{k+1}), D( I_{h}(d_{t} (\rho_{h}^{k+1} \textbf{u}_{h}^{k+1} ))))\\
&-\frac{1}{\mu}(\nabla\times \textbf{B}_{h}^{k+1}, I_{h}(d_{t} (\rho_{h}^{k+1}  \textbf{u}_{h}^{k+1}) )\times \textbf{B}_{h}^{k})-(\phi_{h}^{k}  \cdot \nabla\omega_{h}^{k+1},I_{h}(d_{t} (\rho_{h}^{k+1} \textbf{u}_{h}^{k+1} )))\\
&\leq  (\rho_{2}C_{3}\|\textbf{u}_{h}^{k+1}\|^{\frac{6-d}{6}}\|\nabla\textbf{u}_{h}^{k+1}\|^{\frac{d}{6}}\|\nabla\textbf{u}_{h}^{k+1}\|+\frac{(\rho_{2}-\rho_{1})M_{2}}{2}\| \textbf{u}_{h}^{k+1}\|_{L^{3}}  \|\nabla\omega_{h}^{k+1}\|_{L^{6}}+2\eta_{2}\|\nabla \textbf{u}_{h}^{k+1}\|\\
&\quad + \frac{C_{6}}{\mu}\| \textbf{B}_{h}^{k}\|^{\frac{6-d}{d}} \| \textbf{B}_{h}^{k}\|_{\textbf{H}^{1}_{n} }^{\frac{d}{6}}\| \textbf{B}_{h}^{k+1}\|_{\textbf{H}^{1}_{n} }+ \|\phi_{h}^{k}\|_{H^{1}}\|\nabla\omega_{h}^{k+1}\|    )\|\nabla I_{h}(d_{t} (\rho_{h}^{k+1} \textbf{u}_{h}^{k+1}) )\|.
\end{aligned}
\end{equation*}
Taking  $\textbf{C}_{h}=P_{B h}\textbf{C}$ for any $\textbf{C}\in \textbf{H}_{n}^{1}$ in equation, we have
\begin{equation*}
\begin{aligned}
 (d_{t}\textbf{B}_{h}^{k+1}, \textbf{C})&=(d_{t}\textbf{B}_{h}^{k+1}, P_{B h}\textbf{C})= -\frac{1}{\mu}(  \frac{1}{\sigma (\phi_{h}^{k})} \nabla\times \textbf{B}_{h}^{k+1}, \nabla\times P_{B h}\textbf{C})- \frac{1}{\mu}  (\frac{1}{ \sigma (\phi_{h}^{k})}\nabla\cdot\textbf{ B}_{h}^{k+1}, \nabla\cdot P_{B h}\textbf{C})  \\
 & \quad +(\textbf{u}_{h}^{k+1}\times \textbf{B}_{h}^{k}, \nabla\times P_{B h}\textbf{C})\\
 &\leq (\frac{1}{\mu\sigma_{1}}\|\textbf{B}_{h}^{k+1}\|_{\textbf{H}_{n}^{1}} + C_{4}C_{7}\|\textbf{u}_{h}^{k+1}\|^{\frac{4-d}{4}}\|\nabla\textbf{u}_{h}^{k+1}\|^{\frac{d}{4}} \|\textbf{B}_{h}^{k}\|^{\frac{4-d}{4}}\|\textbf{ B}_{h}^{k}\|_{\textbf{H}_{n}^{1}}^{\frac{d}{4}} )\|\textbf{C}\|_{\textbf{H}_{n}^{1}}.
\end{aligned}
\end{equation*}

Combining equations (\ref{theorem-h1}), (\ref{theorem-L2}), (\ref{theorem-faiH1}), (\ref{theorem-WH1}), we obtain equation (\ref{theorem-PRESSURE}), where
\begin{equation*}
\begin{aligned}
C_{22}&= \left(\frac{d}{6+d}((2\eta_{2})^{\frac{12}{6+d}} +C_{14}^{\frac{6}{6+d}}) +
\frac{(\frac{\rho_{2}-\rho_{1}}{2}M_{1}C_{3})^{\frac{12}{6+d}}(\frac{C_{10}}{\rho_{1}})^{\frac{6-d}{6+d}}6C_{20}^{2}}{6+d}(  C_{16}+  C_{21}^{2}) \right)T\\
&\quad +(\rho_{2}C_{3})^{\frac{12}{6+d}}(\frac{C_{10}}{\rho_{1}})^{\frac{6-d}{6+d}}\frac{C_{11}}{\eta_{1}} + (\frac{C_{6}}{\mu})^{\frac{12}{6+d}}C_{10}^{\frac{6-d}{6+d}}\mu^{\frac{18+d}{6+d}}C_{11}\sigma_{2}
+(\frac{(2\eta_{2})^{\frac{12}{6+d}}}{\eta_{1}}+ \frac{C_{14}^{\frac{6}{6+d}}}{M_{1}})\frac{C_{11}6}{(6+d)} \\
&\quad +
\frac{(\frac{\rho_{2}-\rho_{1}}{2}M_{1}C_{3})^{\frac{12}{6+d}}(\frac{C_{10}}{\rho_{1}})^{\frac{6-d}{6+d}}}{6+d}( \frac{C_{11}d}{\eta_{1}} +\frac{6C_{20}^{2}C_{11}}{M_{1}} ),\\
C_{23}&=\frac{d-2}{(\mu\sigma_{1})^{\frac{4}{d}}d}T+ \frac{2C_{11}\mu^{2}\sigma_{2}}{(\mu\sigma_{1})^{\frac{4}{d}}d}+(C_{4}C_{7})^{\frac{4}{d}}(\frac{\mu}{\rho_{1}})^{\frac{4-d}{2d}}C_{10}\frac{4-d}{d}C_{11}(\frac{1}{2\eta_{1}}+\frac{\mu^{2}\sigma_{2}}{2}).
\end{aligned}
\end{equation*}
According to inf-sup condition (\ref{41}), we can obtain the convergence of  pressure field, for $C_{19}=\max\{C_{22}, C_{23}\}$. The proof is ended.

\end{proof}
\subsection{Existence of the solution of the scheme}
By means of Leray-Schauder fixed point theorem (the Theorem 11.3 of \cite{2001Trudinger}), we obtain the existence of   fully discrete scheme (\ref{fully discrete scheme}).
\begin{lemma}\label{lemma1}
Letting $\Pi$ is a compact mapping of a Banach space $\textbf{B}$ into itself, and there exist a constant $M_{c}>0$, s.t.
\begin{equation}
\|\textbf{X}\|_{\textbf{B}}< M_{c},
\end{equation}
for all $\textbf{X}\in \textbf{B}$ and $\theta\in[0, 1]$ guaranteeing $\textbf{X}=\theta\Pi \textbf{X}$. Then the $\Pi$ has a fixed point.
\end{lemma}

\begin{The}\label{theorem-3}
Supposing  \textbf{Assumption \ref{a1}.} is valid and initial data $\phi_{0}, \textbf{u}_{0}, \textbf{B}_{0}$ satisfy $E(\phi_{0}$, $\textbf{u}_{0}$, $\textbf{B}_{0} )<\infty$. For given any $  \Delta t, h> 0$, there exist a solution ($\phi_{h}^{k+1}$, $\omega_{h}^{k+1}$, $\textbf{u}_{h}^{k+1}$, $p_{h}^{k+1}$, $\textbf{B}_{h}^{k+1}$) of the scheme (\ref{fully discrete scheme}).
\end{The}

\begin{proof}
We define  $\Pi$: $ Y_{h}\times Y_{h}\times X_{h}\times M_{h}\times W_{h}\longrightarrow  Y_{h}\times Y_{h}\times X_{h}\times M_{h}\times W_{h}$  by the   relationship $\Pi (\phi_{h}^{k+1},  \omega_{h}^{k+1}, \textbf{u}_{h}^{k+1}, p_{h}^{k+1}, \textbf{B}_{h}^{k+1})$= $(\hat{\phi}_{h}^{k+1}, \hat{\omega}_{h}^{k+1}, \hat{\textbf{u}}_{h}^{k+1}, \hat{p}_{h}^{k+1}, \hat{\textbf{B}}_{h}^{k+1})$, for $(\hat{\phi}_{h}^{k+1}$,  $\hat{\omega}_{h}^{k+1}$, $\hat{\textbf{u}}_{h}^{k+1}$, $\hat{p}_{h}^{k+1}$, $\hat{\textbf{B}}_{h}^{k+1}$)$\in  Y_{h}\times Y_{h}\times X_{h}\times M_{h}\times W_{h}$ satisfies the following scheme
\begin{subequations}\label{L-S-p-t}
\begin{align}
& 2 (\eta (\phi_{h}^{k})D(\hat{\textbf{u}}_{h}^{k+1}), D( \textbf{v}_{h})   )-(\nabla\cdot \textbf{v}_{h},  \hat{p}_{h}^{k+1}) =-(\frac{\rho_{h}^{k+1} \textbf{u}_{h}^{k+1}-\rho_{h}^{k} \textbf{u}_{h}^{k}}{\Delta t}, \textbf{v}_{h}) +( (\rho_{h}^{k}  \textbf{u}_{h}^{k+1} \otimes\textbf{u}_{h}^{k+1}), \nabla\textbf{v}_{h}) \nonumber \\
& -( (\textbf{u}_{h}^{k+1} \otimes (\frac{\rho_{2}-\rho_{1}}{2} M(\phi_{h}^{k})\nabla\omega_{h}^{k+1}), \nabla\textbf{v}_{h})  -\frac{1}{\mu}(\nabla\times \textbf{B}_{h}^{k+1}, \textbf{v}_{h}\times \textbf{B}_{h}^{k})-(\phi_{h}^{k}  \cdot \nabla\omega_{h}^{k+1}, \textbf{v}_{h}),\label{g1}\\
&(\nabla\cdot\hat{\textbf{u}}_{h}^{k+1}, q_{h}) =0,\label{g2}\\
& (\frac{\hat{\textbf{B}}_{h}^{k+1}-\textbf{B}_{h}^{k}}{\Delta t}, \textbf{C}_{h}) + \frac{1}{\mu}(  \frac{1}{\sigma (\phi_{h}^{k})} \nabla\times \hat{\textbf{B}}_{h}^{k+1}, \nabla\times \textbf{C}_{h})+ \frac{1}{\mu}  (\frac{1}{\sigma (\phi_{h}^{k})}\nabla\cdot\hat{\textbf{ B}}_{h}^{k+1}, \nabla\cdot \textbf{C}_{h}) \nonumber \\
& \quad  =(\textbf{u}_{h}^{k+1}\times \textbf{B}_{h}^{k}, \nabla\times \textbf{C}_{h}),\label{g3}\\
&(\frac{\hat{\phi}_{h}^{k+1}-\phi_{h}^{k}}{\Delta t}, \psi_{h}) + (M(\phi_{h}^{k})\nabla\hat{\omega}_{h}^{k+1}, \nabla\psi_{h})=(\phi_{h}^{k} \textbf{u}_{h}^{k+1}, \nabla \psi_{h}), \label{g4}\\
&(\hat{\omega}_{h}^{k+1}, \chi_{h}) -\gamma\varepsilon(\nabla\hat{\phi}_{h}^{k+1}, \nabla\chi_{h}) =\frac{\gamma}{\varepsilon}(f(\phi_{h}^{k+1}), \chi_{h})\label{g5},
\end{align}
\end{subequations}
for given $ (\phi_{h}^{k+1},  \omega_{h}^{k+1}, \textbf{u}_{h}^{k+1}, p_{h}^{k+1}, \textbf{B}_{h}^{k+1})\in Y_{h}\times   Y_{h}\times \textbf{X}_{h}\times M_{h}\times \textbf{W}_{h}$, and  $\forall (\psi_{h}, \chi_{h}, \textbf{v}_{h}, q_{h}, \textbf{C}_{h})\in  Y_{h}\times Y_{h}\times \textbf{X}_{h}\times M_{h}\times \textbf{W}_{h}$. Next, we will prove the map $\Pi$ satisfies the \textbf{Lemma \ref{lemma1}}  condition and then has a fixed point which is a solution of scheme (\ref{fully discrete scheme}). We  divide the proof into two steps next.

\textbf{Step 1.}
The classical Stokes problem (\ref{g1})-(\ref{g2}) and Maxwell problem (\ref{g3}) can be proved well-posed. Given
$\phi_{h}^{k+1}, \phi_{h}^{k}\in Y_{h}$, $\textbf{u}_{h}^{k+1}\in \textbf{X}_{h}$, the Cahn-Hilliard problem (\ref{g4})-(\ref{g5})  can be rewritten as
\begin{equation*}
\left\{
\begin{aligned}
&a(\hat{\phi}_{h}^{k+1}, \psi_{h})+b(\hat{\omega}_{h}^{k+1},  \psi_{h})=(f,  \psi_{h}),\\
&c(\hat{\omega}_{h}^{k+1}, \chi_{h})-b'(\hat{\phi}_{h}^{k+1},  \chi_{h}) =(g, \chi_{h}),
\end{aligned}
\right.
\end{equation*}
to finding $\hat{\phi}_{h}^{k+1}\in Y_{h}$ and $\hat{\omega}_{h}^{k+1}\in Y_{h}$ and any $\psi_{h}\in Y_{h}$, $\chi_{h}\in Y_{h}$, where
\begin{equation*}
\begin{aligned}
&a(\hat{\phi}_{h}^{k+1}, \psi_{h})=\frac{1}{\Delta t}(\hat{\phi}_{h}^{k+1}, \psi_{h}),\quad  b(\hat{\omega}_{h}^{k+1}, \psi_{h})=(M(\phi_{h}^{k})\nabla\hat{\omega}_{h}^{k+1}, \nabla\psi_{h}), \\
&c(\hat{\omega}_{h}^{k+1}, \chi_{h})=(\hat{\omega}_{h}^{k+1}, \chi_{h}),\quad  b'(\hat{\phi}_{h}^{k+1}, \chi_{h}) =\gamma\varepsilon(\nabla\hat{\phi}_{h}^{k+1}, \nabla\chi_{h}),  \\
&(f, \psi_{h})= (\phi_{h}^{k} \textbf{u}_{h}^{k+1}, \nabla \psi_{h})+\frac{1}{\Delta t}(\phi_{h}^{k}, \psi_{h}),\quad (g, \chi_{h})=\frac{\gamma}{\varepsilon}(f(\phi_{h}^{k+1}),\chi_{h}).
\end{aligned}
\end{equation*}
We can derive  $\|\nabla \psi_{h}\|=0$ for any $\psi_{h}\in \{\psi_{h}\in Y_{h}, b'( \psi_{h}, \chi_{h})=0, \forall\chi_{h}\in  Y_{h}\}$ and $c(\cdot, \cdot)$ is coercive
\begin{equation*}
c(\psi_{h},\psi_{h})=\|\psi_{h}\|^{2}=\|\psi_{h}\|_{H^{1}}.
\end{equation*}
Considering $a(\cdot, \cdot)$ is continuous, positive semi-definite and symmetric,  $b(\cdot, \cdot)$, $b'(\cdot, \cdot)$, $c(\cdot, \cdot)$ are continuous.  Refer to \cite{2019A} and the Section II.1.2 of  \cite{1991MIX},  the Cahn-Hilliard problem  (\ref{g4})-(\ref{g5}) is well-posed for given $\phi_{h}^{k+1}$,  $\phi_{h}^{k}\in Y_{h}$ and $\textbf{u}_{h}^{k+1}\in \textbf{X}_{h}$. The scheme (\ref{L-S-p-t}) is considered in finite element space, thus  $\Pi$ is compact map.

\textbf{Step 2.} Next, we need to prove the compressibility of the map $\Pi$ and the boundedness of    $(\hat{\phi}_{h}^{k+1},  \hat{\omega}_{h}^{k+1}, \hat{\textbf{u}}_{h}^{k+1}, \hat{p}_{h}^{k+1}, \hat{\textbf{B}}_{h}^{k+1})\in Y_{h}\times  Y_{h}\times \textbf{X}_{h}\times M_{h}\times \textbf{W}_{h}$ which satisfy the
\begin{equation}\label{boundedness}
 \|\hat{\textbf{u}}_{h}^{k+1}\|_{H^{1}_{0}}^{2}+\|\hat{p}_{h}^{k+1}\|^{2}+\|\hat{\textbf{B}}_{h}^{k+1}\|_{H^{1}_{n}}^{2}+\|\hat{\phi}_{h}^{k+1}\|_{H^{1}}^{2}+\| \hat{\omega}_{h}^{k+1}\|_{H^{1}}^{2}< M_{c},
\end{equation}
for  $M_{c}>0$ is a constant, which is  independent of $\theta$ and $(\hat{\phi}_{h}^{k+1}, \hat{\omega}_{h}^{k+1}, \hat{\textbf{u}}_{h}^{k+1}, \hat{p}_{h}^{k+1}, \hat{\textbf{B}}_{h}^{k+1})$. If $\theta$=0, the above inequality (\ref{boundedness}) naturally holds. If  $\theta\in(0, 1]$, we consider
\begin{equation*}
\Pi (\hat{\phi}_{h}^{k+1},  \hat{\omega}_{h}^{k+1}, \hat{\textbf{u}}_{h}^{k+1}, \hat{p}_{h}^{k+1}, \hat{\textbf{B}}_{h}^{k+1})=\frac{1}{\theta} (\hat{\phi}_{h}^{k+1},  \hat{\omega}_{h}^{k+1}, \hat{\textbf{u}}_{h}^{k+1}, \hat{p}_{h}^{k+1}, \hat{\textbf{B}}_{h}^{k+1}),
\end{equation*}
thus the equations can be rewritten as
\begin{equation}\label{L-S-p-t-z}
\begin{aligned}
&\theta(\frac{\hat{\rho}_{h}^{k+1} \hat{\textbf{u}}_{h}^{k+1}-\rho_{h}^{k} \textbf{u}_{h}^{k}}{\Delta t},\textbf{v}_{h}) -\theta( \rho_{h}^{k}  \hat{\textbf{u}}_{h}^{k+1} \otimes\hat{\textbf{u}}_{h}^{k+1}, \nabla\textbf{v}_{h})+\theta( \hat{\textbf{u}}_{h}^{k+1} \otimes (\frac{\rho_{2}-\rho_{1}}{2} M(\phi_{h}^{k})\nabla\hat{\omega}_{h}^{k+1}), \nabla\textbf{v}_{h}) \nonumber \\
&  +2 (\eta (\phi_{h}^{k})D(\hat{\textbf{u}}_{h}^{k+1}), D( \textbf{v}_{h})   )-(\nabla\cdot \textbf{v}_{h},  \hat{p}_{h}^{k+1}) +\frac{\theta}{\mu}(\nabla\times \hat{\textbf{B}}_{h}^{k+1}, \textbf{v}_{h}\times \textbf{B}_{h}^{k})+\theta(\phi_{h}^{k}  \cdot \nabla\hat{\omega}_{h}^{k+1}, \textbf{v}_{h})=0,  \\
&(\nabla\cdot\hat{\textbf{u}}_{h}^{k+1}, q_{h}) =0,\\
& (\frac{\hat{\textbf{B}}_{h}^{k+1}-\theta\textbf{B}_{h}^{k}}{\Delta t}, \textbf{C}_{h}) + \frac{1}{\mu}(  \frac{1}{\sigma (\phi_{h}^{k})} \nabla\times \hat{\textbf{B}}_{h}^{k+1}, \nabla\times \textbf{C}_{h})+ \frac{1}{\mu}  (\frac{1}{ \sigma (\phi_{h}^{k}) }\nabla\cdot\hat{\textbf{B}}_{h}^{k+1}, \nabla\cdot \textbf{C}_{h}) \nonumber \\
& \quad  -\theta(\hat{\textbf{u}}_{h}^{k+1}\times \textbf{B}_{h}^{k}, \nabla\times \textbf{C}_{h})=0,\\
&(\frac{\hat{\phi}_{h}^{k+1}-\theta\phi_{h}^{k}}{\Delta t}, \psi_{h}) -\theta(\phi_{h}^{k} \hat{\textbf{u}}_{h}^{k+1}, \nabla \psi_{h})+ (M(\phi_{h}^{k})\nabla\hat{\omega}_{h}^{k+1}, \nabla\psi_{h})=0,  \\
&\gamma\varepsilon(\nabla\hat{\phi}_{h}^{k+1}, \nabla\chi_{h})+ \frac{\gamma\theta}{\varepsilon}((\hat{\phi}_{h}^{k+1})^{3}-\phi_{h}^{k}, \chi_{h})=(\hat{\omega}_{h}^{k+1}, \chi_{h}),
\end{aligned}
\end{equation}
taking $(\psi_{h}, \chi_{h}, \textbf{v}_{h}, q_{h}, \textbf{C}_{h})$=$2\Delta t(\hat{\omega}_{h}^{k+1}, \frac{\hat{\phi}_{h}^{k+1}-\theta\phi_{h}^{k}}{\Delta t}, \hat{\textbf{u}}_{h}^{k+1}, \hat{p}_{h}^{k+1},  \frac{1}{\mu}\hat{\textbf{B}}_{h}^{k+1})$ in above equations, and taking sum of the obtained results, combining with the formulas (\ref{e-1}) and (\ref{e-2}), we have
\begin{equation*}
\begin{aligned}
&\theta \hat{\rho}_{h}^{k+1} \|\hat{\textbf{u}}_{h}^{k+1}\|^{2}+\theta \rho_{h}^{k} \|\hat{\textbf{u}}_{h}^{k+1}- \textbf{u}_{h}^{k}\|^{2}+4\Delta t\eta_{1}\|D (\hat{\textbf{u}}_{h}^{k+1})\|^{2} +\frac{1}{\mu}\|\hat{\textbf{B}}_{h}^{k+1}\|^{2}+ \frac{1}{\mu}\|\hat{\textbf{B}}_{h}^{k+1}-\theta \textbf{B}_{h}^{k}\|^{2} \\
&\quad+\frac{2 \Delta t}{\mu^{2}\sigma_{2}} \|\hat{\textbf{B}}_{h}^{k+1}\|_{H_{n}^{1}}^{2}+2\Delta tM_{1}\|\nabla\hat{\omega}_{h}^{k+1}\|^{2} +\gamma\varepsilon\|\nabla\hat{\phi}_{h}^{k+1}\|^{2}+\gamma\varepsilon\|\nabla(\hat{\phi}_{h}^{k+1}-\theta \phi_{h}^{k})\|^{2} \\
&\quad  +\frac{\gamma\theta}{2\varepsilon}\left(\| (\hat{\phi}_{h}^{k+1})^{2}-1 \|^{2} +  \| (\hat{\phi}_{h}^{k+1})^{2}-(\theta \phi_{h}^{k})^{2} \|^{2}  \right)+ \frac{\gamma\theta}{\varepsilon} \left( \|\hat{\phi}_{h}^{k+1} (\hat{\phi}_{h}^{k+1}-\theta \phi_{h}^{k})\|^{2} +\|\hat{\phi}_{h}^{k+1}-\theta \phi_{h}^{k}\|^{2} \right)  \\
&\leq\rho_{h}^{k} \|\textbf{u}_{h}^{k}\|^{2}+\frac{1}{\mu}\| \textbf{B}_{h}^{k}\|^{2}+\gamma\varepsilon\|\nabla\phi_{h}^{k}\|^{2}+\frac{\gamma\theta}{2\varepsilon}\| (\theta\phi_{h}^{k})^{2}-1 \|^{2}+ \frac{2\gamma\theta (1-\theta)}{\varepsilon}(\hat{\phi}_{h}^{k+1}-\theta\phi_{h}^{k}, \phi_{h}^{k}  ),\\&\leq\rho_{h}^{k} \|\textbf{u}_{h}^{k}\|^{2}+\frac{\| \textbf{B}_{h}^{k}\|^{2}}{\mu}+\gamma\varepsilon\|\nabla\phi_{h}^{k}\|^{2}+\frac{\gamma\theta}{2\varepsilon}\| (\theta\phi_{h}^{k})^{2}-1 \|^{2}+ \frac{2\gamma\theta (1-\theta)^{2}}{\varepsilon}\|\phi_{h}^{k}\|^{2}+\frac{\gamma\theta }{2\varepsilon}\|\hat{\phi}_{h}^{k+1}-\theta\phi_{h}^{k}\|^{2}.
\end{aligned}
\end{equation*}
Thinking about $\theta\in (0, 1]$, we obtain
\begin{equation*}
\begin{aligned}
&\theta \hat{\rho}_{h}^{k+1} \|\hat{\textbf{u}}_{h}^{k+1}\|^{2}+\theta \rho_{h}^{k} \|\hat{\textbf{u}}_{h}^{k+1}- \textbf{u}_{h}^{k}\|^{2}+4\Delta t\eta_{1}\|D (\hat{\textbf{u}}_{h}^{k+1})\|^{2} +\frac{1}{\mu}\|\hat{\textbf{B}}_{h}^{k+1}\|^{2}+ \frac{1}{\mu}\|\hat{\textbf{B}}_{h}^{k+1}-\theta \textbf{B}_{h}^{k}\|^{2} \\
&\quad +\frac{2 \Delta t}{\mu^{2}\sigma_{2}} \|\hat{\textbf{B}}_{h}^{k+1}\|_{H_{n}^{1}}^{2}+2\Delta tM_{1}\|\nabla\hat{\omega}_{h}^{k+1}\|^{2} +\gamma\varepsilon\|\nabla\hat{\phi}_{h}^{k+1}\|^{2}+\gamma\varepsilon\|\nabla(\hat{\phi}_{h}^{k+1}-\theta \phi_{h}^{k})\|^{2} \\
&\quad  +\frac{\gamma\theta}{2\varepsilon}\left(\| (\hat{\phi}_{h}^{k+1})^{2}-1 \|^{2} +  \| (\hat{\phi}_{h}^{k+1})^{2}-(\theta \phi_{h}^{k})^{2} \|^{2}  \right)+ \frac{\gamma\theta}{\varepsilon} \left( \|\hat{\phi}_{h}^{k+1} (\hat{\phi}_{h}^{k+1}-\theta \phi_{h}^{k})\|^{2} +\frac{1}{2}\|\hat{\phi}_{h}^{k+1}-\theta \phi_{h}^{k}\|^{2} \right)  \\
&\leq\rho_{h}^{k} \|\textbf{u}_{h}^{k}\|^{2}+\frac{1}{\mu}\| \textbf{B}_{h}^{k}\|^{2}+\gamma\varepsilon\|\nabla\phi_{h}^{k}\|^{2}+\frac{\gamma }{2\varepsilon}\left(\| \phi_{h}^{k}\|^{4}+2 \| \phi_{h}^{k}\|^{2}+|\Omega|\right)\\
&:=M_{c1},
\end{aligned}
\end{equation*}
here $M_{c1}$ is a positive  constant, which  independent of $\theta$ and $(\hat{\phi}_{h}^{k+1}, \hat{\omega}_{h}^{k+1}, \hat{\textbf{u}}_{h}^{k+1}, \hat{p}_{h}^{k+1}, \hat{\textbf{B}}_{h}^{k+1})$.

Then we need to prove the boundedness of $\|\hat{\phi}_{h}^{k+1}\|^{2}$ and $\|\hat{\omega}_{h}^{k+1}\|^{2}$. Taking  $\psi_{h}$=$2\Delta t \hat{\phi}_{h}^{k+1}$ and $\chi_{h}$=$\frac{2\Delta t M(\phi_{h}^{k})\hat{\omega}_{h}^{k+1}}{\gamma\varepsilon}$,  we have
\begin{equation*}
\begin{aligned}
&\|\hat{\phi}_{h}^{k+1}\|^{2}-\theta^{2}\|\phi_{h}^{k}\|^{2}+\|\hat{\phi}_{h}^{k+1}-\theta \phi_{h}^{k}\|^{2}+\frac{2\Delta t}{\gamma\varepsilon}\|\sqrt{M(\phi_{h}^{k})}\hat{\omega}_{h}^{k+1}\|^{2}\\
&=2\Delta t\theta (\phi_{h}^{k} \hat{\textbf{u}}_{h}^{k+1}, \nabla\hat{\phi}_{h}^{k+1})+\frac{2\Delta t\theta  }{ \varepsilon^{2}}( (\hat{\phi}_{h}^{k+1})^{3}-\phi_{h}^{k},  M (\phi_{h}^{k})\hat{\omega}_{h}^{k+1}).
\end{aligned}
\end{equation*}
We introduce the following equality
\begin{equation}
(a^{k+1})^{4}=((a^{k+1})^{2}-1)^{2}+2( a^{k+1}-\theta a^{k} )^{2}+4\theta ( a^{k+1}-\theta a^{k}) a^{k}-1+2\theta^{2} (a^{k})^{2}.
\end{equation}
Thus, we can conclude the following inequality by means of above equality and H\"{o}lder inequality
\begin{equation*}
\begin{aligned}
2\Delta t\theta (\phi_{h}^{k} \hat{\textbf{u}}_{h}^{k+1}, \nabla\hat{\phi}_{h}^{k+1})&\leq 2\Delta t\theta \|\phi_{h}^{k}\|_{L^{3}} \|\hat{\textbf{u}}_{h}^{k+1}\|_{L^{6}} \|\nabla\hat{\phi}_{h}^{k+1}\|,\\
\frac{2\Delta t\theta }{ \varepsilon^{2}}( (\hat{\phi}_{h}^{k+1})^{3}-\phi_{h}^{k},   M (\phi_{h}^{k}) \hat{\omega}_{h}^{k+1})&  \leq \frac{2\Delta t\theta M_{2} }{ \varepsilon^{2}}(\|\hat{\phi}_{h}^{k+1}\|_{L^{4}}^{3}\|\hat{\omega}_{h}^{k+1}\|_{L^{4}}+\|\phi_{h}^{k}\|     \|\hat{\omega}_{h}^{k+1}\|)\\
&\leq \frac{2\Delta t\theta M_{2} }{ \varepsilon^{2}} ( \|\hat{\phi}_{h}^{k+1}\|_{L^{4}}^{6}+ \|\hat{\omega}_{h}^{k+1}\|_{L^{4}}^{2} +\|\phi_{h}^{k}\|^{2} + \|\hat{\omega}_{h}^{k+1}\|^{2})\\
&\leq \frac{C_{24}\Delta t\theta  }{ \varepsilon^{2}} ( \|(\hat{\phi}_{h}^{k+1})^{2}-1 \|^{2}+  \|\hat{\phi}_{h}^{k+1}- \theta \phi_{h}^{k} \|^{2} )^{\frac{3}{2}}\\
&\quad +\frac{C_{25}\Delta t\theta  }{ \varepsilon^{2}} (\theta^{3}\|\phi_{h}^{k} \|^{3}+\|\phi_{h}^{k} \|^{2}+|\Omega|^{\frac{3}{2}})  + \frac{C_{26}\Delta t\theta  }{ \varepsilon^{2}} \| \omega_{h}^{k+1}\|^{2}_{H^{1}},
\end{aligned}
\end{equation*}
here $C_{24}$, $C_{25}$, $C_{26}$ are positive constants. In summary, it can be concluded that
\begin{equation*}
\begin{aligned}
\|\hat{\phi}_{h}^{k+1}\|^{2} +\frac{2\Delta tM_{1} }{\gamma\varepsilon}\|\hat{\omega}_{h}^{k+1}\|^{2}&\leq \|\phi_{h}^{k}\|^{2}+2\Delta t \|\phi_{h}^{k}\|_{L^{3}} \|\hat{\textbf{u}}_{h}^{k+1}\|_{L^{6}} \|\nabla\hat{\phi}_{h}^{k+1}\|+ \frac{C_{27}\Delta t }{ \varepsilon^{2}} \|\nabla \omega_{h}^{k+1}\|^{2}\\
& \quad +\frac{C_{28}\Delta t }{ \varepsilon^{2}} ( \|(\hat{\phi}_{h}^{k+1})^{2}-1 \|^{2}+  \|\hat{\phi}_{h}^{k+1}-\theta \phi_{h}^{k} \|^{2} )^{\frac{3}{2}}\\
&\quad +\frac{C_{29}\Delta t  }{ \varepsilon^{2}} ( \|\phi_{h}^{k} \|^{3}+\|\phi_{h}^{k} \|^{2}+|\Omega|^{\frac{3}{2}}),
\end{aligned}
\end{equation*}
for   $C_{27}$, $C_{28}$, $C_{29}$ are positive constants. Thus, we obtain  $\|\hat{\phi}_{h}^{k+1}\|_{H^{1}}^{2}+\| \hat{\omega}_{h}^{k+1}\|_{H^{1}}^{2}\leq M_{c}$ is established. Based on inf-sup condition (\ref{41}), we can estimate
\begin{equation}
\beta_{0}\|\hat{p}_{h}^{k+1} \|\leq M_{c}.
\end{equation}

Synthesize the above analysis, we can obtain the equation (\ref{boundedness}) is proved. And the $\Pi$ has a fixed point, which is a solution to the scheme (\ref{fully discrete scheme}).
\end{proof}

%
\section{Convergence of the numerical scheme and existence of the weak solution}
This section mainly give the   convergence results of the numerical scheme (\ref{fully discrete scheme}) and existence of  weak solution to the unsteady incompressible CH-MHD with large density ratios model (\ref{TWO-PHASE MHD})-(\ref{2-boundary}).  Firstly, we give  $\{\phi_{h\Delta t}(\cdot, t), \rho_{h\Delta t}(\cdot, t), \textbf{u}_{h\Delta t}(\cdot, t), \textbf{B}_{h\Delta t}(\cdot, t)\}$ are the piece linear interpolation of the fully discrete finite element solution $\{\phi_{h}^{m+1}, \rho_{h}^{m+1}, \textbf{u}_{h}^{m+1}, \textbf{B}_{h}^{m+1} \}$, $m$=0, 1, $\cdots$, N-1, for any $t\in [t^{m}, t^{m+1}]$.

\begin{equation*}
\begin{aligned}
&\phi_{h\Delta t}(\cdot, t)=\frac{t-t^{m}}{\Delta t}\phi_{h}^{m+1}(\cdot)+ \frac{t^{m+1}-t}{\Delta t}\phi_{h}^{m}(\cdot),\quad \rho_{h\Delta t}(\cdot, t)=\frac{t-t^{m}}{\Delta t}\rho_{h}^{m+1}(\cdot)+ \frac{t^{m+1}-t}{\Delta t}\rho_{h}^{m}(\cdot),\\
&\textbf{u}_{h\Delta t}(\cdot, t)=\frac{t-t^{m}}{\Delta t}\textbf{u}_{h}^{m+1}(\cdot)+ \frac{t^{m+1}-t}{\Delta t}\textbf{u}_{h}^{m}(\cdot), \quad \textbf{B}_{h\Delta t}(\cdot, t)=\frac{t-t^{m}}{\Delta t}\textbf{B}_{h}^{m+1}(\cdot)+ \frac{t^{m+1}-t}{\Delta t}\textbf{B}_{h}^{m}(\cdot).
\end{aligned}
\end{equation*}
Another $\{\bar{\phi}_{h\Delta t}(\cdot, t)$, $\bar{\rho}_{h\Delta t}(\cdot, t)$, $\bar{\omega}_{h\Delta t}(\cdot, t)$, $\bar{\textbf{u}}_{h\Delta t}(\cdot, t)$, $\bar{p}_{h\Delta t}(\cdot, t)$,  $\bar{\textbf{B}}_{h\Delta t}(\cdot, t)\}$ and $\{\bar{\bar{\phi}}_{h\Delta t}(\cdot, t)$, $\bar{\bar{\rho}}_{h\Delta t}(\cdot, t)$,   $\bar{\bar{\textbf{u}}}_{h\Delta t}(\cdot, t)$,   $\bar{\bar{\textbf{B}}}_{h\Delta t}(\cdot, t)\}$ are the piecewise constant extensions of $\{\phi_{h}^{m+1}$, $\rho_{h}^{m+1}$, $\omega_{h}^{m+1}$, $\textbf{u}_{h}^{m+1}$, $p_{h}^{m+1}$, $\textbf{B}_{h}^{m+1} \}$ and $\{\phi_{h}^{m}$, $\rho_{h}^{m}$,  $\textbf{u}_{h}^{m}$,  $\textbf{B}_{h}^{m} \}$, $m=0, 1,\cdots, N-1$, and the $t\in (t^{m}, t^{m+1}]$,
\begin{equation*}
\begin{aligned}
&\bar{\phi}_{h\Delta t}(\cdot, t)= \phi_{h}^{m+1}(\cdot),\quad \bar{\rho}_{h\Delta t}(\cdot, t)= \rho_{h}^{m+1}(\cdot),\quad \bar{\omega}_{h\Delta t}(\cdot, t)= \omega_{h}^{m+1}(\cdot),\\
& \bar{\textbf{u}}_{h\Delta t}(\cdot, t)= \textbf{u}_{h}^{m+1}(\cdot),\quad \bar{p}_{h\Delta t}(\cdot, t)= p_{h}^{m+1}(\cdot),\quad \bar{\textbf{B}}_{h\Delta t}(\cdot, t)= \textbf{B}_{h}^{m+1}(\cdot),\\
&\bar{\bar{\phi}}_{h\Delta t}(\cdot, t)= \phi_{h}^{m}(\cdot),\quad \bar{\bar{\rho}}_{h\Delta t}(\cdot, t)= \rho_{h}^{m}(\cdot),\quad  \bar{\bar{\textbf{u}}}_{h\Delta t}(\cdot, t)= \textbf{u}_{h}^{m}(\cdot),\quad \bar{\bar{\textbf{B}}}_{h\Delta t}(\cdot, t)= \textbf{B}_{h}^{m}(\cdot).\\
\end{aligned}
\end{equation*}

For briefness, the convergence sequences are denoted by the same symbols. In addition, weak convergence and weak $\ast$ convergence are collectively referred to as weak convergence.
\begin{lemma}\label{Weak}
(\textbf{Weak Convergence}) For the sequences $\{\bar{\bar{\phi}}_{h\Delta t}$, $\bar{\bar{\rho}}_{h\Delta t}$,   $\bar{\bar{\textbf{u}}}_{h\Delta t}$,   $\bar{\bar{\textbf{B}}}_{h\Delta t}\}$, $\{\bar{\phi}_{h\Delta t}$, $\bar{\rho}_{h\Delta t}$, $\bar{\omega}_{h\Delta t}$, $\bar{\textbf{u}}_{h\Delta t}$, $\bar{p}_{h\Delta t}$,  $\bar{\textbf{B}}_{h\Delta t}\}$ and $\{\phi_{h\Delta t}$, $\rho_{h\Delta t}$, $\textbf{u}_{h\Delta t}$, $\textbf{B}_{h\Delta t}\}$, there exist convergence subsequences which are satisfying the following weak convergence
\begin{subequations}
\begin{align}
&\bar{\bar{\phi}}_{h\Delta t},   \bar{\phi}_{h\Delta t}, \phi_{h\Delta t}\rightharpoonup \ast\phi \quad \rm in\, L^{\infty}(0, T; H^{1}),\label{5-a}\\
&(\phi_{h\Delta t})_{t}\rightharpoonup  \phi_{t}\quad  \rm in\, L^{2}(0, T; (H^{1})'), \label{5-b}\\
&\bar{\bar{\rho}}_{h\Delta t},   \bar{\rho}_{h\Delta t}, \rho_{h\Delta t}\rightharpoonup \ast\rho \quad  \rm in\, L^{\infty}(0, T; H^{1}),\\
&(\rho_{h\Delta t})_{t}\rightharpoonup  \rho_{t}\quad  \rm in\, L^{2}(0, T; (H^{1})'),\\
&\bar{\omega}_{h\Delta t}\rightharpoonup \omega\quad  \rm in\, L^{2}(0, T; H^{1}),\\
&\bar{\bar{\textbf{u}}}_{h\Delta t},   \bar{\textbf{u}}_{h\Delta t}, \textbf{u}_{h\Delta t}\rightharpoonup \ast\textbf{u} \quad  \rm in\, L^{\infty}(0, T; L^{2}),\\
&\bar{\bar{\textbf{u}}}_{h\Delta t},   \bar{\textbf{u}}_{h\Delta t}, \textbf{u}_{h\Delta t}\rightharpoonup \textbf{u} \quad  \rm in\, L^{2}(0, T; H_{0}^{1}),\label{u1}\\
&(\rho_{h\Delta t}\textbf{u}_{h\Delta t})_{t}\rightharpoonup \rho\textbf{u} \quad  \rm in\, L^{\frac{12}{6+d}}(0, T;  (H_{0}^{1})'),    \label{rhou}\\
&\bar{p}_{h\Delta t}\rightharpoonup p\quad  \rm in\, L^{\frac{12}{6+d}}(0, T; L_{0}^{2}),\\
&\bar{\bar{\textbf{B}}}_{h\Delta t},   \bar{\textbf{B}}_{h\Delta t}, \textbf{B}_{h\Delta t}\rightharpoonup \ast\textbf{B} \quad  \rm in\, L^{\infty}(0, T; L^{2}),\\
&\bar{\bar{\textbf{B}}}_{h\Delta t},   \bar{\textbf{B}}_{h\Delta t}, \textbf{B}_{h\Delta t}\rightharpoonup \textbf{B} \quad  \rm in\, L^{2}(0, T; H_{n}^{1}),\label{B2}\\
&(\textbf{B}_{h\Delta t})_{t}\rightharpoonup \textbf{B}_{t} \quad  \rm in\, L^{\frac{4}{d}}(0, T; (H_{n}^{1})'), \label{B}\\
\end{align}
\end{subequations}
when  $h, \Delta t\rightarrow 0$ and  $\rightharpoonup\ast$ means weak $\ast$ convergence.
\end{lemma}

\begin{proof}
We   present the proof of   equation (\ref{5-a}), and the others are similar with it. Considering above $\{\bar{\bar{\phi}}_{h\Delta t}\}$, $\{ \bar{\phi}_{h\Delta t}\}$, $\{ \phi_{h\Delta t}\}$ are bounded sequences in $L^{\infty}(0, T; H^{1})$, and  $\{\bar{\bar{\phi}}_{h\Delta t}\}$, $\{ \bar{\phi}_{h\Delta t}\}$, $\{ \phi_{h\Delta t}\}$ weakly $\ast$ converge to $\phi_{2}$, $\phi_{1}$, $\phi$ in $L^{\infty}(0, T; H^{1})$ respectively. Thus, we have
\begin{equation}\label{limit}
\lim_{h, \Delta t\rightarrow 0}\int_{0}^{T}(\phi_{h\Delta t}-\bar{\phi}_{h\Delta t}, \psi)\mbox{d}t=\int_{0}^{T}(\phi - \phi_{1}, \psi)\mbox{d}t, \quad \forall\psi \in L^{1}(0, T; (H^{1})').
\end{equation}
According to   $L^{2}\hookrightarrow (H^{1})'$ with continuous injection, $L^{\infty}(0, T; H^{1})\hookrightarrow L^{1}(0, T; (H^{1})')$, and equation (\ref{theorem-faichaH1}), we obtain
\begin{equation*}
\begin{aligned}
\int_{0}^{T}(\phi_{h\Delta t}-\bar{\phi}_{h\Delta t}, \phi -\phi_{1})\mbox{d}t&\leq \int_{0}^{T}\|\phi_{h\Delta t}-\bar{\phi}_{h\Delta t}\|_{H^{1}}\|\phi -\phi_{1}\|_{(H^{1})'}\mbox{d}t\\
&\leq C_{30}\int_{0}^{T}\|\phi_{h\Delta t}-\bar{\phi}_{h\Delta t}\|_{H^{1}}\|\phi -\phi_{1}\|_{L^{2}}\mbox{d}t\\
&\leq C_{31}\|\phi - \phi_{1}\|_{L^{\infty}(L^{2})}\sum_{n=1}^{N}\int_{t_{n}}^{t_{n+1}}\frac{t_{n+1}-t}{\Delta t}\|\phi_{h}^{n+1}-\phi_{h}^{n}\|_{H^{1}}\mbox{d}t\\
&\leq C_{32}\Delta t^{\frac{1}{2}}\|\phi - \phi_{1}\|_{L^{\infty}(L^{2})}(\sum_{n=1}^{N}\|\phi_{h}^{n+1}-\phi_{h}^{n}\|_{H^{1}}^{2})^{\frac{1}{2}}\\
& \stackrel{\Delta t\rightarrow 0}{\longrightarrow}0,
\end{aligned}
\end{equation*}
for  $C_{30}$, $C_{31}$ and $C_{32}$ are positive constants. Letting $\psi=\phi - \phi_{1}$ in equation (\ref{limit}),  we obtain that $\phi$ and $\phi_{1}$ are equal. Similar to the above derivation,  it can  conclude that $\phi$ and $\phi_{2}$ also are equal. The proof of the equation (\ref{rhou}), see for instance \cite{2011Existence}. I have not give  further details here.
\end{proof}

\begin{lemma}\label{Strong}
(\textbf{Strong Convergence}) For the sequences $\{\bar{\bar{\phi}}_{h\Delta t}$, $\bar{\bar{\rho}}_{h\Delta t}$,   $\bar{\bar{\textbf{u}}}_{h\Delta t}$,   $\bar{\bar{\textbf{B}}}_{h\Delta t}\}$, $\{\bar{\phi}_{h\Delta t}$, $\bar{\rho}_{h\Delta t}$, $\bar{\textbf{u}}_{h\Delta t}$, $\bar{\textbf{B}}_{h\Delta t}\}$ and $\{\phi_{h\Delta t}$, $\rho_{h\Delta t}$, $\textbf{u}_{h\Delta t}$, $\textbf{B}_{h\Delta t}\}$, there exist convergence subsequences which are satisfying the following strong convergence
\begin{subequations}
\begin{align}
&\phi_{h\Delta t} \rightarrow \phi \quad  \rm in\, C(0, T; L^{p}),\label{strong1}\\
&\bar{\bar{\phi}}_{h\Delta t},   \bar{\phi}_{h\Delta t} \rightarrow\phi \quad  \rm in\, L^{\infty}(0, T; L^{p}),\label{strongphi}\\
&\rho_{h\Delta t} \rightarrow \rho \quad  \rm in\, C(0, T; L^{p}),\label{strongrho2}\\
&\bar{\bar{\rho}}_{h\Delta t},   \bar{\rho}_{h\Delta t} \rightarrow\rho \quad  \rm in\, L^{\infty}(0, T; L^{p}), \label{rhog}\\
&\bar{\bar{\textbf{u}}}_{h\Delta t},   \bar{\textbf{u}}_{h\Delta t}, \textbf{u}_{h\Delta t}\rightarrow \textbf{u} \quad \rm in\, L^{2}(0, T; L^{p}),\label{t-u}\\
&\bar{\bar{\textbf{B}}}_{h\Delta t},   \bar{\textbf{B}}_{h\Delta t}, \textbf{B}_{h\Delta t}\rightarrow\textbf{B} \quad \rm in\, L^{2}(0, T; L^{p}), \label{t-B}
\end{align}
\end{subequations}
when  $h, \Delta t\rightarrow 0$ and   $p\in [1, \frac{2d}{d-2})$.
\end{lemma}

\begin{proof}
Inspired by \cite{1987Compact} and combined with the definition of $\rho$, we know the equations (\ref{strong1}) and (\ref{strongrho2}) are valid. Then we show the proof procedure for equation (\ref{strongphi}). According to equation (\ref{strong1}),   $\{\phi_{h\Delta t} \}$ is compact in $C(0, T; L^{p})$, namely, for $\forall \kappa> 0$, there is $\delta>0$, s.t. for all $h, \Delta t>0$, we have $\|\phi_{h\Delta t}(t1)-\phi_{h\Delta t}(t2)\|_{L^{p}}\leq \kappa$, for  $t1, t2\in [0, T]$ and $|t1-t2|\leq\delta$. Thus for $\forall\kappa>0$, there is $\delta>0$ such that the following estimate hold true
\begin{equation*}
\|\bar{\bar{\phi}}_{h\Delta t}- \phi_{h\Delta t}\|_{L^{\infty}(L^{p})}=\| \bar{\phi}_{h\Delta t}- \phi_{h\Delta t}\|_{L^{\infty}(L^{p})}=ess\sup_{1\leq m\leq N}\|  \phi_{h }^{m+1}- \phi_{h }^{m}\|_{ L^{p} }\leq\kappa,
\end{equation*}
where  $\Delta t\leq\delta$. Obviously, the equation (\ref{strongphi}) is effective. And then, we give $\{\bar{\textbf{u}}_{h\Delta t}\}$ strong converges to $ \textbf{u} $ in $L^{2}(0, T; L^{p})$. For any $p\in (1, \frac{2d}{d-2})$, taking $p_{1}\in (p, \frac{2d}{d-2})$ and by feat of interpolation inequality (\ref{inter}), Young inequality and \textbf{Theorem \ref{t-fully}}, we derive
\begin{equation*}
\begin{aligned}
\|\bar{\textbf{u}}_{h\Delta t}- \textbf{u}_{h\Delta t}\|^{2}_{L^{2}(L^{p})}&=\sum_{n=1}^{N}\int_{t_{n}}^{t_{n+1}}(\frac{t_{n+1}-t}{\Delta t})^{2}\|\textbf{u}_{h}^{n+1}-\textbf{u}_{h}^{n}\|^{2}_{L^{p}}\mbox{d}t\\
&\leq C_{33}\Delta t \sum_{n=1}^{N}( \|\textbf{u}_{h}^{n+1}-\textbf{u}_{h}^{n}\|_{L^{1}}^{\theta}  \|\textbf{u}_{h}^{n+1}-\textbf{u}_{h}^{n}\|_{L^{p_{1}}}^{1-\theta} )^{2}\\
&\leq C_{34}\Delta t^{\theta} (\sum_{n=1}^{N}  \|\textbf{u}_{h}^{n+1}-\textbf{u}_{h}^{n}\|_{L^{2}}^{2} )^{\theta }( \sum_{n=1}^{N} \Delta t\|\nabla (\textbf{u}_{h}^{n+1}-\textbf{u}_{h}^{n})\|_{L^{2}}^{2}   )^{1-\theta}\\
&  \stackrel{\Delta t\rightarrow 0}{\longrightarrow}0,
\end{aligned}
\end{equation*}
where $C_{33}$ and $C_{34}$ are positive constants and  $\theta=\frac{p_{1}-p}{p(p_{1}-1)}$. Similar with above proof, the $\{\bar{\bar{\textbf{u}}}_{h\Delta t}\}$ strong converges to $ \textbf{u} $ in $L^{2}(0, T; L^{p})$. Therefore the equation (\ref{t-u}) is holding. Similar to (\ref{t-u}), the equation (\ref{t-B}) is valid.
\end{proof}

\begin{The}
Hypothesis the inf-sup condition and assumption (\ref{equation4}) be effective and initial data $\phi_{0},  \textbf{u}_{0}, \textbf{B}_{0}$ satisfy the $E(\phi_{0}, \textbf{u}_{0}, \textbf{B}_{0} )<\infty$. There exist a subsequence of $\{(\phi_{h\Delta t}$, $\bar{\rho}_{h\Delta t}$ $\bar{\omega}_{h\Delta t}$,    $\bar{\textbf{u}}_{h\Delta t}$, $\bar{p}_{h\Delta t}$, $\textbf{B}_{h\Delta t})\}$,  which have an accumulation point $(\phi, \rho, \omega, \textbf{u}, p, \textbf{B})$. And the $(\phi,  \omega, \textbf{u}, p, \textbf{B})$ is a weak solution to the model (\ref{TWO-PHASE MHD})-(\ref{2-boundary}).
\end{The}

\begin{proof}
For  $\forall (\psi,  \chi, \textbf{v}, q, \textbf{C})\in C^{\infty}(\bar{\Omega})\times C^{\infty}(\bar{\Omega})\times \textbf{C}^{\infty}_{c}(\Omega) \times (C^{\infty}_{c}(\Omega) \cap L_{0}^{2} )\times ( \textbf{C}^{\infty}(\bar{\Omega})\cap \textbf{H}_{n}^{1} )$, we choose   $(\psi_{h},  \chi_{h}, \textbf{v}_{h}, q_{h}, \textbf{C}_{h})$=$(P_{\phi h}\psi,  P_{\phi h}\chi, P_{uh}\textbf{v}, P_{ph}q, P_{Bh}\textbf{C})\in  Y_{h} \times Y_{h}\times \textbf{X}_{h}\times M_{h}\times \textbf{W}_{h} $ such that the following convergence hold. And  $ P_{ph}$ is $L^{2}$ orthogonal projection operator from $L^{2}$ to $M_{h}$.
\begin{equation*}
\begin{aligned}
&\psi_{h}\stackrel{h\rightarrow 0}{\longrightarrow}\psi \quad \rm in\, H^{1},  \quad  \chi_{h}\stackrel{h\rightarrow 0}{\longrightarrow}\chi \quad \rm in\, H^{1},\\
&\textbf{v}_{h}\stackrel{h\rightarrow 0}{\longrightarrow}\textbf{v} \quad \rm in\, \textbf{H}_{0}^{1},\quad q_{h}\stackrel{h\rightarrow 0}{\longrightarrow}q \quad \rm in\, L^{2},\quad \textbf{C}_{h}\stackrel{h\rightarrow 0}{\longrightarrow}\textbf{C} \quad \rm in\, \textbf{H}^{1}_{n}.
\end{aligned}
\end{equation*}
The equations (\ref{fully discrete scheme}) multiplying by time function $\zeta (t)\in C^{\infty}([0, T])$ and integrating  from 0 to T, we get
\begin{subequations}\label{exist equation}
\begin{align}
&\int_{0}^{T}\{( (\bar{\rho}_{h\Delta t}  \bar{\textbf{u}}_{h\Delta t} )_{t},\textbf{v}_{h}) +(\nabla\cdot (\bar{\bar{\rho}}_{h\Delta t}   \bar{\textbf{u}}_{h\Delta t}  \otimes\bar{\textbf{u}}_{h\Delta t}), \textbf{v}_{h})-(\nabla\cdot(\bar{\textbf{u}}_{h\Delta t} \otimes (\frac{\rho_{2}-\rho_{1}}{2} M(\bar{\bar{\phi}}_{h\Delta t} )\nabla\bar{\omega}_{h\Delta t} )), \textbf{v}_{h}) -\nonumber \\
&(\nabla\cdot \textbf{v}_{h},  \bar{p}_{h\Delta t} )   +2 (\eta (\bar{\bar{\phi}}_{h\Delta t})D(\bar{\textbf{u}}_{h\Delta t} ), D( \textbf{v}_{h})   )+\frac{1}{\mu}(\nabla\times \bar{\textbf{B}}_{h\Delta t} , \textbf{v}_{h}\times \bar{\bar{\textbf{B}}}_{h\Delta t} )+(\bar{\bar{\phi}}_{h\Delta t}   \nabla\bar{\omega}_{h\Delta t}, \textbf{v}_{h})\}\zeta  (t)\mbox{d}t=0, \label{3u}\\
&\int_{0}^{T} (\nabla\cdot \bar{\textbf{u}}_{h\Delta t}, q_{h}) \zeta  (t)\mbox{d}t =0,\\
& \int_{0}^{T}\{ ( (\textbf{B}_{h\Delta t})_{t}, \textbf{C}_{h}) + \frac{1}{\mu} (\frac{1}{\sigma (\bar{\bar{\phi}}_{h\Delta t} )} \nabla\times \bar{\textbf{B}}_{h\Delta t}, \nabla\times \textbf{C}_{h})+ \frac{1}{\mu}  (\frac{1}{\sigma (\bar{\bar{\phi}}_{h\Delta t} )}\nabla\cdot\bar{\textbf{B}}_{h\Delta t}, \nabla\cdot \textbf{C}_{h}) \nonumber \\
& \qquad  -(\bar{\textbf{u}}_{h\Delta t} \times \bar{\bar{\textbf{B}}}_{h\Delta t}, \nabla\times \textbf{C}_{h})\}\zeta  (t)\mbox{d}t=0, \label{Bb-1}\\
&\int_{0}^{T}\{( (\phi_{h\Delta t})_{t}, \psi_{h}) -(\bar{\bar{\phi}}_{h\Delta t} \bar{\textbf{u}}_{h\Delta t}, \nabla \psi_{h})+ (M(\bar{\bar{\phi}}_{h\Delta t})\nabla\bar{\omega}_{h\Delta t}, \nabla\psi_{h}) \}\zeta  (t)\mbox{d}t=0, \label{eqn111}\\
&\int_{0}^{T} (\bar{\omega}_{h\Delta t}, \chi_{h}) \zeta  (t)dt=\int_{0}^{T}\{\gamma\varepsilon(\nabla\bar{\phi}_{h\Delta t}, \nabla\chi_{h})+ \frac{\gamma}{\varepsilon}(\bar{f}_{h\Delta t}, \chi_{h})\}\zeta  (t)\mbox{d}t,
\end{align}
\end{subequations}
for $\bar{f}_{h\Delta t}= \bar{\phi}_{h\Delta t}^{3}-\bar{\bar{\phi}}_{h\Delta t}$. Then we need to analyze   convergence of each term of the above equations. Combining the equations (\ref{rhou}), (\ref{5-b}) and (\ref{B}), we have
\begin{equation*}
\begin{aligned}
&\int_{0}^{T} ( (\bar{\rho}_{h\Delta t}  \bar{\textbf{u}}_{h\Delta t} )_{t}, \textbf{v}_{h}) \zeta (t)\mbox{d}t \stackrel{h, \Delta t\rightarrow 0}{\longrightarrow}\int_{0}^{T} ( (\rho   \textbf{u} )_{t}, \textbf{v }) \zeta (t)\mbox{d}t,\\
&\int_{0}^{T} ( (\phi_{h\Delta t})_{t}, \psi_{h}) \zeta  (t)\mbox{d}t \stackrel{h, \Delta t\rightarrow 0}{\longrightarrow} \int_{0}^{T} (  \phi _{t}, \psi ) \zeta  (t)\mbox{d}t,\\
&\int_{0}^{T} ( (\textbf{B}_{h\Delta t})_{t}, \textbf{C}_{h}) \zeta  (t)\mbox{d}t \stackrel{h, \Delta t\rightarrow 0}{\longrightarrow} \int_{0}^{T} (\textbf{B}_{t}, \textbf{C}) \zeta  (t)\mbox{d}t.
\end{aligned}
\end{equation*}
Similar, we can conclude
\begin{equation*}
\int_{0}^{T} (\bar{\omega}_{h\Delta t}, \chi_{h}) \zeta  (t)\mbox{d}t-\int_{0}^{T} ( \omega, \chi ) \zeta  (t)\mbox{d}t\stackrel{h, \Delta t\rightarrow 0}{\longrightarrow}0.
\end{equation*}
Through the definition of weak convergence, it hold
\begin{equation*}
\begin{aligned}
&\int_{0}^{T}(\nabla\cdot\bar{\textbf{u}}_{h\Delta t}, q_{h} ) \zeta  (t)\mbox{d}t \stackrel{h, \Delta t\rightarrow 0}{\longrightarrow} \int_{0}^{T}(\nabla\cdot \textbf{u}, q ) \zeta  (t)\mbox{d}t,\\
&\int_{0}^{T}(\nabla\cdot \textbf{v}_{h\Delta t}, \bar{p}_{h\Delta t} ) \zeta  (t)\mbox{d}t\stackrel{h, \Delta t\rightarrow 0}{\longrightarrow} \int_{0}^{T}(\nabla\cdot \textbf{v}, p ) \zeta  (t)\mbox{d}t.\\
\end{aligned}
\end{equation*}


By the equations (\ref{rhog}), (\ref{t-u}), we have following estimate
\begin{equation*}
\begin{aligned}
&\int_{0}^{T} (\nabla\cdot (\bar{\bar{\rho}}_{h\Delta t}   \bar{\textbf{u}}_{h\Delta t}  \otimes\bar{\textbf{u}}_{h\Delta t}), \textbf{v}_{h})\zeta (t)\mbox{d}t- \int_{0}^{T} (\nabla\cdot ( \rho    \textbf{u} \otimes \textbf{u}), \textbf{v} )\zeta (t)\mbox{d}t\\
&\leq \left|\int_{0}^{T} ( (\bar{\bar{\rho}}_{h\Delta t}   \bar{\textbf{u}}_{h\Delta t}  \otimes\bar{\textbf{u}}_{h\Delta t}), \nabla\textbf{v}_{h})\zeta (t)\mbox{d}t- \int_{0}^{T} ( ( \rho    \textbf{u} \otimes \textbf{u}), \nabla\textbf{v} )\zeta (t)\mbox{d}t \right|\\
&\leq \|\bar{\bar{\rho}}_{h\Delta t}-\rho\|_{L^{\infty}(L^{4})}\|\bar{\textbf{u}}_{h\Delta t} \|_{L^{4}(L^{4})}^{2}\|\nabla \textbf{v}_{h}\zeta (t)\|_{L^{2}(L^{4})}\\
&\quad +\|\rho\|_{L^{\infty}(L^{4})}\|\bar{\textbf{u}}_{h\Delta t} \|_{L^{2}(L^{4})}\|\bar{\textbf{u}}_{h\Delta t}-\textbf{u} \|_{L^{2}(L^{4})}\|\nabla \textbf{v}_{h}\zeta (t)\|_{L^{\infty}(L^{4})}\\
&\quad +\|\rho\|_{L^{\infty}(L^{4})}\|\bar{\textbf{u}}_{h\Delta t}-\textbf{u} \|_{L^{2}(L^{4})}\|\textbf{u}\|_{L^{\infty}(L^{4})}\|\nabla \textbf{v}_{h}\zeta (t)\|_{L^{2}(L^{4})}\\
&\quad +\|\rho\|_{L^{2}(L^{4})}\| \textbf{u} \|_{L^{\infty}(L^{4})}^{2}\|\nabla (\textbf{v}_{h}-\textbf{v})\zeta (t)\|_{L^{2}(L^{4})}\\
&\stackrel{h, \Delta t\rightarrow 0}{\longrightarrow} 0.
\end{aligned}
\end{equation*}
Similar with above estimate, we obtain
\begin{equation*}
\begin{aligned}
&\int_{0}^{T} (\nabla\cdot(\bar{\textbf{u}}_{h\Delta t} \otimes (\frac{\rho_{2}-\rho_{1}}{2} M(\bar{\bar{\phi}}_{h\Delta t} )\nabla\bar{\omega}_{h\Delta t} )), \textbf{v}_{h})-
\int_{0}^{T} (\nabla\cdot( \textbf{u}  \otimes (\frac{\rho_{2}-\rho_{1}}{2} M( \phi )\nabla \omega  )), \textbf{v} )\stackrel{h, \Delta t\rightarrow 0}{\longrightarrow} 0.
\end{aligned}
\end{equation*}

Considering the definition of $\rho( \phi )$, $\eta ( \phi )$, $M ( \phi )$ and $ \sigma( \phi )$, they  are Lipschitz continuous functions of $\phi$ an satisfy equation (\ref{canshu}). We obtain
\begin{equation*}
\begin{aligned}
&|\eta (\bar{\bar{\phi}}_{h\Delta t})-\eta ( \phi_{h\Delta t})|\leq C_{36} |\bar{\bar{\phi}}_{h\Delta t}- \phi_{h\Delta t}|, \\
&|\sigma (\bar{\bar{\phi}}_{h\Delta t})-\sigma ( \phi_{h\Delta t})|\leq C_{37} |\bar{\bar{\phi}}_{h\Delta t}- \phi_{h\Delta t}|,\\
&|M (\bar{\bar{\phi}}_{h\Delta t})-M ( \phi_{h\Delta t})|\leq C_{38} |\bar{\bar{\phi}}_{h\Delta t}- \phi_{h\Delta t}|,
\end{aligned}
\end{equation*}
 where   $C_{36}$, $C_{37}$ and $C_{38}$ are positive constants. By the equations (\ref{theorem-h1}), (\ref{u1}) and (\ref{strongphi}), we can obtain the elliptic term as
\begin{equation*}
\begin{aligned}
&\int_{0}^{T}(\eta (\bar{\bar{\phi}}_{h\Delta t})D(\bar{\textbf{u}}_{h\Delta t} ), D( \textbf{v}_{h})  )\zeta  (t)\mbox{d}t- \int_{0}^{T}(\eta ( \phi)D(\textbf{\textbf{u}}), D( \textbf{v} )  )\zeta  (t)\mbox{d}t\\
&\leq \|\eta (\bar{\bar{\phi}}_{h\Delta t})-\eta ( \phi)\|_{L^{4}(L^{4})}\|D(\bar{\textbf{u}}_{h\Delta t} )\|_{L^{2}(L^{2})}\|D(\textbf{v}_{h}  )  \zeta (t)\|_{L^{4}(L^{4})}\\
&\quad +\eta_{2}\|D(\bar{\textbf{u}}_{h\Delta t} )\|_{L^{2}(L^{2})}\|(D(\textbf{v}_{h}  )-D(\textbf{v}   ))  \zeta (t)\|_{L^{2}(L^{2})}\\
&\quad+\left |\int_{0}^{T}(\eta (\phi)( D(\bar{\textbf{u}}_{h\Delta t})-D( \textbf{u} )), D(\textbf{v}   ) )\zeta (t)\mbox{d}t \right|\\
&\stackrel{h, \Delta t\rightarrow 0}{\longrightarrow}0.
\end{aligned}
\end{equation*}
Thus, we can obtain similarly
\begin{equation*}
\begin{aligned}
&\int_{0}^{T} (M(\bar{\bar{\phi}}_{h\Delta t})\nabla\bar{\omega}_{h\Delta t}, \nabla\psi_{h}) \zeta (t)\mbox{d}t- \int_{0}^{T} (M( \phi )\nabla \omega, \nabla\psi ) \zeta (t)\mbox{d}t\stackrel{h, \Delta t\rightarrow 0}{\longrightarrow}0,\\
&\int_{0}^{T} (\nabla\bar{\phi}_{h\Delta t}, \nabla\chi_{h}) \zeta (t)\mbox{d}t- \int_{0}^{T} (\nabla \phi, \nabla\chi ) \zeta(t)\mbox{d}t\stackrel{h, \Delta t\rightarrow 0}{\longrightarrow}0.
\end{aligned}
\end{equation*}

Similar to the derivation above, combining the weak convergence and equations (\ref{theorem-h1}), (\ref{B2}), (\ref{strongphi}), we have
\begin{equation*}
\begin{aligned}
&\int_{0}^{T}\left\{ (  \frac{1}{\mu\sigma (\bar{\bar{\phi}}_{h\Delta t} )} \nabla\times \bar{\textbf{B}}_{h\Delta t}, \nabla\times \textbf{C}_{h})+  (\frac{1}{\mu\sigma (\bar{\bar{\phi}}_{h\Delta t} )} \nabla\cdot\bar{\textbf{B}}_{h\Delta t}, \nabla\cdot \textbf{C}_{h})\right\}\zeta(t)\mbox{d}t\\
&\stackrel{h, \Delta t\rightarrow 0}{\longrightarrow}  \int_{0}^{T} \left\{ (  \frac{1}{\mu\sigma (\phi )} \nabla\times  \textbf{B}, \nabla\times \textbf{C} )+   (\frac{1}{\mu\sigma ( \phi )}\nabla\cdot \textbf{B}, \nabla\cdot \textbf{C} )\right\}\zeta(t)\mbox{d}t.
\end{aligned}
\end{equation*}
Then, from equations  (\ref{theorem-h1}), (\ref{B2}), (\ref{t-B}), we derive
\begin{equation*}
\begin{aligned}
&\int_{0}^{T}(\bar{\bar{\textbf{B}}}_{h\Delta t}\times \nabla\times \bar{\textbf{B}}_{h\Delta t} , \textbf{v}_{h} )\zeta(t)\mbox{d}t-\int_{0}^{T}(\textbf{B} \times \nabla\times  \textbf{B}, \textbf{v})\zeta(t)\mbox{d}t\\
&\leq \|\bar{\bar{\textbf{B}}}_{h\Delta t}- \textbf{B}\|_{L^{2}(L^{4})}\|\nabla\times \bar{\textbf{B}}_{h\Delta t} \|_{L^{2}(L^{2})}\|\textbf{v}_{h}\zeta(t)\|_{L^{\infty}(L^{4})}\\
&\quad +\|\textbf{B}\|_{L^{2}(L^{4})}\|\nabla\times \bar{\textbf{B}}_{h\Delta t} \|_{L^{2}(L^{2})}\|(\textbf{v}_{h}-\textbf{v})\zeta(t)\|_{L^{\infty}(L^{4})}\\
&\quad+ \left|\int_{0}^{T}(\textbf{B}\times\nabla\times (\bar{\textbf{B}}_{h\Delta t}-\textbf{B}), \textbf{v})\zeta(t)\mbox{d}t\right|\\
&\stackrel{h, \Delta t\rightarrow 0}{\longrightarrow} 0.
\end{aligned}
\end{equation*}
Similar with above process, we have
\begin{equation*}
\int_{0}^{T}(\bar{\textbf{u}}_{h\Delta t} \times \bar{\bar{\textbf{B}}}_{h\Delta t}, \nabla\times \textbf{C}_{h}) \zeta (t)\mbox{d}t-\int_{0}^{T}( \textbf{u} \times  \textbf{B}, \nabla\times \textbf{C} ) \zeta (t)\mbox{d}t\stackrel{h, \Delta t\rightarrow 0}{\longrightarrow} 0.
\end{equation*}

Considering equations (\ref{theorem-h1}), (\ref{strongphi}) and (\ref{t-u}), there is
\begin{equation*}
\begin{aligned}
&\int_{0}^{T}(\bar{\bar{\phi}}_{h\Delta t} \bar{\textbf{u}}_{h\Delta t}, \nabla \psi_{h})\zeta(t)\mbox{d}t-\int_{0}^{T}( \phi \textbf{u}, \nabla \psi )\zeta(t)\mbox{d}t\\
&\leq\left(\|\bar{\bar{\phi}}_{h\Delta t} - \phi  \|_{L^{\infty}(L^{4})} \|\bar{\textbf{u}}_{h\Delta t}\|_{L^{2}(L^{4})}  +\|\phi\|_{L^{\infty}(L^{4})}\|\bar{\textbf{u}}_{h\Delta t}-\textbf{u}\|_{L^{2}(L^{4})} \right)\|\nabla \psi_{h}\zeta(t)\|_{L^{2}(L^{2})}\\
&\quad +\|\phi\|_{L^{\infty}(L^{4})} \|\textbf{u}\|_{L^{2}(L^{4})}\|\nabla (\psi_{h}-\psi)\zeta(t)\|_{L^{2}(L^{2})}\\
&\stackrel{h, \Delta t\rightarrow 0}{\longrightarrow} 0.
\end{aligned}
\end{equation*}
Similarly, we have
\begin{equation*}
\int_{0}^{T}(\bar{\bar{\phi}}_{h\Delta t}  \nabla\bar{\omega}_{h\Delta t}, \textbf{v}_{h}) \zeta (t)\mbox{d}t-\int_{0}^{T}( \phi \nabla \omega, \textbf{v} ) \zeta(t)\mbox{d}t\stackrel{h, \Delta t\rightarrow 0}{\longrightarrow} 0.
\end{equation*}
For the trilinear terms, we have the following estimate, for  $C_{39}$ is positive constant
\begin{equation*}
\begin{aligned}
&\int_{0}^{T}(\bar{f}_{h\Delta t}, \chi_{h})\zeta(t)\mbox{d}t- \int_{0}^{T}(f(\phi), \chi)\zeta(t)\mbox{d}t\\
&\leq\left(\|\bar{\phi}_{h\Delta t}-\phi \|_{L^{\infty}(L^{4})}\|\bar{\phi}_{h\Delta t}\|_{L^{4}(L^{4})}^{2} +\|\bar{\phi}_{h\Delta t}-\phi \|_{L^{\infty}(L^{4})}\|\phi\|_{L^{4}(L^{4})}^{2} +\|\bar{\bar{\phi}}_{h\Delta t}-\phi \|_{L^{2}(L^{4})}\right) \|\chi_{h}\zeta(t)\|_{L^{2}(L^{4})}\\
&\quad +C_{39}\left(\|\phi\|_{L^{\infty}(L^{4})}^{3}+\|\phi \|_{L^{2}(L^{2})}  \right) \|(\chi_{h}-\chi)\zeta(t)\|_{L^{2}(L^{4})}\\
&\stackrel{h, \Delta t\rightarrow 0}{\longrightarrow} 0.
\end{aligned}
\end{equation*}

Then we need to prove the initial data $\phi(0)=\phi_{0}$, $\rho(0)=\rho_{0}$, $\textbf{u}(0)=\textbf{u}_{0}$, $\textbf{B}(0)=\textbf{B}_{0}$. Because of  $C^{\infty}([0, T])$ is dense in $H^{1}([0, T])$, we choose the
\begin{equation*}
\zeta (t)=
\left \{
\begin{aligned}
&1-\frac{t}{j}, \quad 0\leq t\leq j,\\
&  0, \quad j\leq t\leq T,
\end{aligned}
 \right.
\end{equation*}
in equations  (\ref{3u}), (\ref{Bb-1}) and (\ref{eqn111}). When  $h, \Delta t\rightarrow 0$, we get
\begin{subequations}
\begin{align}
&-(\rho_{0}\textbf{u}_{0}, \textbf{v })+\frac{1}{j}\int_{0}^{j}(\rho(t)\textbf{u}(t), \textbf{v})\mbox{d}t=\\
&\int_{0}^{j}\{-(\nabla\cdot (\rho   \textbf{u}  \otimes\textbf{u}), \textbf{v} )+(\nabla\cdot(\textbf{u}  \otimes (\frac{\rho_{2}-\rho_{1}}{2} M(\phi )\nabla\omega )), \textbf{v })  -2 (\eta (\phi )D(\textbf{u} ), D( \textbf{v} )   )\\
&\quad +(\nabla\cdot \textbf{v} ,  p ) -\frac{1}{\mu}(\nabla\times \textbf{B} , \textbf{v} \times \textbf{B} )-(\phi   \cdot \nabla\omega , \textbf{v} )\}\zeta(t)\mbox{d}t, \\
& -(\textbf{B}_{0}, \textbf{C})+\frac{1}{j}\int_{0}^{j}(\textbf{B}(t), \textbf{C} )\mbox{d}t=\int_{0}^{j}\{
 -\frac{1}{\mu}(  \frac{1}{\sigma (\phi )} \nabla\times \textbf{B} , \nabla\times \textbf{C} )- \frac{1}{\mu }  (\frac{1}{\sigma (\phi )} \nabla\cdot\textbf{ B} , \nabla\cdot \textbf{C} ) \\
& \quad +(\textbf{u} \times \textbf{B} , \nabla\times \textbf{C} )\}\zeta(t)\mbox{d}t, \\
&-(\phi_{0}, \psi  )+\frac{1}{j}\int_{0}^{j}(\phi (t), \psi )\mbox{d}t= \int_{0}^{j}\{(\phi \textbf{u} , \nabla \psi )-(M(\phi )\nabla\omega , \nabla\psi )\}\zeta(t)\mbox{d}t,
\end{align}
\end{subequations}
when $j\rightarrow 0$ in above equations, and see for instance \cite{1987Compact}, we obtain
\begin{equation*}
\begin{aligned}
&(\rho_{0}\textbf{u}_{0}, \textbf{v})=(\rho(0)\textbf{u}(0), \textbf{v}), \quad \forall \textbf{v}\in \textbf{H}^{1}_{0},\\
&(\textbf{B}_{0}, \textbf{C})=(\textbf{B}(0), \textbf{C}), \quad \forall C\in \textbf{H}^{1}_{n},\\
&(\phi_{0}, \psi)=(\phi(0), \psi), \quad \forall \psi \in H^{1}.
\end{aligned}
\end{equation*}

We know  that $H^{1}$ is dense in $L^{2}$, $\textbf{H}^{1}_{n}$ be dense in $\textbf{L}^{2}$, therefor it hold $\phi_{0}=\phi(0)$ and $\textbf{B}_{0}=\textbf{B}(0)$. And   $\phi_{0}=\phi(0)$ implies the  $\rho_{0}=\rho(0)$, thus the $\textbf{u}_{0}=\textbf{u}(0)$ is valid by the $\textbf{H}^{1}_{0}$ densely in $\textbf{L}^{2}$. From equation (\ref{45}),  We have
\begin{equation}
\begin{aligned}
&\frac{\rho_{h\Delta t}(t^{m+1})}{2}\|\textbf{u}_{h\Delta t}(t^{m+1})\|^{2}+\frac{1}{2\mu}\|\textbf{B}_{h\Delta t}(t^{m+1})\|^{2}+\frac{\gamma\varepsilon}{2}\|\nabla \phi_{h\Delta t}(t^{m+1})\|^{2}+\frac{\gamma\varepsilon}{4}\|(\phi_{h\Delta t}(t^{m+1}))^{2}-1\|^{2}\\
&< E(\rho_{h}^{0}, \textbf{u}_{h}^{0}, \textbf{B}_{h}^{0}, \phi_{h}^{0}).
\end{aligned}
\end{equation}

 According to lower semi-continuity of norms and (\ref{equation4}), an cluster point ($\phi$, $\rho$, $\omega$, $\textbf{u}$, $p$, $\textbf{B}$) satisfies the energy
inequality of model (\ref{model-energy}). From above analysis, the cluster point ($\phi$,  $\omega$, $\textbf{u}$, $p$, $\textbf{B}$) be a weak solution of the model (\ref{TWO-PHASE MHD})-(\ref{2-boundary}).
\end{proof}
\section{Numerical  examples}
We consider some numerical simulations to verify the presented numerical scheme (\ref{fully discrete scheme}). It should be noted that the constructed numerical algorithm is nonlinear, and we adopt Newton iterative to linearization. The implementation process can refer to \cite{DANAILA2014826, 2009Convergence, 2020ARA}. The specific  linearization of scheme (\ref{fully discrete scheme}) is
\begin{equation*}
\begin{aligned}
&\frac{\rho_{2}-\rho_{1}}{2\Delta t}(  \phi_{h}^{k+1,n+1} \textbf{u}_{h}^{k+1, n}-\phi_{h}^{k+1,n} \textbf{u}_{h}^{k+1, n},\textbf{v}_{h}) +\frac{1}{\Delta t}(\rho_{h}^{k+1,n} \textbf{u}_{h}^{k+1, n+1}-\rho_{h}^{k} \textbf{u}_{h}^{k},  \textbf{v}_{h}) +2 (\eta (\phi_{h}^{k})D(\textbf{u}_{h}^{k+1}), D( \textbf{v}_{h})   )\\
&+( (\rho_{h}^{k}  \textbf{u}_{h}^{k+1,n}\cdot\nabla)\textbf{u}_{h}^{k+1, n+1}, \textbf{v}_{h})+( (\rho_{h}^{k}  \textbf{u}_{h}^{k+1,n+1}\cdot\nabla)\textbf{u}_{h}^{k+1, n}, \textbf{v}_{h})-( (\rho_{h}^{k}  \textbf{u}_{h}^{k+1,n}\cdot\nabla)\textbf{u}_{h}^{k+1, n}, \textbf{v}_{h})\\
&-\frac{\rho_{2}-\rho_{1}}{2}( ( M(\phi_{h}^{k})\nabla\omega_{h}^{k+1,n} \cdot\nabla) \textbf{u}_{h}^{k+1,n+1}+ (M(\phi_{h}^{k})\nabla\omega_{h}^{k+1,n+1} \cdot\nabla) \textbf{u}_{h}^{k+1,n}, \textbf{v}_{h})  -(\nabla\cdot \textbf{v}_{h},  p_{h}^{k+1, n+1}) \\
& +( (\frac{\rho_{2}-\rho_{1}}{2} M(\phi_{h}^{k})\nabla\omega_{h}^{k+1,n} \cdot\nabla) \textbf{u}_{h}^{k+1,n}, \textbf{v}_{h})  +\frac{1}{\mu}(\nabla\times \textbf{B}_{h}^{k+1, n+1}, \textbf{v}_{h}\times \textbf{B}_{h}^{k})+(\phi_{h}^{k}  \cdot \nabla\omega_{h}^{k+1, n+1}, \textbf{v}_{h})=0, \\
&(\nabla\cdot\textbf{ u}_{h}^{k+1, n+1}, q_{h}) =0,\\
& ( \frac{\textbf{B}_{h}^{k+1, n+1}-\textbf{B}_{h}^{k}}{\Delta t}, \textbf{C}_{h}) + \frac{1}{\mu}(  \frac{1}{\sigma (\phi_{h}^{k})} \nabla\times \textbf{B}_{h}^{k+1, n+1}, \nabla\times \textbf{C}_{h})+ \frac{1}{\mu}  (\frac{1}{\sigma (\phi_{h}^{k})}\nabla\cdot\textbf{ B}_{h}^{k+1, n+1}, \nabla\cdot \textbf{C}_{h})  \\
& \quad  -(\textbf{u}_{h}^{k+1, n+1}\times \textbf{B}_{h}^{k}, \nabla\times \textbf{C}_{h})=0, \\
&(\frac{\phi_{h}^{k+1, n+1}-\phi_{h}^{k}}{\Delta t}, \psi_{h}) -(\phi_{h}^{k} \textbf{u}_{h}^{k+1, n+1}, \nabla \psi_{h})+ (M(\phi_{h}^{k})\nabla\omega_{h}^{k+1, n+1}, \nabla\psi_{h})=0,  \\
&(\omega_{h}^{k+1,  n+1}, \chi_{h})=\gamma\varepsilon(\nabla\phi_{h}^{k+1,  n+1}, \nabla\chi_{h})+ \frac{\gamma}{\varepsilon}( 3(\phi_{h}^{k+1, n})^{2}(\phi_{h}^{k+1, n+1}-\phi_{h}^{k+1, n}   )+ (\phi_{h}^{k+1, n})^{3}-\phi_{h}^{k}, \chi_{h}),  \\
&\textbf{u}_{h}^{0}= P_{uh}\textbf{u}_{0}, \, \,  \textbf{B}_{h}^{0}=P_{B h}\textbf{B}_{0},  \, \,\phi_{h}^{0}=P_{\phi h}\phi_{0},
\end{aligned}
\end{equation*}
where  $k$ indicate the time loop, the $n$ stand for Newton loop. The all numerical examples are based on above  Newton iterative algorithm. The velocity filed and pressure filed are discreted by Mini finite element pair ($\textbf{P}_{1}^{b}, P_{1}$), and the other variables  are discreted by  $\textbf{P}_{1}$ finite element. For simplicity, we take $\Delta t$=$O(h^{2})$ to verify both the temporal and spatial convergence orders. According to finite element theory, the expected convergence orders are as follows

\begin{equation}\label{6GU-a}
\begin{aligned}
&\|\phi(t^{n})-\phi_{h}^{n}\|\lesssim \Delta t+h^{2}\lesssim h^{2}, \,\|\nabla\phi(t^{n})-\nabla\phi_{h}^{n}\|\lesssim \Delta t+h\lesssim h,\\
& \|\textbf{u}(t^{n})-\textbf{u}_{h}^{n}\|\lesssim \Delta t+h^{2}\lesssim h^{2}, \, \|\nabla\textbf{u}(t^{n})-\nabla\textbf{u}_{h}^{n}\|+\|p(t^{n})-p_{h}^{n}\|\lesssim \Delta t+h\lesssim h,\\
& \|\textbf{B}(t^{n})-\textbf{B}_{h}^{n}\|\lesssim \Delta t+h^{2}\lesssim h^{2}, \, \|\nabla\textbf{B}(t^{n})-\nabla\textbf{B}_{h}^{n}\| \lesssim \Delta t+h\lesssim h.
\end{aligned}
\end{equation}

\subsection{Error test}
We mainly test smooth solution in solving region $\Omega$=$[0, 1]^{2}$. Given   parameters: $M_{1}$=$M_{2}$=1, $\eta_{1}$=$\eta_{2}$=1, $\mu$=1, $\sigma_{1}$=$\sigma_{2}$=1, $\varepsilon$=1, $\gamma$=1.
Choosing the following smooth solutions as
\begin{eqnarray*}
\left\{
\begin{array}{ll}
\phi= (\rm cos(t)(cos(\pi x))^2(cos(\pi y))^2,\\
\textbf{u}=\rm cos(t)(\pi sin(2\pi y) (sin(\pi x))^2 ,\ -\rm \pi sin(2\pi x) (sin(\pi y))^2 ),\\
p=\rm cos(t)(2x-2) (2y -1),\\
\textbf{B}=\rm \cos(t)(sin(\pi x)cos(\pi y),  -\rm sin(\pi y) cos(\pi x)). \\
\end{array}
\right.
\end{eqnarray*}

For simplicity, we take $\Delta t$=$O(h^{2})$ to verify both the temporal and spatial convergence orders simultaneously in  Table 1 (matched density $\rho_{2}$=$\rho_{1}$=1) and  Table 2 (the large density ratios $\frac{\rho_{2}}{\rho_{1}}$=$10^{-3}$) refer to \cite{2007DiffuseJH}.  We can observe that almost all the unknowns for the  two choosing show a good performance from the order of convergence point of view.
Both $\|\phi(t^{n})-\phi_{h}^{n}\|$, $\|\textbf{u}(t^{n})-\textbf{u}_{h}^{n}\|$ and  $\|\textbf{B}(t^{n})-\textbf{B}_{h}^{n}\|$ are expected to be $O(h^{2})$. And   $\|\nabla (\phi(t^{n})-\phi_{h}^{n})\|$, $\|\nabla (\textbf{u}(t^{n})-\textbf{u}_{h}^{n})\|$, $\|\nabla (\textbf{B}(t^{n})-\textbf{B}_{h}^{n})\|$ are expected to be $O(h)$. Noting that $\|p(t^{n})-p_{h}^{n}\|$ is $O(h^{\frac{3}{2}})$, the superconvergence  maybe influenced by the Mini finite element pair ($\textbf{P}_{1}^{b}, P_{1}$).

\begin{table}[hpt]\label{1-table1}
\tabcolsep 1.2mm {\footnotesize\textbf{Table 1:} Convergence results for numerical scheme (\ref{fully discrete scheme}) with matched densities ($\rho_{1}$=$\rho_{2}$=1).}
\begin{center}
\begin{tabular}{ccccccccccccccccccccccccccccccc}
\hline
h & $\|\phi(t^{n})-\phi_{h}^{n}\|$ & rate & $\|\nabla (\phi(t^{n})-\phi_{h}^{n})\|$ & rate & $\|\textbf{u}(t^{n})-\textbf{u}_{h}^{n}\|$ & rate & $\|\nabla (\textbf{u}(t^{n})-\textbf{u}_{h}^{n})\|$ & rate &\\ \hline
 1/8 &0.2662 & 1.48  & 0.3622  &  1.20   &0.1305 & 1.66 & 0.3022 & 0.92  & \\
 1/16 &0.0748 & 1.83 & 0.1565 & 1.21  & 0.0344& 1.92   & 0.1518  & 0.99  & \\
 1/32 & 0.0193  &1.95  & 0.0736  & 1.09  & 0.0087 & 1.99 &0.0758 & 1.00  &\\
 1/64 & 0.0049  & 1.99 & 0.0361   & 1.03  & 0.0022  & 2.00  &0.0378   & 1.00  &\\
\hline
h &$\|\textbf{B}(t^{n})-\textbf{B}_{h}^{n}\|$ & rate & $\|\nabla (\textbf{B}(t^{n})-\textbf{B}_{h}^{n})\|$ & rate & $\|p(t^{n})-p_{h}^{n}\|$ & rate & \\ \hline
  1/8 & 0.0152 & 1.58   &  0.1953 &  0.97  & 3.2777 &1.45  & \\
  1/16 & 0.0041    & 1.89   & 0.0980    & 0.99   & 1.0423 & 1.65  & \\
  1/32 & 0.0010  & 1.97   &  0.0491 &  1.00 & 0.3355  & 1.64 &\\
  1/64 & 0.0003   &2.00 & 0.0245  & 1.00  & 0.1126   & 1.57  & \\
\hline
\end{tabular}
\end{center}
\end{table}

\begin{table}[hpt]\label{1-table2}
\tabcolsep 1.2mm {\footnotesize\textbf{Table 2:} Convergence results for numerical scheme (\ref{fully discrete scheme}) with large density ratios ($\frac{\rho_{2}}{\rho_{1}}$=$10^{-3}$).}
\begin{center}
\begin{tabular}{ccccccccccccccccccccccccccccccc}
\hline
h & $\|\phi(t^{n})-\phi_{h}^{n}\|$ & rate & $\|\nabla (\phi(t^{n})-\phi_{h}^{n})\|$ & rate & $\|\textbf{u}(t^{n})-\textbf{u}_{h}^{n}\|$ & rate & $\|\nabla (\textbf{u}(t^{n})-\textbf{u}_{h}^{n})\|$ & rate &\\ \hline
 1/8 & 0.2662 & 1.48   &  0.3622   &  1.20   &0.1282  & 1.67  &  0.3019 &   0.92  & \\
 1/16 & 0.0748 & 1.83    & 0.1565  &  1.21   & 0.03370   & 1.93   & 0.1517     &   0.99  & \\
 1/32 &  0.0193 & 1.95  & 0.0736  &  1.09 & 0.0085 &  1.99   & 0.0758 &  1.00  &\\
 1/64 & 0.0049  & 1.99   &  0.0361    &  1.03  & 0.0021    & 2.00   & 0.0378   &  1.00   &\\
\hline
h &$\|\textbf{B}(t^{n})-\textbf{B}_{h}^{n}\|$ & rate & $\|\nabla (\textbf{B}(t^{n})-\textbf{B}_{h}^{n})\|$ & rate & $\|p(t^{n})-p_{h}^{n}\|$ & rate & \\ \hline
  1/8 &  0.0436  &  1.80    &  0.1953  &  0.97    &  3.3597  &  1.45  & \\
  1/16 &  0.0113    &  1.94  &  0.0980   &   0.99    & 1.0634  & 1.66   & \\
  1/32 & 0.0029   &  1.99  &  0.0491 &  1.00  &  0.3399   &  1.65 &\\
  1/64 &   0.0007  & 2.00 & 0.0245  &   1.00   & 0.1135  &  1.58   & \\
\hline
\end{tabular}
\end{center}
\end{table}
\subsection{Spinodal decomposition}
In this section, we test the algorithm energy, mass conservation of phase field and spinodal decomposition reder ro \cite{2023Energy}. Choosing the domain  $\Omega$=$[0, 1]^{2}$. The initial datas of the variable fields  can be given as
\begin{equation}\label{61}
\begin{aligned}
&\textbf{u}^{0}=0, \quad p^{0}=0,\quad  \textbf{B}^{0}=0,\quad \phi^{0}=\psi_{0}+0.001 \rm rand(r),
\end{aligned}
\end{equation}
for $\psi_{0}$=-0.05, and $\rm rand (r)$ is  a uniformly distributed random function in $[-1, 1]$ with zero mean. And the  homogeneous Dirichlet boundary conditions  for the velocity and magnetic fields,  homogeneous Neumann boundary conditions for the phase field and chemical potential. To obtain the algorithm energy and mass conservation of phase field, we give the following parameters as
\begin{equation}
\begin{aligned}
\eta_{1}=\eta_{2}=\mu=\sigma_{1}=\sigma_{2}=M_{1}=M_{2}=1,\, \varepsilon=\gamma=0.01,\, \frac{\rho_{2}}{\rho_{1}}=10^{-3}.
\end{aligned}
\end{equation}
To test the energy and mass, we set the different time step size $\Delta t$=1, 0.1, 0.01, 0.001, and space step size $h$=1/64. The results shown in Figure 1, where the algorithm energy all dissipate at different time step size shown in (a), and mass of phase field all conserve shown in (b).

In addition, we test the spinodal decomposition with the initial values in equation (\ref{61}), time step size $\Delta t$=0.0001, space step size $h$=1/150,  and the following two kinds of parameters, which the results are shown in Figures 2-3.
\begin{subequations}
\begin{align}
&\varepsilon=\gamma=0.01,\, \frac{\rho_{2}}{\rho_{1}}=10^{-3},\, \eta_{1}=\eta_{2}=\mu=\sigma_{1}=\sigma_{2}=M_{1}=M_{2}=0.001,\\
&\varepsilon=\gamma=0.01,\, \frac{\rho_{2}}{\rho_{1}}=10^{-3},\, \eta_{1}=\eta_{2}=\mu=\sigma_{1}=\sigma_{2}=1,\, M_{1}=M_{2}=1.
\end{align}
\end{subequations}
\begin{figure}
\begin{centering}
\begin{tabular}{ccc}
\includegraphics[scale=0.4 ]{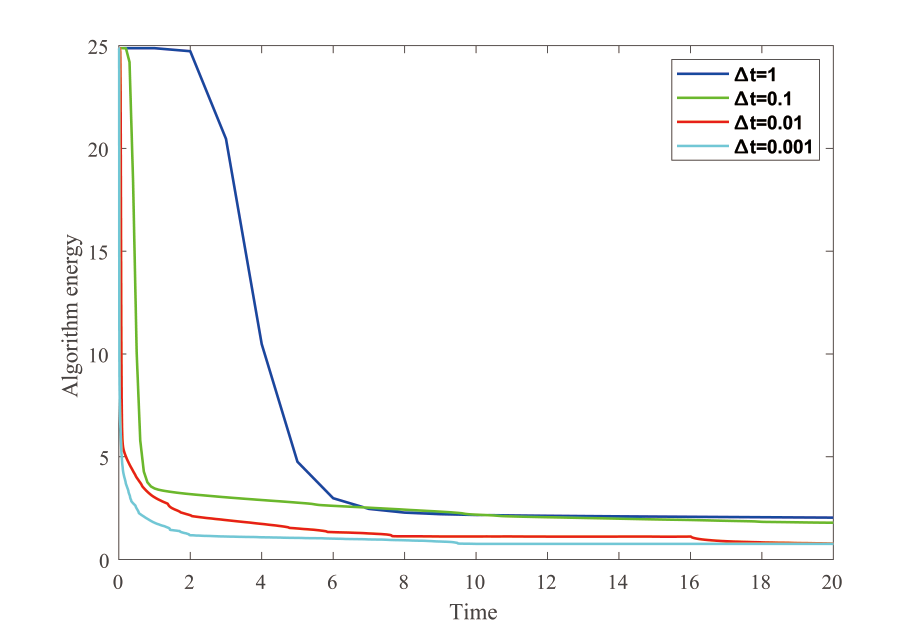}
& \includegraphics[scale=0.4 ]{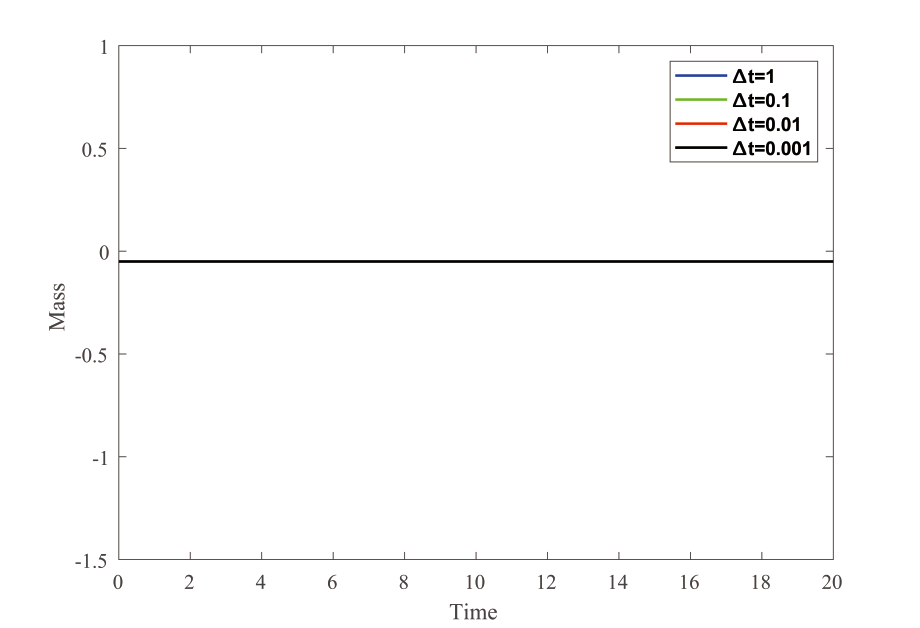} \\
\end{tabular}
\par\end{centering}
\caption{Algorithm energy (a) and  mass of phase field (b).\label{fig:Lid3D}}
\end{figure}
\begin{figure}
\begin{centering}
\begin{tabular}{cccc}
\includegraphics[scale=0.24, trim=50 50 50 50,clip]{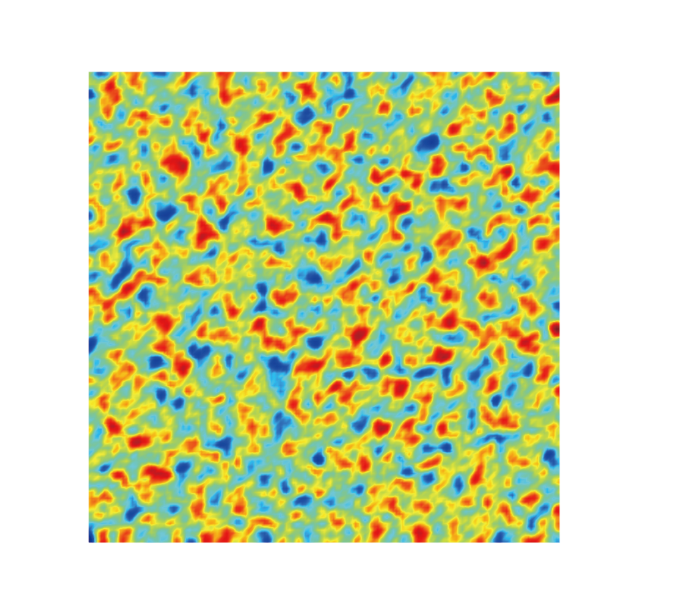}
&\includegraphics[scale=0.24, trim=50 50 50 50,clip]{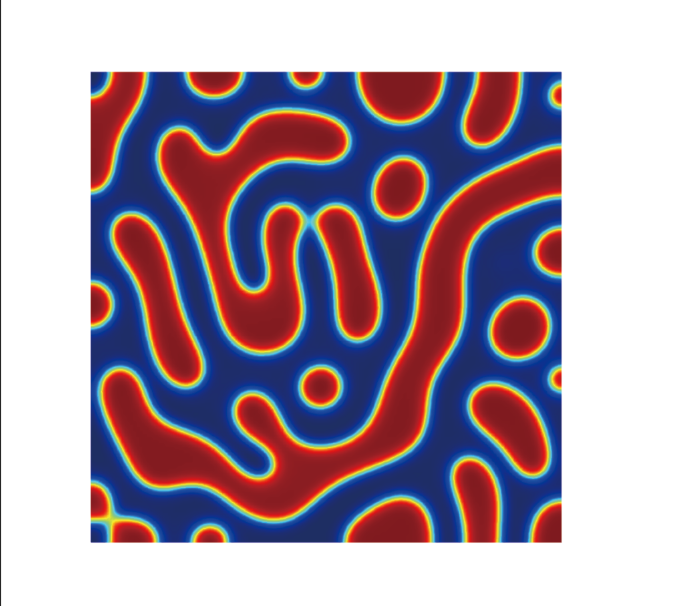}
&\includegraphics[scale=0.24, trim=50 50 50 50,clip]{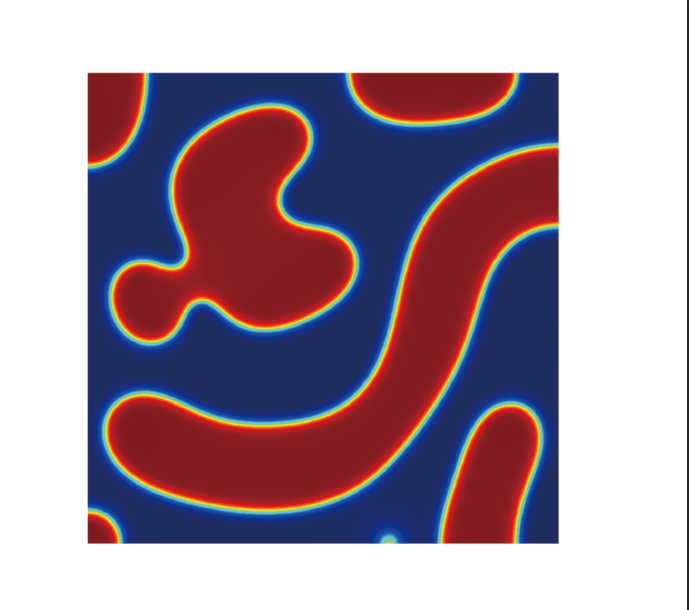}
&\includegraphics[scale=0.24, trim=50 50 50 50,clip]{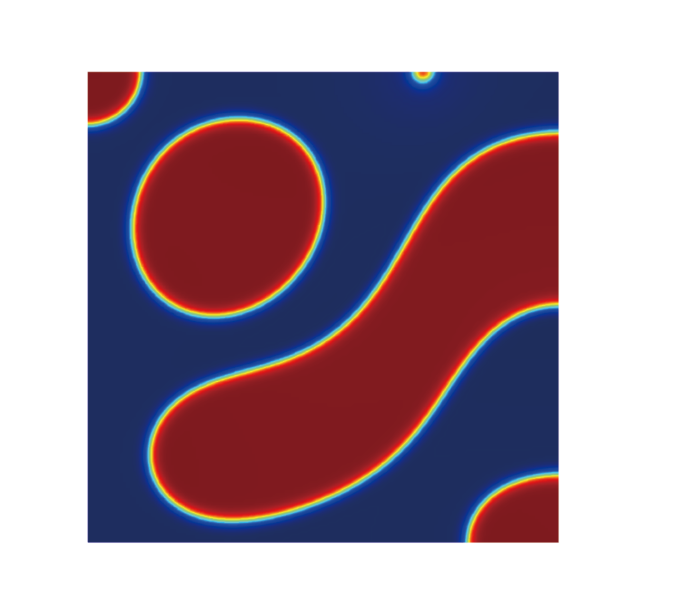}\\
\end{tabular}
\par\end{centering}
\caption{The dynamics of spinodal decomposition examples for scheme at t=0.0001 (a), 0.05 (b), 0.2 (c), 1 (d) with M=0.001.}
\end{figure}
\begin{figure}
\begin{centering}
\begin{tabular}{cccc}
\includegraphics[scale=0.24, trim=50 50 50 50,clip]{00001}
&\includegraphics[scale=0.24, trim=50 50 50 50,clip]{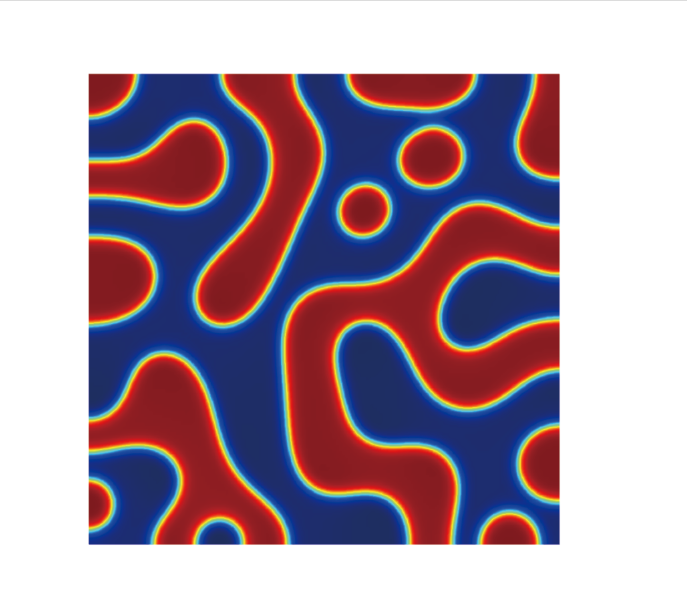}
&\includegraphics[scale=0.24, trim=50 50 50 50,clip]{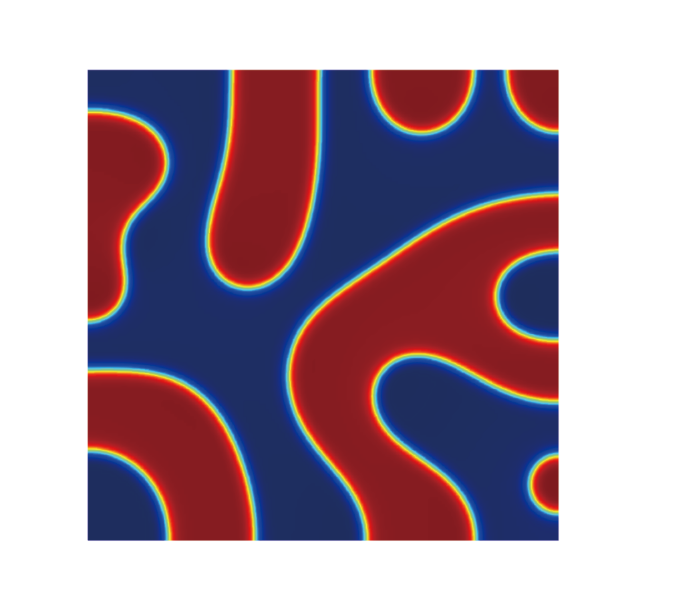}
&\includegraphics[scale=0.24, trim=50 50 50 50,clip]{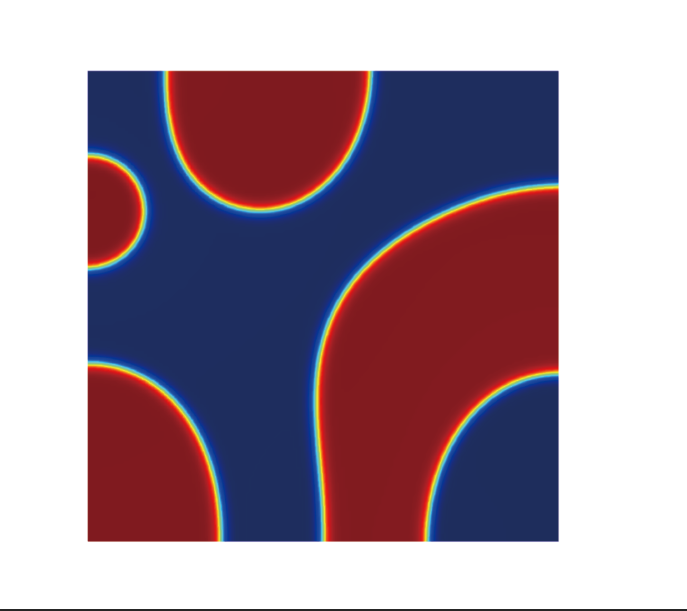}\\
\end{tabular}
\par\end{centering}
\caption{The dynamics of spinodal decomposition examples for scheme at t=0.0001 (a), 0.05 (b), 0.2 (c), 1 (d) with M=1.}
\end{figure}
When $M_{1}$=$M_{2}$=0.001, 1,  the numerical results are shown in Figure 2 and Figure 3 respectively. From Figures 2-3, we can observe the phenomena of spinodal decomposition are well behaved at different time.
\subsection{Rising bubble}
In this subsection, we test the well behavior of the model through rising bubble phenomenon with different density ratios and different radius. The gravity is introduced in the system by adding an external force in the momentum equation, that is, by adding in the right hand side of the first equation in (\ref{2-m}) a term of the form $\textbf{F}=\rho\textbf{g}$. specifically,
\begin{equation*}
\begin{aligned}
&\rho\textbf{u}_{t} +((\rho\textbf{u}-\frac{\partial\rho(\phi)}{\partial\phi} M(\phi)\nabla\omega)\cdot\nabla)\textbf{u} -2\eta (\phi)\nabla\cdot D(\textbf{u}) +\nabla p -\frac{1}{\mu}\nabla\times \textbf{B}\times \textbf{B}+\lambda \phi  \cdot \nabla\omega=\textbf{F}.
\end{aligned}
\end{equation*}
We set $\Omega$=$[0, 1]\times[0, 1.5]$, time step size $\Delta t$=0.001, space step size $h$=1/200, and the following parameters
\begin{equation*}
M_{1}=M_{2}=10^{-4},\quad \eta_{1}=\eta_{2}=1, \quad \mu=1, \quad \lambda=5, \quad \sigma_{1}=\sigma_{2}=1, \quad  \varepsilon=0.01, \quad \gamma=1,
\end{equation*}
and equipped with following boundary conditions
\begin{equation*}
\left\{
\begin{aligned}
&\frac{\partial\phi}{\partial \textbf{n}}|_{\partial\Omega}=0, \frac{\partial w}{\partial \textbf{n}}|_{\partial\Omega}=0,\\
& \textbf{u}|_{y=0,\, 1.5}=\textbf{0}, u_{1} =0 \, \rm on\,  the\, remaining\, edges,\\
&\textbf{n}\times \textbf{B}|_{\partial\Omega}=\textbf{n}\times (0, 1)^{T}|_{\partial\Omega}.
\end{aligned}
\right.
\end{equation*}

The Figure 4 and Figure 5 represent the evolution of the phase field at different radius under the above parameters respectively. Obviously,   the bubbles rise gradually over time in Figures 4-5.  And setting the parameters $\sigma_{1}=\sigma_{2}$=1000, $\mu$=0.001, imply that adding the  Lorentz forces in Figure 6. Compared with Figure 5 and Figure 6, we know that Lorentz force would suppress the buoyancy of the rising bubble, then the bubble rises slowly. In addition, we test the different density ratios in Figure 7 ($\frac{\rho_{2}}{\rho_{1}}$=100) and Figure 8 ($\frac{\rho_{2}}{\rho_{1}}$=1000), can found that the larger the density ratios, the more severe the deformation of the bubble. Bubbles gradually change from circular to flat. Note that the results of the deformation we obtained are consistent with the physical experimental results given in \cite{2017Three} and the similar results of different model is also given in \cite{2006Bubbles, 2021Fully}.
\begin{figure}
\begin{centering}
\begin{tabular}{cccc}
\includegraphics[scale=0.24, trim=50 50 50 50,clip]{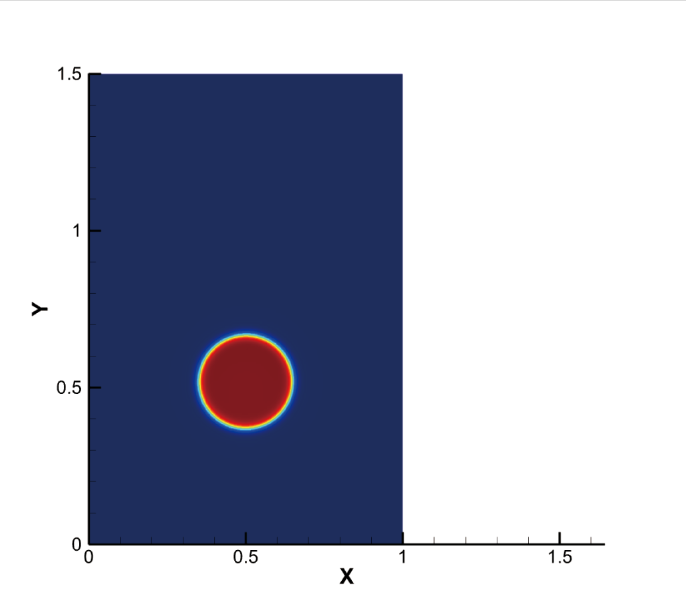}
&\includegraphics[scale=0.24, trim=50 50 50 50,clip]{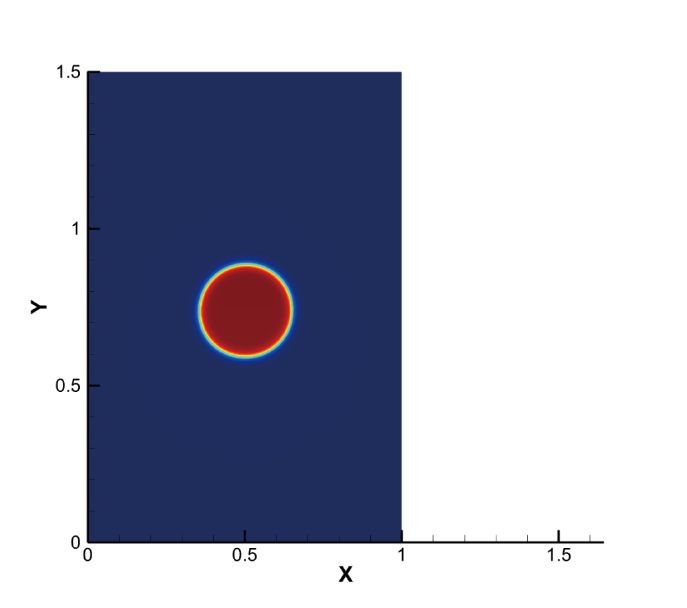}
&\includegraphics[scale=0.24, trim=50 50 50 50,clip]{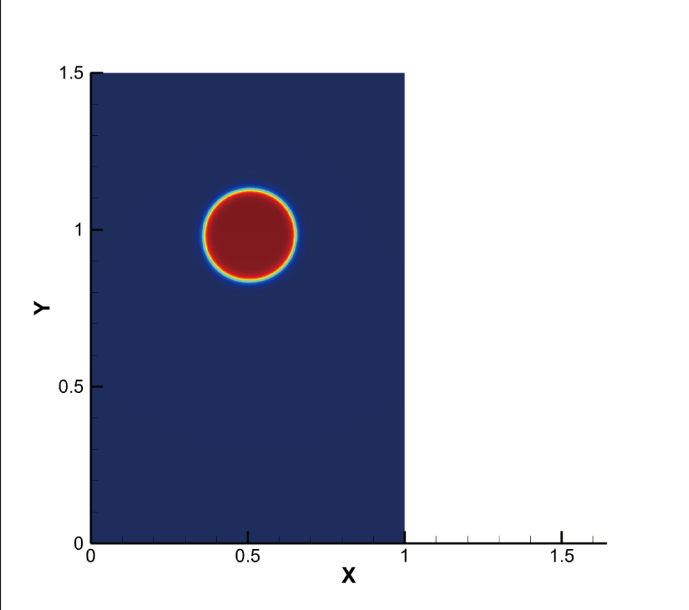}
&\includegraphics[scale=0.24, trim=50 50 50 50,clip]{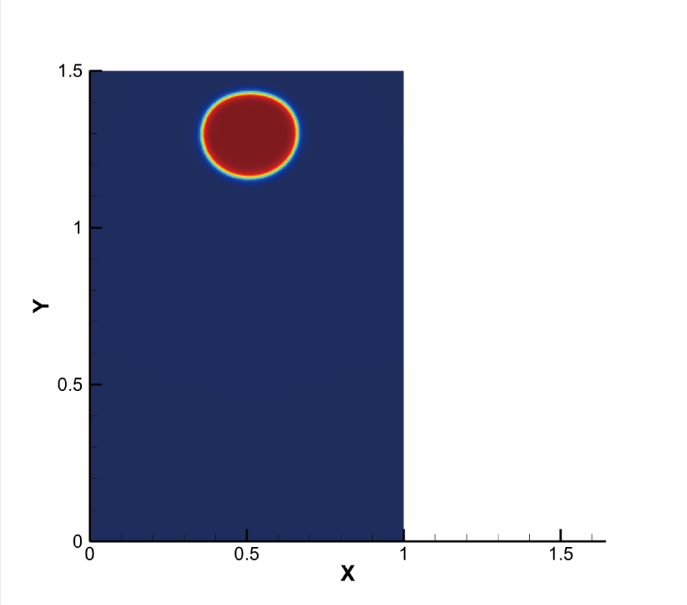}\\
\end{tabular}
\par\end{centering}
\caption{The phase evolution at t=0.1 (a), 1 (b), 2.5 (c), 3.5 (d) with R=0.15, $\frac{\rho_{2}}{\rho_{1}}$=9.}
\end{figure}
\begin{figure}
\begin{centering}
\begin{tabular}{cccc}
\includegraphics[scale=0.24, trim=50 50 50 50,clip]{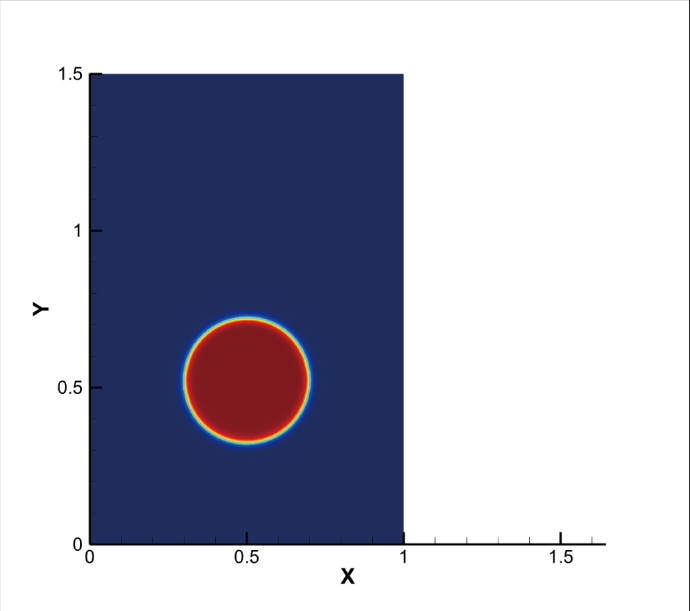}
&\includegraphics[scale=0.24, trim=50 50 50 50,clip]{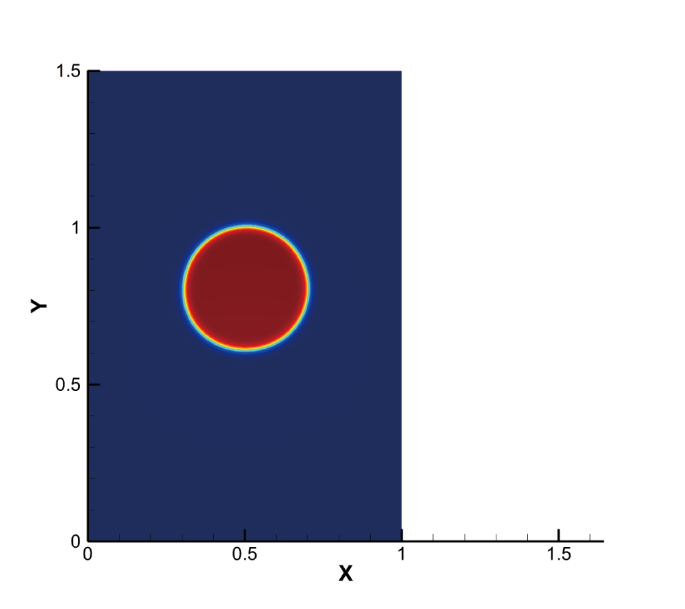}
&\includegraphics[scale=0.24, trim=50 50 50 50,clip]{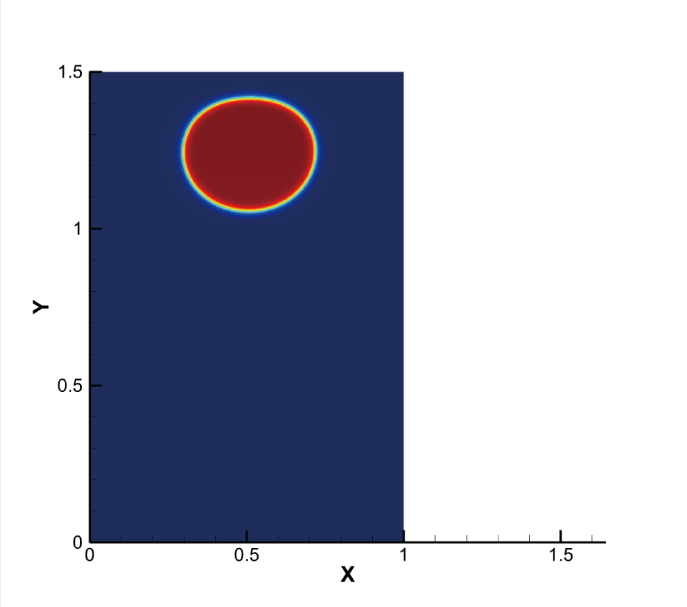}
&\includegraphics[scale=0.24, trim=50 50 50 50,clip]{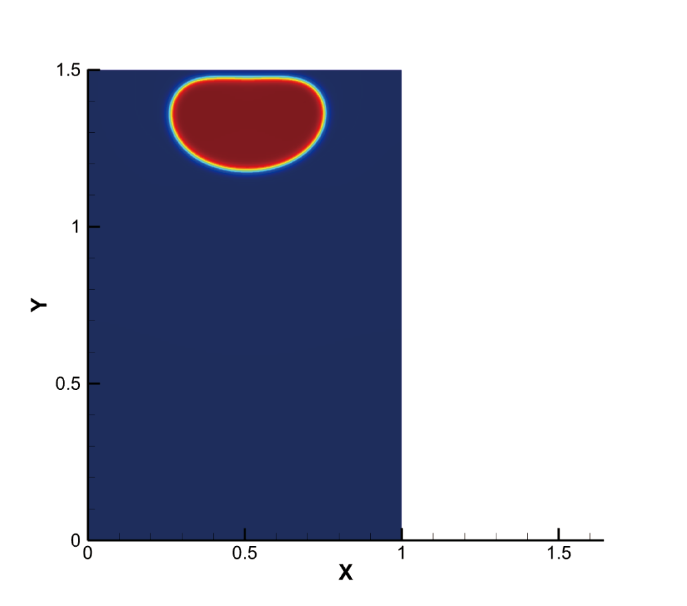}\\
\end{tabular}
\par\end{centering}
\caption{The phase evolution at t=0.1 (a), 1 (b), 2.5 (c), 3.5 (d) with R=0.2, $\frac{\rho_{2}}{\rho_{1}}$=9.}
\end{figure}
\begin{figure}
\begin{centering}
\begin{tabular}{cccc}
\includegraphics[scale=0.24, trim=50 50 50 50,clip]{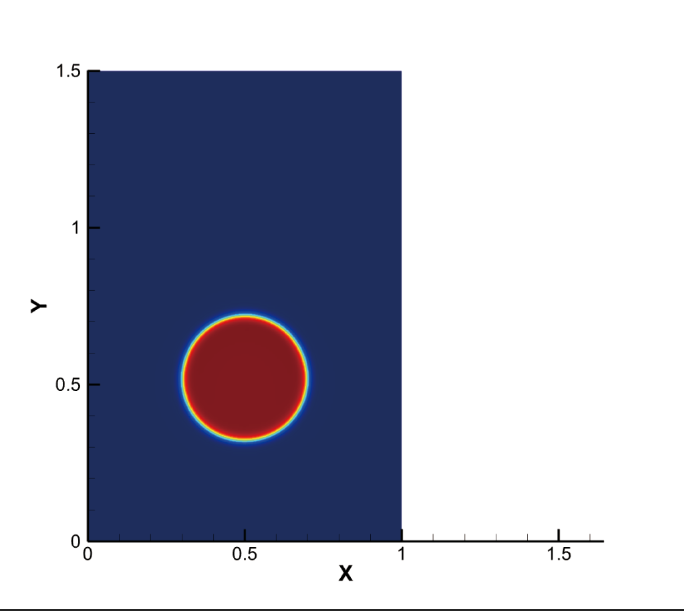}
&\includegraphics[scale=0.24, trim=50 50 50 50,clip]{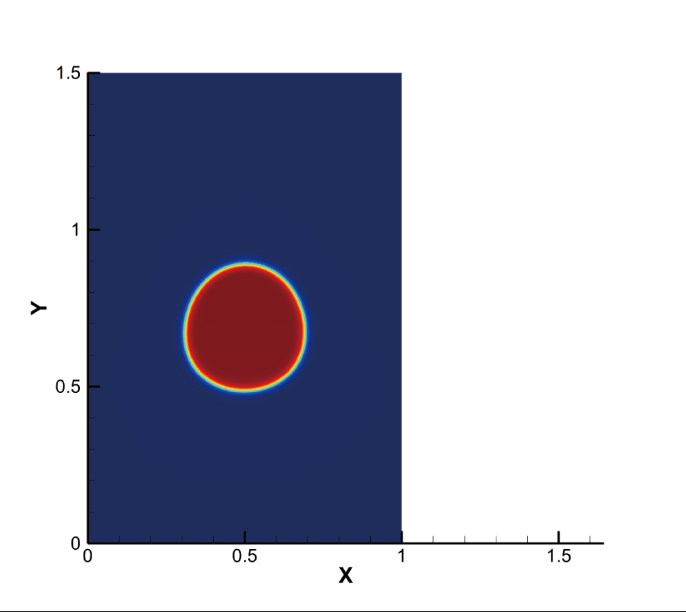}
&\includegraphics[scale=0.24, trim=50 50 50 50,clip]{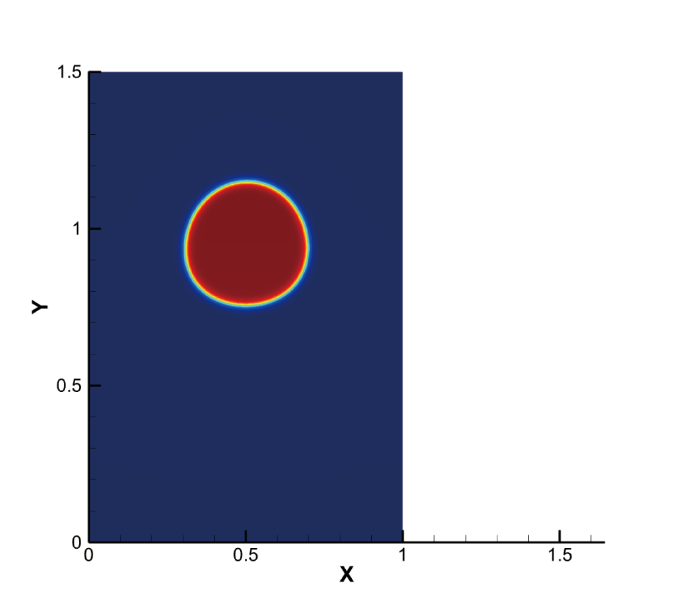}
&\includegraphics[scale=0.24, trim=50 50 50 50,clip]{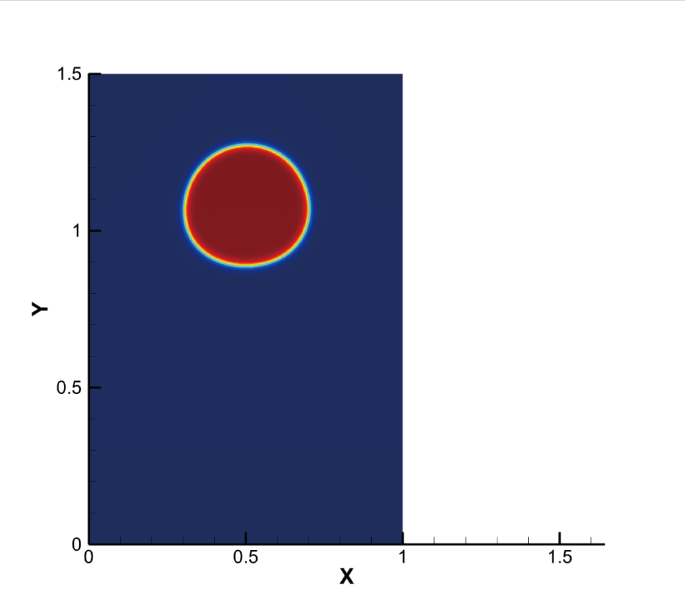}\\
\end{tabular}
\par\end{centering}
\caption{The phase evolution at t=0.1 (a), 1 (b), 2.5 (c), 3.5 (d) with R=0.2, $\frac{\rho_{2}}{\rho_{1}}$=9, $\sigma_{1}=\sigma_{2}$=1000, $\mu$=0.001.}
\end{figure}

\begin{figure}
\begin{centering}
\begin{tabular}{cccc}
\includegraphics[scale=0.24, trim=50 50 50 50,clip]{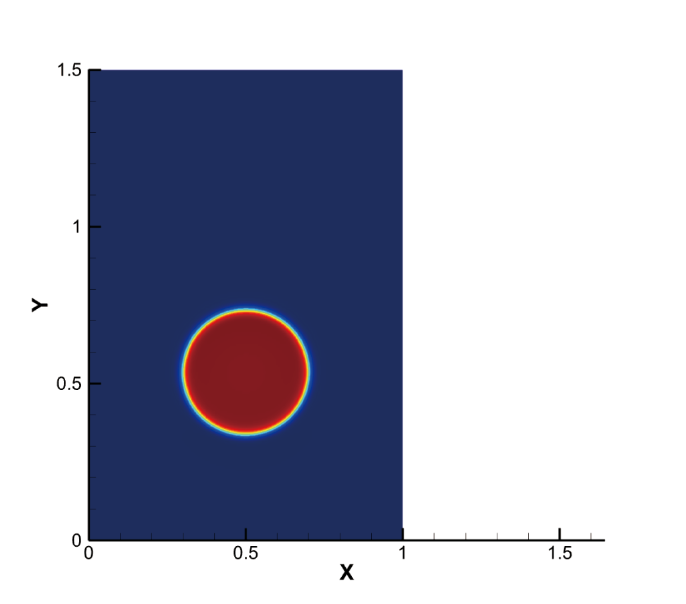}
&\includegraphics[scale=0.24, trim=50 50 50 50,clip]{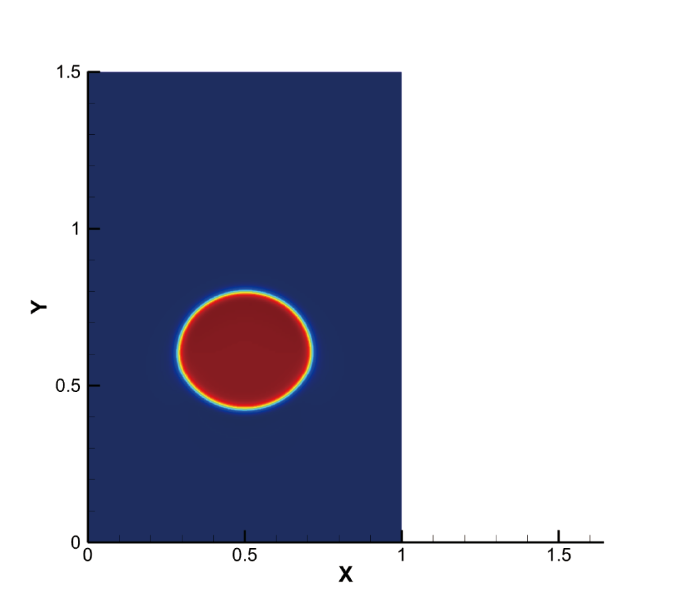}
&\includegraphics[scale=0.24, trim=50 50 50 50,clip]{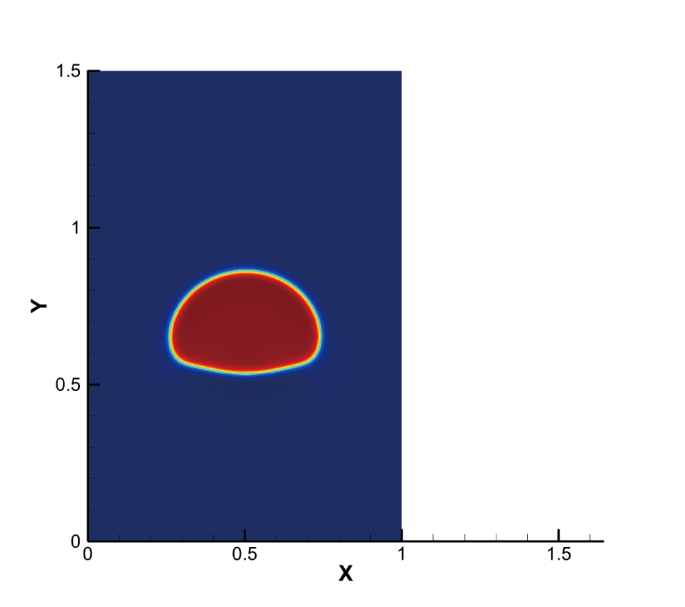}
&\includegraphics[scale=0.24, trim=50 50 50 50,clip]{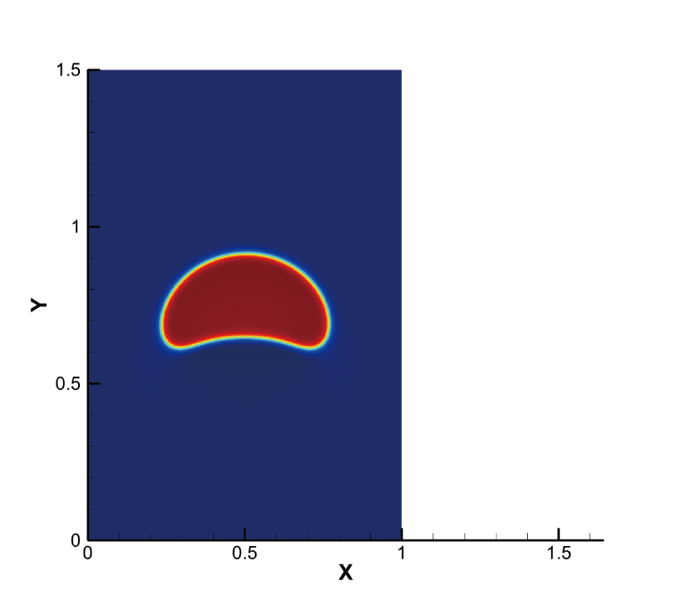}\\
\end{tabular}
\par\end{centering}
\caption{The phase evolution at t=0.1 (a), 0.2 (b), 0.3 (c), 0.4 (d) with R=0.2, $\frac{\rho_{2}}{\rho_{1}}$=100.}
\end{figure}
\begin{figure}
\begin{centering}
\begin{tabular}{cccc}
\includegraphics[scale=0.24, trim=50 50 50 50,clip]{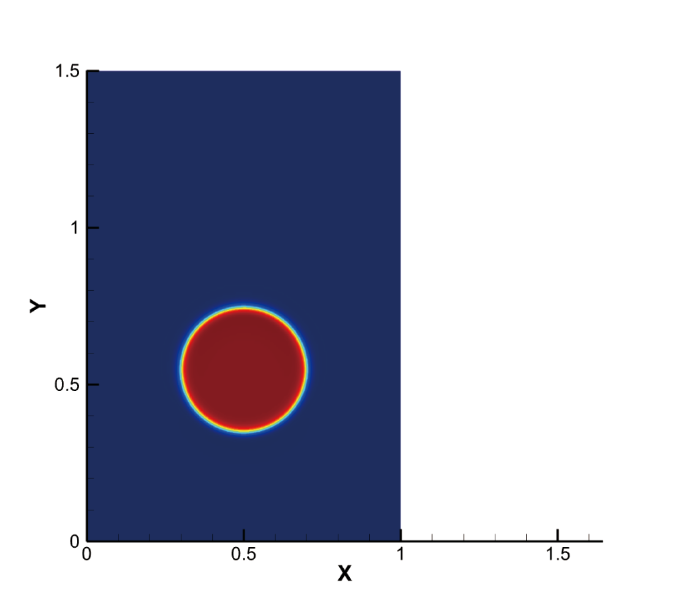}
&\includegraphics[scale=0.24, trim=50 50 50 50,clip]{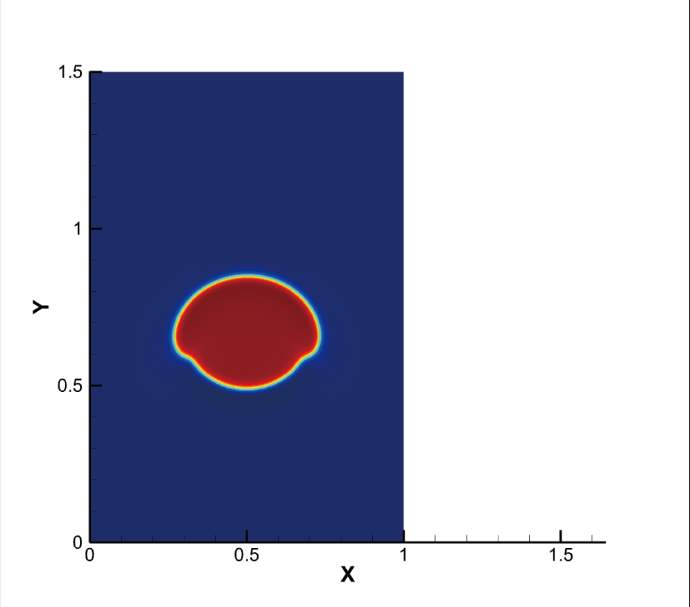}
&\includegraphics[scale=0.24, trim=50 50 50 50,clip]{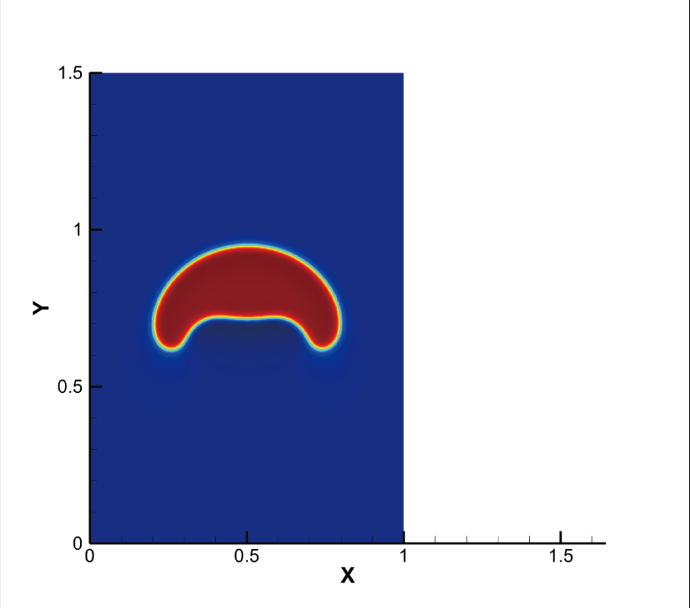}
&\includegraphics[scale=0.24, trim=50 50 50 50,clip]{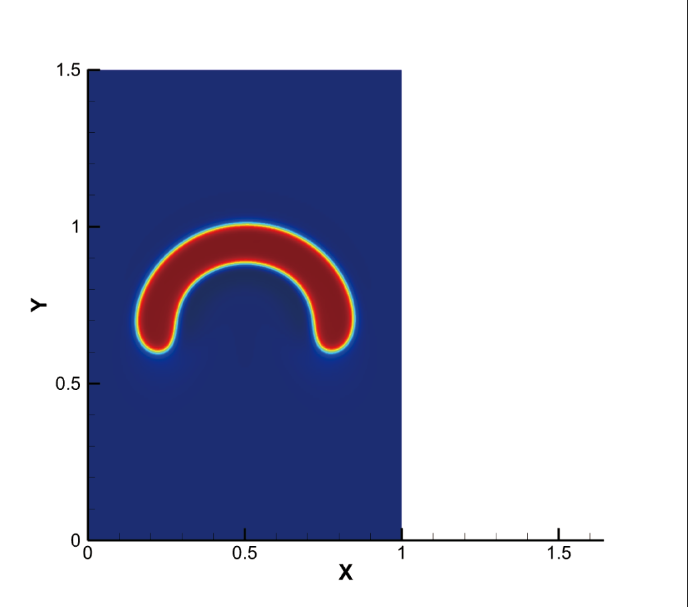}\\
\end{tabular}
\par\end{centering}
\caption{The phase evolution at t=0.1 (a), 0.2 (b), 0.3 (c), 0.4 (d) with R=0.2, $\frac{\rho_{2}}{\rho_{1}}$=1000.}
\end{figure}

\end{The}

\newpage
\bibliographystyle{Ieeetr}

\addcontentsline{toc}{section}{\refname}\bibliography{reference}

\end{document}